\documentclass[12pt,a4paper]{article}
\usepackage[utf8]{inputenc}
\usepackage[english]{babel}
\usepackage{xcolor}

\usepackage{amssymb,mathtools,amsmath,amsthm,stmaryrd}
\usepackage[all,cmtip,2cell]{xy} 

\usepackage[draft=false, hidelinks, linkcolor = blue]{hyperref}
\usepackage{cleveref}
\usepackage{tensor}
\usepackage{url}
\usepackage{bookmark}
\usepackage{cleveref}

\usepackage{graphicx}
\usepackage{lmodern}
\usepackage{enumitem}
\usepackage{array}
\usepackage[left=2.5cm,right=2.5cm,top=2.5cm,bottom=2.5cm]{geometry}

\def\defthm#1#2#3#4{
  \newtheorem{#1}[theorem]{#3}
  \newtheorem*{#1*}{#3}
  \newtheorem{#2}[theorem]{#4}
  \newtheorem*{#2*}{#4}
  \crefname{#1}{#3}{#4}
  \crefname{#2}{#4}{#4}  
}

\newtheoremstyle{mythm}% 
{10pt}% Space above 
{}% Space below 
{\itshape}% Body font 
{}% Indent amount 
{\bf}%  Theorem head font 
{.}% Punctuation after theorem head 
{.5em}% Space after theorem head 
{}% 

\newtheoremstyle{mydef}% 
{10pt}% Space above 
{3pt}% Space below 
{}% Body font 
{}% Indent amount 
{\bf}%  Theorem head font 
{.}% Punctuation after theorem head 
{.5em}% Space after theorem head 
{}% 

\newtheoremstyle{myrmk}% 
{10pt}% Space above 
{3pt}% Space below 
{}% Body font 
{}% Indent amount 
{\bf}%  Theorem head font 
{.}% Punctuation after theorem head 
{.5em}% Space after theorem head 
{}% 

\theoremstyle{mythm}
\newtheorem{theorem}{Theorem}[section]
\newtheorem*{theorem*}{Theorem}

\defthm{corollary}{corollaries}{Corollary}{Corollaries}
\defthm{lemma}{lemmata}{Lemma}{Lemmata}
\defthm{proposition}{propositions}{Proposition}{Propositions}
\defthm{axiom}{axioms}{Axiom}{Axioms}
\defthm{propdef}{props/defs}{Proposition/Definition}{Prositions/Definitions}
\defthm{defcor}{defscors}{Corollary/Definition}{Corollaries/Definitions}
\defthm{conjecture}{conjectures}{Conjecture}{Conjectures}

\theoremstyle{mydef}
\defthm{definition}{definitions}{Definition}{Definitions}

\theoremstyle{myrmk}
\defthm{acknowledgment}{acknowledgments}{Acknowledgment}{Acknowledgments}
\defthm{remark}{remarks}{Remark}{Remarks}
\defthm{motivation}{motivations}{Motivation}{Motivations}
\defthm{example}{examples}{Example}{Examples}
\defthm{question}{questions}{Question}{Questions}
\defthm{observation}{observations}{Observation}{Observations}
\defthm{claim}{claims}{Claim}{Claims}
\defthm{notation}{notations}{Notation}{Notations}

\newcommand{\sprime}{^{\prime}}
\newcommand{\pbs}{\scalebox{1.5}{\rlap{$\cdot$}$\lrcorner$}}
\newcommand{\pos}{\rotatebox[origin=c]{180}{\pbs}}
\newcommand{\adj}{\rotatebox[origin=c]{180}{$\vdash$}\hspace{0.1pc}}
\newcommand*{\phv}{\makebox[1.5ex]{\textbf{$\cdot$}}}

\author{Raffael Stenzel}

\title{$(\infty,1)$-Categorical comprehension schemes}

\begin{document}
\maketitle

\begin{abstract}
We define and study notions of comprehension in $(\infty,1)$-category theory. In essence, we do so by implementing B\'{e}nabou's 
foundations of naive category theory in a univalent meta-theory. In particular, we develop natural generalizations of smallness and relative 
definability in this context, and show for instance that the universal cartesian fibration is small. Furthermore, by building on 
Johnstone's notion of comprehension schemes for ordinary fibered categories, we characterize and relate numerous higher categorical 
properties and structures such as left exactness, local cartesian closedness, univalent morphisms and internal $(\infty,1)$-categories in 
terms of comprehension schemes.
\end{abstract}
%\tableofcontents

\section{Introduction}

\subsection*{Comprehension \`{a} la B\'{e}nabou}

Comprehension schemes arose as crucial notions in the early work on the foundations of set theory, and hence found 
expression in a large variety of foundational settings for mathematics. Particularly, they have been introduced to the 
context of categorical logic first by Lawvere and then by B\'{e}nabou in the 1970s. Since, they have been studied in 
different forms and have been applied to many examples throughout the literature of category theory.
The notion of a comprehension scheme as used in this paper is ultimately rooted in the Axiom scheme of Restricted 
Comprehension (often referred to as the Axiom scheme of Separation) which is part of the Zermelo-Fraenkel 
axiomatization of set theory. The scheme states that every definable subclass of a set is again a set. Or in other 
words, that all sets in the set theoretic universe satisfy all set theoretical comprehension schemes.

In his critique of the foundations of naive category theory \cite{benaboufibfound}, B\'{e}nabou provided an intuition to 
define the notion of a comprehension scheme not only in category theory (over the topos of 
sets), but in category theory over any other category in a syntax-free way. In this generality, comprehension schemes become properties of 
Grothendieck fibrations over arbitrary categories. Which particular comprehension schemes are satisfied by a given Grothendieck fibration
$\mathcal{E}\twoheadrightarrow\mathcal{C}$ then depends on the categorical constructions available in $\mathcal{E}$ over
$\mathcal{C}$ and in $\mathcal{C}$ itself.
We point out that B\'{e}nabou in his work 
purposefully did not even specify what exactly a ``category'' is \cite[Paragraph 0.5]{benaboufibfound}, and that, on 
the basis of his effectively meta-meta-mathematical analysis, the framework of $(\infty,1)$-category theory over the 
base $\infty$-topos of spaces playing the role of ``sets'' appears to yield a suitable notion of such a category 
theory in his sense. This will be justified implicitly by the notions recalled in Section 2 together with the constructions and results of this paper. In this reading of ``sets'' as spaces one then is to replace 
the occurrences of ``ZF'' in \cite{benaboufibfound} by a suitable univalent type theory. By virtue of univalence of 
the universes $\mathbb{U}$ in the meta-theory of spaces (\cite{klvsimp} and \cite[Paragraph 6.2]{benaboufibfound}), we argue that the theory of fibered
$(\infty,1)$-categories is in fact a model of B\'{e}nabou's theory of ``categories'' where equality of objects is well-behaved and intrinsic 
(rather than ``extra structure, which a ``category'' $p$ may or may not admit'' \cite[Paragraph 8.9.1]{benaboufibfound}). 
This is instantiated by the fact that the difference between functors and pseudo-functors and hence the difference between cartesian 
fibrations and split cartesian fibrations vanishes. We will explain this in more detail in the coming paragraphs.

A style of definition for (what we will call ``diagrammatic'') comprehension schemes which tie together the elementary examples given in the 
glossary of \cite{benaboufibfound} has been introduced by Johnstone in \cite[Section B1.3]{elephant}. Subject to a few technical adjustments, 
these comprehension schemes give rise to a very well-behaved theory in the $(\infty,1)$-categorical context which we will introduce in 
Section~\ref{seccomp}. 
For instance, just as all sets in the set theoretical universe under ZF satisfy all set theoretic comprehension schemes, we will see that all 
$(\infty,1)$-categories internal to any complete (left exact) universe $\mathcal{C}$ of discrete $(\infty,1)$-categories 
satisfy all (finite) diagrammatic comprehension schemes (see, Corollary~\ref{corsmallallcompschemes} and Corollary~\ref{corfinitecomp}). 
Among all fibered $(\infty,1)$-categories over $\mathcal{C}$ they in fact are characterized by this property, as 
internalizability (i.e.\ ``smallness'') is a set of comprehension schemes.
While B\'{e}nabou himself derived his motivation of the notion of comprehension at least in part directly from set 
theory (i.e.\ in the case of ``relative definability'' \cite[Paragraphs 6.2, 7.1, 7.4]{benaboufibfound}), we will give a motivation of 
the notion via the categorical semantics of type theory and its underlying Lawvere-style notion of comprehension 
\cite{lawverecomp} instead. This is not due to personal preference, but, first, due to the fact 
that a categorical model of type theory is virtually the same thing as a general comprehension scheme on a category 
in the role of a universe as far as this paper is concerned, and second, due to the fact that the meta-theory of 
spaces from an $(\infty,1)$-categorical point of view is currently best expressed in type theoretic terms 
\cite{klvsimp}.

Therefore, first, let us recall some of the intuition about ``categories'' laid out in \cite{benaboufibfound} which 
we will apply to the theories of both ordinary and higher (fibered) categories. Thus, we think of a ``category''
$\mathcal{C}$ as an ($(\infty,1)$-)category fibered over the ($(\infty,1)$-)category $\mathrm{Cat}$ of small
($(\infty,1)$-)categories by way of its associated fibration
$\mathrm{Diag}(\mathcal{C})\twoheadrightarrow\mathrm{Cat}$ of diagrams. This fibration is defined as the 
Grothendieck construction of the exponential
$\mathrm{Fun}(\phv,\mathcal{C})\colon\mathrm{Cat}^{op}\rightarrow\mathrm{Cat}$. When restricted to the category
$\iota\colon\mathcal{S}\hookrightarrow\mathrm{Cat}$ of discrete categories, it returns the classic
fibration-of-families construction $\mathrm{Fam}(\mathcal{C})\twoheadrightarrow\mathcal{S}$. When postcomposed with 
the underlying discrete category construction $(\cdot)^{\simeq}\colon\mathrm{Cat}\rightarrow\mathcal{S}$, it returns 
the representable presheaf of diagrams $y(\mathcal{C})\colon\mathrm{Cat}^{op}\rightarrow\mathcal{S}$. Both functors
$\mathrm{Fam}:=\mathrm{Fun}(\iota(\cdot),-)\colon\mathrm{Cat}\rightarrow\mathrm{Fun}(\mathcal{S}^{op},\mathrm{Cat})$ and
$y=\mathrm{Fun}(\phv,-)^{\simeq}\colon\mathrm{Cat}\rightarrow\mathrm{Fun}(\mathrm{Cat}^{op},\mathcal{S})$ are fully faithful and preserve 
both limits and exponentials. Indeed, for fully faithfulness of the former in the $(\infty,1)$-categorical case see 
Example~\ref{expleextqcat}; for continuity, exponentials and $(\infty,2)$-categorical aspects see \cite{rs_ext}.
In this sense, both category theory indexed over discrete categories as well as discrete category theory indexed over categories 
are faithful generalizations of category theory, where we consider $\mathcal{S}$ as a universe for discrete category theory and
$\mathrm{Cat}$ as a universe for category theory. Indeed, category theory is always a theory of categories in context (of some meta-theory). 
This can be exemplified by changing the base of the indexing we consider from $\mathcal{S}$ or $\mathrm{Cat}$ to any category $\mathcal{C}$ 
in the role of a universe for category theory. Thereby, if we think of a given category $\mathcal{C}$ as a synthetic universe of discrete 
categories, we think of fibrations over $\mathcal{C}$ -- or $\mathcal{C}$-indexed categories, respectively -- as categories in context of
$\mathcal{C}$. The indexed category $\mathrm{Fam}(\mathcal{C})\colon\mathcal{S}^{op}\rightarrow\mathrm{Cat}$ for a category $\mathcal{C}$ 
(internal to $\mathcal{S}$) is replaced by the externalization $\mathrm{Ext}(X)\colon\mathcal{C}^{op}\rightarrow\mathrm{Cat}$ of a category 
$X$ internal to $\mathcal{C}$ (\cite[Section 7.3]{jacobsttbook} and Section~\ref{secext}). The (large) indexed category of families
$\mathrm{Fam}(\mathcal{S})\simeq\mathcal{S}_{/(\cdot)}\colon\mathcal{S}^{op}\rightarrow\mathrm{Cat}$ associated to the large category
$\mathcal{S}$ itself is replaced by the canonical indexing $\mathcal{C}_{/(\cdot)}\colon\mathcal{C}^{op}\rightarrow\mathrm{Cat}$ whenever 
$\mathcal{C}$ has pullbacks. If we think of a given category $\mathcal{C}$ as a synthetic universe of categories instead, then the presheaf 
of diagrams $\mathrm{Fun}(\phv,\mathcal{C})^{\simeq}\colon\mathrm{Cat}^{op}\rightarrow\mathcal{S}$ associated to a category
$\mathcal{C}\in\mathrm{Cat}$ is replaced by the representable presheaves of the form $y(C)\colon\mathcal{C}^{op}\rightarrow\mathcal{S}$ 
for objects $C\in\mathcal{C}$. 

Furthermore, it is useful to think of the objects in the base category $\mathcal{C}$ themselves as contexts, too, as well as of a general 
indexed category $\mathcal{E}\colon\mathcal{C}^{op}\rightarrow\mathrm{Cat}$ as a similarly generalized doctrine of types in varying contexts 
in $\mathcal{C}$. 
Therefore, we recall that to make a category of contexts $\mathcal{C}$ together with a family
$\mathrm{Ty}\colon\mathcal{C}^{op}\rightarrow\mathcal{S}$ of types in context into an actual categorical model of type theory, one 
essentially only further requires a notion of context extension in $\mathcal{C}$ with respect to $\mathrm{Ty}$. Such an associated notion of 
context extension then automatically gives rise to a notion of typed terms in context as well by considering sections of context extensions 
in $\mathcal{C}$. In ordinary category theory, a notion of context extension can 
be expressed in various (but essentially equivalent) terms, e.g.\ by way of full comprehension categories 
\cite{jacobscompcats}, categories with attributes \cite{cartmell, moggicatatt} or categories with families 
\cite{dybjerintt}. One further such notion in indexed terms is 
given by Awodey's natural models \cite{awodeynatmod}; that is, a category $\mathcal{C}$ together with a representable 
natural transformation
$\mathrm{Tm}\rightarrow\mathrm{Ty}$ in $\hat{\mathcal{C}}=\mathrm{Fun}(\mathcal{C}^{op},\mathcal{S})$ as originally 
introduced by Grothendieck (see Definition~\ref{defrepnattransf}). 
Such a natural transformation equips $\mathcal{C}$ for every $C\in\mathcal{C}$ and every type $A\in \mathrm{Ty}(C)$ 
with a context $C.A\in\mathcal{C}_{/C}$ which represents the type $A$ in context $C$ in a suitable fashion. This 
particular abstract expression of a ``categorical type theory'' has been studied axiomatically in
$(\infty,1)$-category theory by Nguyen and Uemura in \cite{ngyuenuemuratt}.

\begin{definition}\label{defgencomp}
Let $\mathcal{C}$ be a category. A \emph{general comprehension scheme} on $\mathcal{C}$ is a natural transformation $f\colon X\rightarrow Y$ 
in $\hat{\mathcal{C}}$. We say that $\mathcal{C}$ has $f$-comprehension if $f$ is representable.
\end{definition}

In this paper we study various classes of instances of this particular notion of comprehension. It may be worth to point out here
however that, in contrast to the ordinary categorical situation, the structure of a general comprehension scheme over an
$(\infty,1)$-category $\mathcal{C}$ is provably richer than that of a full comprehension $(\infty,1)$-category over $\mathcal{C}$ (when 
defined accordingly, see Proposition~\ref{propaltdefgencomp} and Remark~\ref{remaltdefgencomp}).

Now, every indexed category $\mathcal{E}\colon\mathcal{C}^{op}\rightarrow\mathrm{Cat}$ naturally induces a $\mathrm{Cat}$-indexed collection 
of general comprehension schemes on $\mathcal{C}$ by way of its accordingly indexed presheaf of diagrams. We will call 
these the diagrammatic comprehension schemes associated to $\mathcal{E}$. More precisely, the $\mathcal{C}$-indexed category
$\mathcal{E}\colon\mathcal{C}^{op}\rightarrow\mathrm{Cat}$ (after a choice of an equivalent splitting in the ordinary 
case) induces the $\mathcal{C}$-indexed ``generalized category''
\begin{align*}
y\circ\mathcal{E}\colon\mathcal{C}^{op} & \rightarrow\mathrm{Fun}(\mathrm{Cat}^{op},\mathcal{S})\\
C & \mapsto \mathrm{Fun}(\phv,\mathcal{E}(C))^{\simeq},
\end{align*}
which we will consider as a functor
$\llbracket\phv,\mathcal{E}\rrbracket\colon\mathrm{Cat}^{op}\rightarrow\mathrm{Fun}(\mathcal{C}^{op},\mathcal{S})$ by Currying.

\begin{definition}\label{defdiagcompsheme}
Let $\mathcal{C}$ be a category and let $\mathcal{E}$ be a $\mathcal{C}$-indexed category. A \emph{(standard) diagrammatic 
comprehension scheme} on $\mathcal{E}$ over $\mathcal{C}$ is a functor $G\colon I\rightarrow J$ of categories together with the induced 
general comprehension scheme
\[G^{\ast}\colon\llbracket J,\mathcal{E}\rrbracket\rightarrow\llbracket I,\mathcal{E}\rrbracket\]
in $\hat{\mathcal{C}}$. We say that $\mathcal{E}$ has $G$-comprehension if $\mathcal{C}$ has $G^{\ast}$-comprehension as in 
Definition~\ref{defgencomp}.
\end{definition}

With Definition~\ref{defdiagcompsheme} in mind, we think of each presheaf
$\llbracket I,\mathcal{E}\rrbracket=\mathrm{Cat}(I,\mathcal{E}(\cdot))$ as a family of types over $\mathcal{C}$. Among the most basic such 
families will be the type family $\mathcal{E}(\cdot)^{\simeq}$ of objects in $\mathcal{E}$ and the type family
$\mathrm{Cat}(\Delta^1,\mathcal{E}(\cdot))$ of morphisms in $\mathcal{E}$. 
For every functor $G\colon I\rightarrow J$ of categories -- which we think of as a morphism of categorical shapes and which often will be an extension of such -- satisfaction of the according diagrammatic comprehension scheme requires the 
existence of a context extension operator in $\mathcal{C}$ which for a context $C\in\mathcal{C}$ and a type
$F\colon I\rightarrow\mathcal{E}(C)$ assigns an object $C.F\in\mathcal{C}_{/C}$ which represents diagram extensions of $F$ along $G$ in
$\mathcal{E}$ in a suitable way. In most of the paper we will be interested in specifically these kinds of diagrammatic comprehension 
schemes. However we will define diagrammatic comprehension schemes more generally for contravariant functors of the form 
$\mathcal{E}\colon\mathcal{C}^{op}\rightarrow K$ valued in categories $K$ potentially different from
$\mathrm{Cat}$ (Definition~\ref{defGcomp}). This additional generality ensures that diagrammatic comprehension 
schemes are capable to capture additional algebraic structures which a given $\mathcal{C}$-indexed 
category $\mathcal{E}$ may come equipped with. This for example subsumes all instances of Johnstone's ``generalized'' 
diagrammatic comprehension schemes \cite[Section B1.3]{elephant}. This additional generality 
also has been considered conceptually at least in part already by B\'{e}nabou himself in his definition of
$K$-fibrations respective a corpus $K$ \cite{benaboupetit, benaboucorpus}. Apart from a few examples however, we 
will consider general comprehension schemes which go beyond Definition~\ref{defdiagcompsheme} only in 
Section~\ref{seccompcats} to briefly discuss their general theory and address definability in this context.

The expressive power of (generalized) diagrammatic comprehension schemes encompasses important examples in ordinary fibered category theory, 
most notably necessary and sufficient conditions to characterize elementary toposes (over other elementary toposes), 
see \cite[Paragraph 11.6.(iii)]{benaboufibfound} and \cite[Theorem 11.1]{streicherfibcats}. We will see in the course 
of the paper that it does so 
in the context of fibered $(\infty,1)$-category theory as well (see the outline of main results below). In fact, we 
will see that diagrammatic comprehension schemes are a more expressive and more natural notion in $(\infty,1)$-category theory than they are in ordinary category 
theory, for the reason that meta-theoretical equalities which exceed the language of formal category theory -- and 
which yet still naturally arise in the context of ordinary category theory all the time -- will intrinsically be replaced by instances of 
equivalences between $(\infty,1)$-categories. For instance, commutativity of squares 
in $(\infty,1)$-categories is again a matter of equivalence between $(\infty,1)$-categories rather than of set-theoretical equality.
This eliminates the gymnastics with elementary fibrations in \cite{streicherfibcats}, and also 
affects the remarks on the ``strangeness'' of equality in \cite[Section 8]{benaboufibfound}. In this 
sense, the study of equality becomes a study of equivalence, which is exactly due to the fact that the meta-theory satisfies the Univalence 
Axiom. This difference to ordinary category theory is in practice exemplified by B\'{e}nabou's comments about ``equality of objects'' 
which in ordinary fibered category theory is instantiated by the generally non-existent, non-canonical and in particular non-unique choice of 
a splitting of a fibration \cite[Paragraph 9.4]{benaboufibfound}. In $(\infty,1)$-category theory, as composition on all levels itself is 
only defined up to equivalence, a splitting only can be defined in homotopy-coherent terms. However, every cartesian fibration is
split homotopy-coherently in an essentially unique way. This means that $(\infty,1)$-categories in fact do have equality (i.e.\ equivalence) 
of objects in the sense of \cite[Paragraphs 8.6 and 8.7]{benaboufibfound}. Indeed we will see that all locally small 
fibered $(\infty,1)$-categories over a left exact base satisfy all of the identity principles in \cite[Paragraph 8.7]{benaboufibfound} which 
are there said to be unnatural and hence implausible for a general ``category'' (in a non-univalent meta-theory) to satisfy 
(Corollary~\ref{coridprinciples8.7}).

The main sources for the ordinary categorical constructions and definitions we have adapted are Johnstone 
\cite{elephant}, Jacobs \cite{jacobscompcats, jacobsttbook} and Streicher \cite{streicherfibcats}. 
The results presented in this paper build to a large extent on the fundamental work on quasi-categories provided by 
Lurie in \cite{luriehtt} and, ultimately, by Joyal in \cite{joyalqcats}.
An exhaustive treatment of fibered $(\infty,1)$-categories with a scope comparable to \cite[Chapter B1]{elephant} or 
\cite{streicherfibcats} in the ordinary categorical context would naturally require $(\infty,2)$-categorical 
considerations. These however are barely touched upon in this paper because the notions studied here are 
properties of one indexed $(\infty,1)$-category at a time. The collection of all $\mathcal{C}$-indexed
$(\infty,1)$-categories is only introduced in Section~\ref{secpre} to state its equivalence to the collection of 
fibered $(\infty,1)$-categories over $\mathcal{C}$ via Lurie's unstraightening construction, and therefore it is 
defined as an $(\infty,1)$-category itself. Partial treatments of fibered  $(\infty,1)$-category theory in such 
generality can be found e.g.\ in \cite{barwicketalparahct, luriehtt, riehlverityelements}. 

\subsection*{Outline of the paper and main results}

The results of this paper can be understood as a proof that fibered $(\infty,1)$-category theory is a (univalent model of) ``category'' 
theory in the sense of B\'{e}nabou \cite{benaboufibfound}, and that it in particular subsumes all positive statements about definability 
discussed in the paper and more.

We will develop the notion of comprehension schemes in the language of indexed quasi-categories.\footnote{Possibly much to the dismay of 
B\'{e}nabou himself according to \cite[Paragraph 12]{benaboufibfound}. One may argue however that the coherence issues he lamented to arise 
in indexed category theory do not arise in the same way in $(\infty,1)$-category theory.} We do so because the basic technical 
notions underlying the definition of comprehension schemes are more straightforward here than they are for fibered quasi-categories, and the 
proofs are easier to read (following suggestions of a referee).
We will however give an equivalent formulation of comprehension entirely in the language of fibered quasi-categories in 
Section~\ref{secsubcartfib}. This latter formulation is not only more faithful to Lawvere's, B\'{e}nabou's and Johnstone's original 
frameworks, but it will also provide an environment which allows to generalize the notion to $\infty$-cosmoses of not necessarily
$(\infty,1)$-categories in the sense of \cite{riehlverityelements}, as well as to give a straightforward proof of the fact that the notions 
and results in this paper are independent of the specific choice of model for $(\infty,1)$-category theory (Section~\ref{secmodindep}).

In Section~\ref{secpre} we recall the necessary material about cartesian fibrations from Lurie's book \cite{luriehtt}. We 
introduce the basic notions relevant to define and study comprehension schemes for fibered $(\infty,1)$-categories, give 
reference to the  equivalence between indexed $(\infty,1)$-category theory and fibered $(\infty,1)$-category theory via 
the corresponding Grothendieck construction (``unstraightening''), and discuss the most essential examples of 
indexed $(\infty,1)$-categories and their fibered counterparts. We recall the notion of a representable natural transformation and prove all 
statements about such that will be used in our applications in the subsequent sections.

Section~\ref{seccomp} introduces the definition of comprehension schemes and discusses many examples like global smallness, 
local smallness and definability of equivalences between objects and parallel pairs of $n$-morphisms for different cartesian 
fibrations. 
We provide tools for mutual reduction and verification of various instances (Lemma~\ref{lemmacompclosureprops}, 
Lemma~\ref{lemmachangeofbasecomp}), and apply them to prove an interplay between the most fundamental examples
(Proposition~\ref{lemmalocsm->definv}). Some of these instances are generalizations of notions and results from
\cite[Section B1.3]{elephant}, some are entirely new. 
In Section~\ref{secsubcartfib} we give an equivalent formulation of diagrammatic comprehension schemes (for the base case of
$K=\mathrm{Cat}_{\infty}$) entirely in the language of fibered $(\infty,1)$-category theory.

\begingroup
\def\thetheorem{\ref{lemmabenabou}}
\begin{proposition}
Let $G\colon I\rightarrow J$ be a map of simplicial sets and $p\colon\mathcal{E}\twoheadrightarrow\mathcal{C}$ be a cartesian fibration.
The fibration $p$ has $G$-comprehension if and only if for every vertical diagram $X\in\llbracket I,\mathcal{E}\rrbracket$ the
$(\infty,1)$-category $G^{\ast}\downarrow X$ of vertical $J$-structures extending $X$ horizontally has a terminal object.
\end{proposition}
\addtocounter{theorem}{-1}
\endgroup

The following are some of the fundamental instances of comprehension that are introduced in Section~\ref{seccomp} when formulated for various 
fibrations via Section~\ref{secsubcartfib}.

\begin{examples}\mbox{ }
\begin{itemize}
\item A cartesian fibration $p\colon\mathcal{E}\twoheadrightarrow\mathcal{C}$ over an $(\infty,1)$-category $\mathcal{C}$ with finite 
products is globally small if and only if its core $\mathcal{E}^{\times}$ has a terminal object (Example~\ref{exmplsmallness}).
\item The canonical fibration $t\colon\mathrm{Fun}(\Delta^1,\mathcal{C})\twoheadrightarrow\mathcal{C}$ associated to a small
$(\infty,1)$-category $\mathcal{C}$ with pullbacks is locally small if and only if $\mathcal{C}$ is locally cartesian closed 
(Proposition~\ref{exmplelocsmall}).
\item The representable right fibrations $\mathcal{C}_{/C}\twoheadrightarrow\mathcal{C}$ are locally small if and only if $\mathcal{C}$ has 
equalizers. They are always globally small whenever $\mathcal{C}$ has finite products (Example~\ref{remlocsmallrep}).
\item Any $(\infty,1)$-category $\mathcal{C}$ considered as a cartesian fibration over the point is globally small if and only if its core
$\mathcal{C}^{\times}\simeq\mathcal{C}^{\simeq}$ is contractible. It is locally small if and only if its hom-spaces are contractible 
(Example~\ref{exmplecntrblty}).
\item Given a (potentially large and locally large) $(\infty,1)$-category $\mathcal{C}$, its fibration of families
$\mathrm{Fam}(\mathcal{C})\twoheadrightarrow\mathcal{S}$ is (locally) small if and only if the $(\infty,1)$-category $\mathcal{C}$ is a 
(locally) small $(\infty,1)$-category (Example~\ref{exmplesize}).
\item Given a cartesian fibration $p\colon\mathcal{E}\twoheadrightarrow\mathcal{C}$ and an $(\infty,1)$-category $\mathcal{D}$ with pullbacks 
together with a left adjoint functor $F\colon\mathcal{D}\rightarrow\mathcal{C}$, we have $\mathrm{Comp}(F^{\ast}p)\subseteq\mathrm{Comp}(p)$ 
(Lemma~\ref{lemmachangeofbasecomp}).
\item The universal cartesian fibration $\pi^{op}\colon\mathrm{Dat}_{\infty}^{op}\twoheadrightarrow\mathrm{Cat}_{\infty}^{op}$ and the 
universal right fibration $\mathcal{S}_{\ast}^{op}\twoheadrightarrow\mathcal{S}^{op}$ have $G$-comprehension for every functor 
$G\colon\mathcal{I}\rightarrow\mathcal{J}$ between small $(\infty,1)$-categories $\mathcal{I}$, $\mathcal{J}$ 
(Example~\ref{exmplecompschemesidcat}). 
\end{itemize}
\end{examples}

In Section~\ref{secext} we define the externalization construction 
\begin{align}\label{equintroext}
\mathrm{Ext}\colon\mathrm{CS}(\mathcal{C)}\rightarrow\mathrm{Fun}(\mathcal{C}^{op},\mathrm{Cat}_{\infty})
\end{align}
of $(\infty,1)$-categories internal to an $(\infty,1)$-category $\mathcal{C}$ with pullbacks. The functor (\ref{equintroext}) is essentially given as a pushforward of 
the Yoneda embedding of $y\colon\mathcal{C}\rightarrow\hat{\mathcal{C}}$, and is hence a direct generalization of the synonymous ordinary 
categorical construction e.g.\ given in \cite[Section 7.3]{jacobsttbook}. It hence shares many properties of the Yoneda embedding; in fact 
when restricted to the $(\infty,1)$-category of internal $\infty$-groupoids it recovers $y$ up to equivalence. 
The functor $\mathrm{Ext}$ furthermore is fully faithful as well, and allows the detection of the according ``representables'' via a 
smallness criterion: namely, we will see that the indexed $(\infty,1)$-categories arising from internal $(\infty,1)$-categories in this way 
can be characterized as exactly the globally small and locally small ones. 
\begingroup
\def\thetheorem{\ref{thmsmall=ext}}
\begin{theorem}
Let $\mathcal{C}$ be an $(\infty,1)$-category with finite limits and $\mathcal{E}$ be a $\mathcal{C}$-indexed $(\infty,1)$-category. Then
$\mathcal{E}$ is represented by an internal $(\infty,1)$-category if and only if it is both globally small and locally small.
\end{theorem}
\addtocounter{theorem}{-1}
\endgroup

Given the chosen definition of global smallness in terms of comprehension, this is a result which holds in ordinary category theory only with 
some caveats, see Remark~\ref{reminftysmall} for more details. 

Two fundamental examples of such representable cartesian fibrations are given by the universal ($\kappa$-small) cartesian 
fibration $\mathrm{Dat}_{\infty}^{op}\twoheadrightarrow\mathrm{Cat}_{\infty}^{op}$ and the universal ($\kappa$-small) right fibration
$\mathcal{S}_{\ast}^{op}\twoheadrightarrow\mathcal{S}^{op}$ as constructed in \cite[Section 3.3.2]{luriehtt}.\footnote{We will eventually 
show that the domain $\mathrm{Dat}_{\infty}$ is the $(\infty,1)$-category $(\mathrm{Cat}_{\infty})_{\Delta^{\bullet}/ }$ of ``oplax-pointed''
$(\infty,1)$-categories. At least until then an abstract notation is necessary to refer to it given that it is constructed by entirely 
abstract means. We chose $\mathrm{Dat}_{\infty}\twoheadrightarrow\mathrm{Cat}_{\infty}$ for this purpose simply because the fibration 
functorially associates to an $(\infty,1)$-category $\mathcal{C}$ its higher categorical data $\Delta^n\rightarrow\mathcal{C}$.}  
There, Lurie showed that the universal right fibration is represented by the terminal object in the $(\infty,1)$-category
$\mathcal{S}$ of spaces. The universal cartesian fibration on the other hand cannot be representable by an object in $\mathrm{Cat}_{\infty}$ 
in the same sense, but we will see that it is representable in terms of an internal $(\infty,1)$-category via the externalization functor 
(Example~\ref{exmpleuniversalext}). The associated internal $(\infty,1)$-category is given by the interval object
$\Delta^{\bullet}\in\mathrm{Fun}(N(\Delta),\mathrm{Cat}_{\infty})$. In fact, it will 
follow from general considerations about comprehension schemes that the universal right fibration must then be 
represented by the pointwise invertible interval $I\Delta^{\bullet}$ in $\mathcal{S}$ in the same way 
(Example~\ref{exmpleunivcomptransfer}, Remark~\ref{exmpleuniversalext}). But this object is contractible in
$\mathrm{Fun}(N(\Delta),\mathcal{S})$, which recovers Lurie's original result by other means. These results have various applications. For 
instance, they allow for the definition of cotensors $X^J$ of internal $(\infty,1)$-categories $X\in\mathrm{CS}(\mathcal{C})$ with finite
$(\infty,1)$-categories $J$, and furthermore with arbitrarily sized $(\infty,1)$-categories $J$ whenever $\mathcal{C}$ is complete 
(Remark~\ref{remcorfinitecomp}). They also allow us to prove a generalization of the well-known fact that the limit preserving 
presheaves over a presentable $(\infty,1)$-category are exactly the representable ones.
\begingroup
\def\thetheorem{\ref{proprightadj}}
\begin{proposition}
Suppose $\mathcal{C}$ is a complete $(\infty,1)$-category. Then a $\mathcal{C}$-indexed $(\infty,1)$-category $\mathcal{E}$ is small if and 
only if the functor $\mathcal{E}\colon\mathcal{C}^{op}\rightarrow\mathrm{Cat}_{\infty}$  has a left adjoint.
\end{proposition}
\addtocounter{theorem}{-1}
\endgroup
In Section~\ref{seccompcats} we use the interplay between indexed and fibered structures to discuss various alternative characterizations of 
general comprehension schemes (Proposition~\ref{propaltdefgencomp}, Remark~\ref{remaltdefgencomp}), to show the existence of a universal such 
comprehension scheme whenever $\mathcal{C}$ is small and has pullbacks (Proposition~\ref{propunivreptransf}), and to
discuss definability and smallness in the context of full subfibrations. For instance, we show the following.
\begingroup
\def\thetheorem{\ref{propintfunctorscomp}}
\begin{proposition}
Let $\mathcal{C}$ be an $(\infty,1)$-category with finite limits, let $Y$ be an internal $(\infty,1)$-category in $\mathcal{C}$, and let
$f\colon p\rightarrow\mathrm{Ext}(Y)$ be a cartesian functor over $\mathcal{C}$. Then $p$ is represented by an internal $(\infty,1)$-category 
if and only if $\mathcal{C}$ has $f$-comprehension. 
In particular, $\mathcal{C}$ has $\mathrm{Ext}(f)$-comprehension for all internal functors $f\colon X\rightarrow Y$ in
$\mathrm{CS}(\mathcal{C})$.
\end{proposition}
\addtocounter{theorem}{-1}
\endgroup
In Section~\ref{secextuniv} we apply our results about the externalization construction from Section~\ref{secext} to 
full and replete comprehension $(\infty,1)$-categories over a fixed $(\infty,1)$-category $\mathcal{C}$ with 
pullbacks in the sense of Jacobs \cite{jacobscompcats} as introduced in Section~\ref{seccompcats}. Here we show that 
global smallness characterizes univalent morphisms in $\mathcal{C}$ whenever $\mathcal{C}$ has enough structure to 
define univalence in the first place.
\begingroup
\def\thetheorem{\ref{propunivalentcomprehensioncats}}
\begin{proposition}
Suppose $\mathcal{C}$ is left exact and locally cartesian closed. Given a full and replete comprehension $(\infty,1)$-category $J$ over
$\mathcal{C}$, the following are equivalent.
\begin{enumerate}
\item $J\twoheadrightarrow\mathcal{C}$ is globally small.
\item $J$ is represented by an internal $(\infty,1)$-category.
\item There is a univalent morphism $q\colon E\rightarrow B$ in $\mathcal{C}$ that classifies the morphisms in $J$.
\end{enumerate}
\end{proposition}
\addtocounter{theorem}{-1}
\endgroup
Thus, first, via Proposition~\ref{proprightadj} above this generalizes Gepner and Kock's characterization of univalent 
morphisms in presentable $(\infty,1)$-categories in terms of an associated sheaf property (\cite[Proposition 3.8]{gepnerkock}, 
Corollary~\ref{corcharunivalencela}).
And second, it allows us to characterize (elementary) $\infty$-toposes entirely in terms of comprehension schemes: they are the finitely 
bicomplete $\infty$-categories which -- while generally not globally small over themselves -- can instead be covered by a collection of small 
``neighbourhoods'' (Corollary~\ref{chareltopssmall}). This entails the following characterization of (Grothendieck) $\infty$-toposes.
\begingroup
\def\thetheorem{\ref{chartopssmall}}
\begin{corollary}
A presentable $(\infty,1)$-category $\mathcal{C}$ is an $\infty$-topos if and only if the following two conditions hold.
\begin{enumerate}
\item The canonical fibration $t\colon\mathrm{Fun}(\Delta^1,\mathcal{C})\twoheadrightarrow\mathcal{C}$ is locally small.
\item For all sufficiently large regular cardinals $\kappa$, the full comprehension $(\infty,1)$-category
$t_{\kappa}\colon\mathrm{Fun}(\Delta^1,\mathcal{C})_{\kappa}\twoheadrightarrow\mathcal{C}$
is small.
\end{enumerate}
\end{corollary}
\addtocounter{theorem}{-1}
\endgroup
Lastly, in Section~\ref{secmodindep} we use the fibered formulations of both general and diagrammatic comprehension to give a simple proof of 
model independence of our results and constructions. 

\begin{notation}\label{notationglobal}
The notation will follow Lurie \cite{luriehtt} in large. In particular, for the sake of readability, whenever we say ``$\infty$-category'' we 
mean $(\infty,1)$-category. The exponential $\mathcal{D}^{\mathcal{C}}$ of two quasi-categories $\mathcal{C}$, $\mathcal{D}$ in the category 
$\mathbf{S}$ of simplicial sets will be denoted by $\mathrm{Fun}(\mathcal{C},\mathcal{D})$. The large quasi-category of small
quasi-categories as defined and studied in \cite[Chapter 3]{luriehtt} will be denoted by $\mathrm{Cat}_{\infty}$. Its (small) hom-spaces
$\mathrm{Cat}_{\infty}(\mathcal{C},\mathcal{D})$ are given by the core $\mathrm{Fun}(\mathcal{C},\mathcal{D})^{\simeq}$, 
that is, the largest Kan complex contained in the quasi-category $\mathrm{Fun}(\mathcal{C},\mathcal{D})$. The quasi-category of spaces will 
be denoted by $\mathcal{S}$, the (super-large) quasi-category of large quasi-categories will be denoted by $\mathrm{CAT}_{\infty}$. A
quasi-category will be said to be small precisely if it is contained in $\mathrm{Cat}_{\infty}$. Smallness, largeness (and super-largeness) 
are relative notions; thus, to not carry around cardinals whose absolute size does not matter after all as all results apply to all regular 
cardinals polymorphically, ``smallness'' may refer to $\kappa$-smallness for some regular cardinal $\kappa$. In particular, we 
will work with a single universal fibration $\pi\colon\mathrm{Dat}_{\infty}\twoheadrightarrow\mathrm{Cat}_{\infty}$ which is to be thought of 
as a place-holder for the universal $\kappa$-small fibration
$\pi_{\kappa}\colon(\mathrm{Dat}_{\infty})_{\kappa}\twoheadrightarrow(\mathrm{Cat}_{\infty})_{\kappa}$ for any given regular cardinal
$\kappa$ (as introduced in \cite[Section 3.3.2]{luriehtt}). Furthermore, we will implicitly assume that a general base quasi-category 
is locally (essentially) small by definition.

Given a quasi-category $\mathcal{C}$ and an object $C\in\mathcal{C}$, we denote the associated slice quasi-category (called
\emph{overcategory} in \cite{luriehtt}) by $\mathcal{C}_{/C}$. To avoid awkward iterated hyphenations, we will 
refer to quasi-categories $\mathcal{C}$ together with an inclusion of simplicial sets into another quasi-category $\mathcal{D}$ 
as quasi-subcategories of $\mathcal{D}$. A full quasi-subcategory is hence such where the respective inclusion of simplicial 
sets is the identity on the sets of edges. 

The standard $n$-simplex in the category of simplicial sets will be denoted by $\Delta^n$.

There are two model structures on the category $\mathbf{S}$ of particular importance for this paper. These are the 
Quillen model structure $(\mathbf{S},\mathrm{Kan})$ for Kan complexes and the Joyal model structure
$(\mathbf{S},\mathrm{QCat})$ for quasi-categories.
\end{notation}

While the main results of this paper are invariant under the choice of any common model of $(\infty,1)$-category theory (see 
Section~\ref{secmodindep}), the author chose to work with the analytical presentation of quasi-categories for the sake of computations, 
argumentation and constructions by virtue of the extensive use of the results in the literature. As far as the results are 
concerned, the terms ``quasi-category'' and ``$\infty$-category'' will be used interchangeably.

\begin{acknowledgments*}
The author would like to thank John Bourke for multiple stimulating discussions, as well as an anonymous referee for many 
valuable comments. The initial impetus to rewrite the paper with some more general notion of comprehension (rather than the standard 
diagrammatic one only) at its center, and furthermore to define it in indexed rather than fibered terms is due to the same referee. The first 
draft of this paper was written at Masaryk University with support of the Grant Agency of the Czech Republic under the
grant 19-00902S. The current version of this paper was written as a guest at the Max Planck Institute for Mathematics in Bonn, Germany,
whose hospitality is greatly appreciated.
\end{acknowledgments*}

\section{Preliminaries on indexed and fibered $(\infty,1)$-category theory}\label{secpre}

In order to define and study comprehension schemes for cartesian fibrations of $\infty$-categories, in this section we 
cover, first, the underlying basic definitions and examples of indexed $\infty$-categories and cartesian 
fibrations, second, the $\infty$-categorical Grothendieck construction in form of Lurie's straightening and 
unstraightening pair, and, lastly, a few definitions and statements regarding representable natural transformations.
The presentation of the material up to the representable natural transformations is largely non-technical and mainly serves 
the purpose of introducing and organizing the notions and examples to be studied in later sections. The underlying calculations and proofs 
are contained in \cite{luriehtt} where all given $\infty$-categorical concepts are developed in detail. 

\begin{definition}\label{defindcat}
Let $\mathcal{C}$ be an $\infty$-category. A \emph{$\mathcal{C}$-indexed $\infty$-category} is a functor
$\mathcal{E}\colon\mathcal{C}^{op}\rightarrow \mathrm{Cat}_{\infty}$.
If $\mathcal{C}$ has a terminal object $1\in\mathcal{C}$, we say that $\mathcal{E}(1)$ is the underlying
$\infty$-category of $\mathcal{E}$.
A \emph{functor of $\mathcal{C}$-indexed $\infty$-categories} $\mathcal{E},\mathcal{F}$ is simply a morphism 
in the $\infty$-category $\mathrm{Fun}(\mathcal{C}^{op},\mathrm{Cat}_{\infty})$ of $\mathcal{C}$-indexed $\infty$-categories, that is, a 
natural transformation from $\mathcal{E}$ to $\mathcal{F}$.
Likewise, a $\mathcal{C}$-indexed $\infty$-groupoid is a functor $X\colon\mathcal{C}^{op}\rightarrow\mathcal{S}$.
The $\infty$-category of $\mathcal{C}$-indexed $\infty$-groupoids is the $\infty$-category
$\hat{\mathcal{C}}\simeq\mathrm{Fun}(\mathcal{C}^{op},\mathcal{S})$ of presheaves on $\mathcal{C}$.  	
\end{definition}

The following examples are generalizations of some of the 1-categorical examples considered in \cite[Section B1.2]{elephant}.

\begin{examples}\label{exmplebasicindcats} 
\begin{enumerate}
\item For $1\in\mathrm{Cat}_{\infty}$ the terminal $\infty$-category, the $\infty$-category
$\mathrm{Fun}(1^{op},\mathrm{Cat}_{\infty})$ of $1$-indexed $\infty$-categories is just the $\infty$-category $\mathrm{Cat}_{\infty}$ of 
small $\infty$-categories itself.
\end{enumerate}
Let $\mathcal{C}$ be an $\infty$-category.
\begin{enumerate}
\item[2.] Whenever $\mathcal{C}$ is small, the exponential
\[\mathrm{Fun}(\phv,\mathcal{C})\colon\mathrm{Cat}_{\infty}^{op}\rightarrow\mathrm{Cat}_{\infty}\]
is a $\mathrm{Cat}_{\infty}$-indexed $\infty$-category and is represented by $\mathcal{C}$ in an
$(\infty,2)$-categorical sense.
Its restriction to the category $\mathcal{S}$ of spaces is the ``naive'' $\mathcal{S}$-indexed $\infty$-category
$\mathrm{Fun}(\phv,\mathcal{C})\colon\mathcal{S}^{op}\rightarrow\mathrm{Cat}_{\infty}$. Its underlying
$\infty$-category is $\mathcal{C}$ itself.
\item[3.] If $\mathcal{C}$ is small and has pullbacks, the slice functor
$\mathcal{C}_{/(\cdot)}\colon\mathcal{C}^{op}\rightarrow\mathrm{Cat}_{\infty}$ is the ``canonical indexing of
$\mathcal{C}$ over itself''. It takes an object $C\in\mathcal{C}$ to the sliced $\infty$-category $\mathcal{C}_{/C}$ and a morphism
$f\colon C\rightarrow D$ in $\mathcal{C}$ to the pullback functor $f^{\ast}\colon \mathcal{C}_{/D}\rightarrow\mathcal{C}_{/C}$. Technically, 
it is defined as the unstraightening of the target fibration over $\mathcal{C}$ (Example~\ref{explscartfibs}.2). Again, its 
underlying $\infty$-category is $\mathcal{C}$ whenever $\mathcal{C}$ is left exact. When applied to $\mathcal{C}\simeq\mathcal{S}$, the naive 
indexing $\mathrm{Fun}(\phv,\mathcal{S})\colon\mathcal{S}^{op}\rightarrow\mathrm{CAT}_{\infty}$ and 
the canonical indexing $\mathcal{S}_{/(\cdot)}\colon\mathcal{S}^{op}\rightarrow\mathrm{CAT}_{\infty}$ are pointwise 
equivalent by the (unmarked) straightening construction over Kan complexes \cite[Theorem 2.2.1.2]{luriehtt}. The 
straightening construction will be discussed below in more generality.

\item[4.] Given a $\mathcal{C}$-indexed $\infty$-category
$\mathcal{E}\colon\mathcal{C}^{op}\rightarrow\mathrm{Cat}_{\infty}$, each functor
$F\colon\mathcal{D}\rightarrow\mathcal{C}$ induces a change of base
$F^{\ast}\mathcal{E}\colon\mathcal{D}^{op}\rightarrow\mathrm{Cat}_{\infty}$ via precomposition with $F^{op}$. It defines a functor
\[F^{\ast}\colon\mathrm{Fun}(\mathcal{C}^{op},\mathrm{Cat}_{\infty})\rightarrow\mathrm{Fun}(\mathcal{D}^{op},\mathrm{Cat}_{\infty}).\]

\item[5.] The most fundamental examples of indexed $\infty$-groupoids are the representable presheaves
$y(C)=\mathcal{C}(\phv,C)$. Whenever $\mathcal{C}$ has finite limits, they are 
the archetypes of ``small'' indexed $\infty$-groupoids over $\mathcal{C}$ and as such they will serve as reference 
structures in the definition of comprehension schemes in Section~\ref{seccomp}. 

\item[6.] If $\mathcal{C}$ has finite products, we obtain an action
$(-)^{(\cdot)}\colon\mathcal{C}^{op}\times\mathrm{Fun}(\mathcal{C}^{op},\mathrm{Cat}_{\infty})\rightarrow\mathrm{Fun}(\mathcal{C}^{op},\mathrm{Cat}_{\infty})$ given on objects 
by $\mathcal{E}^{C}(C\sprime)=\mathcal{E}(C\sprime\times C)$. It is the Currying of the change of base along the product
$\times\colon\mathcal{C}\times\mathcal{C}\rightarrow\mathcal{C}$. Whenever $\mathcal{C}$ is small, this action induces a functor
\[[\phv,\phv]\colon\mathrm{Fun}(\mathcal{C}^{op},\mathrm{Cat}_{\infty})^{op}\times\mathrm{Fun}(\mathcal{C}^{op},\mathrm{Cat}_{\infty})\rightarrow\hat{\mathcal{C}}\]
as the curried adjoint of the composition of
\[\mathrm{Fun}(\mathcal{C}^{op},\mathrm{Cat}_{\infty})^{op}\times\mathcal{C}^{op}\times\mathrm{Fun}(\mathcal{C}^{op},\mathrm{Cat}_{\infty})\xrightarrow{1\times(-)^{(\cdot)}}\mathrm{Fun}(\mathcal{C}^{op},\mathrm{Cat}_{\infty})^{op}\times\mathrm{Fun}(\mathcal{C}^{op},\mathrm{Cat}_{\infty})\]
with the hom-functor
$\mathrm{Hom}\colon\mathrm{Fun}(\mathcal{C}^{op},\mathrm{Cat}_{\infty})^{op}\times\mathrm{Fun}(\mathcal{C}^{op},\mathrm{Cat}_{\infty})\rightarrow\mathcal{S}$.
For $\mathcal{C}$-indexed $\infty$-categories $\mathcal{E}$, $\mathcal{F}$ one may thus form the
$\mathcal{C}$-indexed $\infty$-groupoid $[\mathcal{E},\mathcal{F}]$ of
$\mathcal{C}$-indexed functors between them. It is given pointwise by the mapping spaces
\[[\mathcal{E},\mathcal{F}](C)=\mathrm{Fun}(\mathcal{C}^{op},\mathrm{Cat}_{\infty})(\mathcal{E},\mathcal{F}^C),\]
its underlying $\infty$-category is the mapping space $\mathrm{Fun}(\mathcal{C}^{op},\mathrm{Cat}_{\infty})(\mathcal{E},\mathcal{F})$.

Generally, given that $\mathcal{C}$ is small, it follows for instance from the Yoneda lemma for indexed
$\infty$-categories \cite[Theorem 5.7.3]{riehlverityelements} that the $\infty$-category
$\mathrm{Fun}(\mathcal{C}^{op},\mathrm{Cat}_{\infty})$ is cartesian closed (see e.g.\
\cite[Section 9]{barwicketalparahct}; to make sense of this however requires to consider
$\mathrm{Fun}(\mathcal{C}^{op},\mathrm{Cat}_{\infty})$ together with its canonical $(\infty,2)$-categorical structure). Then, whenever
$\mathcal{C}$ has finite products, for $\mathcal{F}\in\mathrm{Fun}(\mathcal{C}^{op},\mathrm{Cat}_{\infty})$ and $C\in\mathcal{C}$ the
$\mathcal{C}$-indexed $\infty$-category $\mathcal{F}^C$ computes the exponential
$\mathcal{F}^{y(C)}\in\mathrm{Fun}(\mathcal{C}^{op},\mathrm{Cat}_{\infty})$, and $[\mathcal{E},\mathcal{F}]$ computes pointwise the maximal
$\infty$-subgroupoids contained in the exponential $\mathcal{F}^{\mathcal{E}}\in\mathrm{Fun}(\mathcal{C}^{op},\mathrm{Cat}_{\infty})$.
\end{enumerate}
\end{examples}

In ordinary category theory -- that is, the $2$-category theory of 1-categories -- the Grothendieck construction of 
pseudo-functors into the category $\mathrm{Cat}$ of small categories yields a 2-categorical equivalence between 
indexed category theory and fibered category theory, see e.g.\ \cite[Theorem 1.3.6]{elephant}. An
$(\infty,1)$-categorical analogue of the Grothendieck construction was first given by Lurie's (un)straightening 
construction \cite[Section 3.2]{luriehtt}, which provides a very general framework to pass 
between cartesian fibrations over a simplicial set $S$ and functors from $S^{op}$ into the $\infty$-category
$\mathrm{Cat}_{\infty}$. This construction yields an equivalence
\[\xymatrix{
\text{St}_{\mathcal{C}}\colon\text{Cart}(\mathcal{C})\ar@<1ex>[r]\ar@{}[r]|(.4){\simeq} & \text{Fun}(\mathcal{C}^{op},\mathrm{Cat}_{\infty})\colon\mathrm{Un}_{\mathcal{C}}\ar@<1ex>[l]
}\]
between the $\infty$-category of cartesian fibrations and cartesian functors over a small $\infty$-category
$\mathcal{C}$ \cite[Definition 2.4.2.1]{luriehtt} and the $\infty$-category of $\mathcal{C}$-indexed small
$\infty$-categories. The fibers of the fibration $\text{Un}_{\mathcal{C}}(\mathcal{E})\in\mathrm{Cart}(\mathcal{C})$ 
are equivalent to the values of the functor $\mathcal{E}\colon\mathcal{C}^{op}\rightarrow\mathrm{Cat}_{\infty}$ on 
the objects of $\mathcal{C}$. Furthermore, the construction yields an equivalence
\[\xymatrix{
\text{St}_{\mathcal{C}}\colon\text{RFib}(\mathcal{C})\ar@<1ex>[r]\ar@{}[r]|(.45){\simeq} & \text{Fun}(\mathcal{C}^{op},\mathcal{S})\colon\mathrm{Un}_{\mathcal{C}}\ar@<1ex>[l]
}\]
between the $\infty$-category of right fibrations over $\mathcal{C}$ \cite[Section 2]{luriehtt} and the
$\infty$-category of $\mathcal{C}$-indexed $\infty$-groupoids (presheaves over $\mathcal{C}$, that is). We will omit 
the index $\mathcal{C}$ and simply write $\mathrm{St}$ and $\mathrm{Un}$ to refer to the respective straightening and 
unstraightening functors whenever the base $\mathcal{C}$ is clear from context.
In this sense, indexed $\infty$-category theory and fibered $\infty$-category theory are equivalent notions for all 
purposes of this paper.\footnote{In Section~\ref{secsubcartfib} we will implicitly need that unstraightening in fact
constitutes an equivalence of $(\infty,2)$-categories; this will be addressed in due course however.}

\begin{definition}\label{defcore}
Let $(\cdot)^{\simeq}\colon\mathrm{Cat}_{\infty}\rightarrow\mathcal{S}$ be the functor that assigns to each
$\infty$-category its core, i.e.\ $\mathcal{C}^{\simeq}$ is the largest $\infty$-groupoid contained in $\mathcal{C}$.
Then postcomposition of an indexed $\infty$-category
$\mathcal{E}\colon\mathcal{C}^{op}\rightarrow\mathrm{Cat}_{\infty}$ with this core functor yields a presheaf
$\mathcal{E}^{\simeq}\colon\mathcal{C}^{op}\rightarrow\mathcal{S}$.

Given a cartesian fibration $p\colon\mathcal{E}\twoheadrightarrow\mathcal{C}$, we denote the unstraightening of
$\mathrm{St}(p)^{\simeq}$ by $p^{\times}\colon\mathcal{E}^{\times}\twoheadrightarrow\mathcal{C}$ and call it \emph{the core 
of $p$}.
\end{definition}

Given a cartesian fibration $p\colon\mathcal{E}\twoheadrightarrow\mathcal{C}$, the domain $\mathcal{E}^{\times}$ is an
$\infty$-category itself. It can be constructed as the wide $\infty$-subcategory $\mathcal{E}^{\times}\subseteq\mathcal{E}$ spanned by the 
$p$-cartesian morphisms in $\mathcal{E}$. The restriction $p|_{\mathcal{E}^{\times}}$ is a right fibration by 
\cite[Corollary 2.4.2.5]{luriehtt} and is equivalent to the right fibration $p^{\times}$. Thus, the
$p^{\times}$-cartesian morphisms in $\mathcal{E}^{\times}$ are exactly the $p$-cartesian morphisms in
$\mathcal{E}$, and the vertical morphisms in $\mathcal{E}^{\times}$ are exactly the vertical equivalences in $\mathcal{E}$. 

\begin{examples}\label{explscartfibs}
Let $\mathcal{C}$ be an $\infty$-category.
\begin{enumerate}
\item Whenever $\mathcal{C}$ is small, the forgetful functor
$s\colon(\mathrm{Cat}_{\infty})_{/\mathcal{C}}\twoheadrightarrow\mathrm{Cat}_{\infty}$ is the 
unstraightening of the representable functor
$\mathrm{Cat}_{\infty}(\phv,\mathcal{C})\colon\mathrm{Cat}_{\infty}^{op}\rightarrow\mathcal{S}$. It therefore gives an 
example of a representable fibration as to be defined below. Similarly, the forgetful functor
$s\colon\mathcal{S}_{/\mathcal{C}}\twoheadrightarrow\mathcal{S}$ is the unstraightening of the restriction
$\mathrm{Cat}_{\infty}(\phv,\mathcal{C})\colon\mathcal{S}^{op}\rightarrow\mathcal{S}$.

More generally, we may consider the unstraightening of the exponential
\[\text{Fun}(\phv,\mathcal{C})\colon\mathrm{Cat}_{\infty}^{op}\rightarrow\mathrm{Cat}_{\infty}\]
and the unstraightening of the naive indexing
\[\mathrm{Fun}(\phv,\mathcal{C})\colon\mathcal{S}^{op}\rightarrow\mathrm{Cat}_{\infty}\]
of $\mathcal{C}$. These are the fibration of diagrams and the fibration of families of $\mathcal{C}$, respectively, which generalize the 
classic construction of the 1-category $\mathrm{Fam}(\mathbb{C})$ of families over a 1-category $\mathbb{C}$. By construction, the core of 
the fibration of diagrams of $\mathcal{C}$ over $\mathrm{Cat}_{\infty}$ is exactly the representable 
fibration $s\colon(\mathrm{Cat}_{\infty})_{/\mathcal{C}}\twoheadrightarrow\mathrm{Cat}_{\infty}$.
\item The codomain functor $t\colon\mathrm{Fun}(\Delta^1,\mathcal{C})\twoheadrightarrow \mathcal{C}$ is a cartesian 
fibration whenever $\mathcal{C}$ has pullbacks. If $\mathcal{C}$ is also small, its unstraightening is the canonical indexing of
$\mathcal{C}$ over itself (by definition). It therefore will be called the canonical fibration over $\mathcal{C}$.

\item Change of base of an indexed $\infty$-category $\mathcal{E}\colon\mathcal{C}^{op}\rightarrow\mathrm{Cat}_{\infty}$ 
along a functor $F\colon\mathcal{D}\rightarrow\mathcal{C}$ unstraightens to the pullback of fibrations
\[\xymatrix{
\text{Un}(F^{\ast}\mathcal{I})\ar@{->>}[d]\ar[r]\ar@{}[dr]|(.3){\pbs} & \text{Un}(\mathcal{I})\ar@{->>}[d] \\
\mathcal{D}\ar[r]_F & \mathcal{C}.
}\]
This follows directly from \cite[Proposition 3.2.1.4]{luriehtt}. 

Up to cartesian equivalence, every cartesian fibration $p\colon\mathcal{E}\twoheadrightarrow\mathcal{C}$ with essentially small fibers is the 
change of base of the universal cartesian fibration
\[\pi^{op}\colon\mathrm{Dat}_{\infty}^{op}\twoheadrightarrow\mathrm{Cat}_{\infty}^{op}\]
along its straightening $\mathrm{St}(p)^{op}\colon\mathcal{C}\rightarrow\mathrm{Cat}_{\infty}^{op}$. The universal cocartesian fibration
$\pi\colon\mathrm{Dat}_{\infty}\twoheadrightarrow\mathrm{Cat}_{\infty}$ itself is the unstraightening of the identity
$\mathrm{id}\colon\mathrm{Cat}_{\infty}\rightarrow\mathrm{Cat}_{\infty}$ \cite[Section 3.3.2]{luriehtt}.

In particular, whenever $\mathcal{C}$ is small and has pullbacks, the canonical fibration
$t\colon\mathrm{Fun}(\Delta^1,\mathcal{C})\twoheadrightarrow \mathcal{C}$ is equivalent to the change of base of the 
universal cartesian fibration along the canonical indexing of $\mathcal{C}$. In particular, we obtain a
homotopy-cartesian square
\[\xymatrix{
\mathrm{Fun}(\Delta^1,\mathcal{C})\ar@{->>}[d]_t\ar[r] & \mathrm{Dat}_{\infty}^{op}\ar@{->>}[d]^{\pi^{op}} \\
\mathcal{C}\ar[r]_{\mathcal{C}_{/(\cdot)}} & \mathrm{Cat}_{\infty}^{op}
}\]
in the Joyal model structure $(\mathbf{S},\mathrm{QCat})$.

Change of base along some functor $F\colon\mathcal{D}\rightarrow\mathcal{C}$ of the canonical indexing of
$\mathcal{C}$ over itself gives rise to the ``Artin gluing'' of $F$:
\[\xymatrix{
\mathcal{C}\downarrow F \ar@{->>}[d]_{\text{gl}(F)}\ar[r]\ar@{}[dr]|(.3){\pbs} & \mathrm{Fun}(\Delta^1,\mathcal{C})\ar@{->>}[d]^t \\
\mathcal{D}\ar[r]_F & \mathcal{C}.
}\]
One computes that the $\infty$-category of cartesian sections of the Artin gluing $\mathrm{gl}(F)$ is the
$\infty$-category of cartesian natural transformations over $F$, and as such it is equivalent to the limit of the 
functor $F^{\ast}(\mathcal{C}_{/(\cdot)})\colon\mathcal{D}^{op}\rightarrow\mathrm{Cat}_{\infty}$ 
\cite[Section 3.3.3]{luriehtt}. The localization of the $\infty$-category $\mathcal{C}\downarrow F$ at the class of
$\mathrm{gl}(F)$-cartesian morphisms is equivalent to the colimit of the very same functor
\cite[Section 3.3.4]{luriehtt}.

\item Similarly, up to cartesian equivalence, every right fibration with small fibers is the change of base of the universal 
right fibration $\mathcal{S}_{\ast}^{op}\twoheadrightarrow\mathcal{S}^{op}$ along its straightening.

For example, every object $C\in\mathcal{C}$ yields a functor
$\mathrm{ev}(C)\colon\hat{\mathcal{C}}\rightarrow\mathcal{S}$ by evaluation of presheaves over $\mathcal{C}$ at $C$. If 
by $y\colon\mathcal{C}\rightarrow\hat{\mathcal{C}}$ we denote the Yoneda embedding, the associated right fibration is 
given by the forgetful functor $s\colon\hat{\mathcal{C}}^{op}_{/y(C)}\twoheadrightarrow\hat{\mathcal{C}}^{op}$. That 
means the square
\[\xymatrix{
\hat{\mathcal{C}}_{y(C)/}\ar@{->>}[d]\ar[r] & \mathcal{S}_{\ast}\ar@{->>}[d] \\
\hat{\mathcal{C}}\ar[r]_{\mathrm{ev}(C)} & \mathcal{S}
}\]
is homotopy-cartesian in the Joyal model structure. This is shown in \cite[Theorem 5.1.5.2]{luriehtt}, but also
follows from Corollary~\ref{corextexchange} and the fact that the evaluation functor at $C$ has a left adjoint.
\end{enumerate}
\end{examples}

\begin{definition}
A right fibration $p\colon\mathcal{E}\twoheadrightarrow\mathcal{C}$ is \emph{representable} if there is a representable presheaf over
$\mathcal{C}$ whose unstraightening is equivalent to $p$ over $\mathcal{C}$.
\end{definition}

Thus, up to equivalence, the representable right fibrations over an $\infty$-category $\mathcal{C}$ are exactly the right fibrations 
of the form $s\colon\mathcal{C}_{/C}\twoheadrightarrow\mathcal{C}$ for objects $C$ in $\mathcal{C}$. To verify whether a given right 
fibration is representable without having to compute its straightening, we may apply the following lemma.

\begin{lemma}[{\cite[Proposition 4.4.4.5]{luriehtt}}]\label{lemmareprightfibs}
A right fibration $p\colon\mathcal{E}\twoheadrightarrow\mathcal{C}$ is representable if and only if
the domain $\mathcal{E}$ has a terminal object $t\in\mathcal{E}(C)$; then $\mathcal{E}\simeq\mathcal{C}_{/C}$ over $\mathcal{C}$.\qed
\end{lemma}

\begin{definition}\label{defrepnattransf}
Let $\mathcal{C}$ be an $\infty$-category. A natural transformation $f\colon X\rightarrow Y$ in $\hat{\mathcal{C}}$ is 
\emph{representable} if for every object $C\in\mathcal{C}$ and every natural transformation $v\colon y(C)\rightarrow Y$, the pullback
$v^{\ast}X\in\hat{\mathcal{C}}$ is a representable presheaf.
\end{definition}

Given an $\infty$-category $\mathcal{C}$ and an object $C\in\mathcal{C}$, there is a natural equivalence
$\hat{\mathcal{C}}_{/y(C)}\simeq\widehat{\mathcal{C}_{/C}}$ of $\infty$-categories. Under this equivalence, a natural transformation 
$f\colon X\rightarrow Y$ in $\hat{\mathcal{C}}$ is representable as in Definition~\ref{defrepnattransf} if and only if for all 
generalized elements $v\colon y(C)\rightarrow Y$, the pullback $v^{\ast}f\in\widehat{\mathcal{C}_{/C}}$ is a representable presheaf.

\begin{example}\label{explerepslenattransfoverpt}
Let $\mathcal{C}$ be an $\infty$-category. By definition, a natural transformation $X\rightarrow\ast$ over the terminal presheaf in
$\hat{\mathcal{C}}$ is representable if and only if for all objects $C\in\mathcal{C}$ the product $y(C)\times X$ is a representable presheaf. 
Whenever $\mathcal{C}$ has finite products, this holds if and only if the presheaf $X$ is representable itself. Indeed, the Yoneda embedding 
is limit preserving and so if every $y(C)\times X$ is representable, then in particular so is $y(\ast)\times X\simeq \ast\times X\simeq X$. 
And if $X\simeq y(D)$ for some $D\in\mathcal{C}$, then $y(C)\times X\simeq y(C)\times y(D)\simeq y(C\times D)$ is representable as well.
\end{example}

\begin{proposition}\label{proprepcartmaps}
Let $\mathcal{C}$ be an $\infty$-category. A natural transformation $f\colon X\rightarrow Y$ in $\hat{\mathcal{C}}$ is 
representable if and only if its unstraightening
\[\mathrm{Un}(f)\colon\mathrm{Un}(X)\rightarrow\mathrm{Un}(Y)\]
over $\mathcal{C}$ has a (generally not fibered) right adjoint.
\end{proposition}
\begin{proof}
The functor $\mathrm{Un}(f)\colon\mathrm{Un}(X)\rightarrow\mathrm{Un}(Y)$ has a right adjoint (as a functor of total $\infty$-categories) if 
and only if for every vertex $v\in\mathrm{Un}(Y)$, the comma $\infty$-category $\mathrm{Un}(f)\downarrow v$ defined as the pullback
\begin{align}\label{diagproprepcartmaps}
\begin{gathered}
\xymatrix{
\mathrm{Un}(f)\downarrow v\ar@{->}[rr]\ar@{->>}[d]\ar@{}[drr]|(.3){\pbs} & & \mathrm{Un}(Y)_{/v}\ar@{->>}[d]^s \\
\mathrm{Un}(X)\ar@{->}[rr]^{\mathrm{Un}(f)}\ar@{->>}@/_/[dr]_p & & \mathrm{Un}(Y)\ar@{->>}@/^/[dl]^q \\
 & \mathcal{C} & 
}
\end{gathered}
\end{align}
has a terminal object (\cite[Proposition 3.1.2]{nrsadjoint} and \cite[Proposition 6.1.11]{cisinskibook}).
The projection $s\colon\mathrm{Un}(Y)_{/v}\twoheadrightarrow\mathrm{Un}(Y)$ is a right fibration by
\cite[Corollary 2.1.2.2]{luriehtt}, and hence so is the composite
\[\mathrm{Un}(f)\downarrow v\twoheadrightarrow\mathrm{Un}(X)\twoheadrightarrow\mathcal{C}.\]
Thus, by Lemma~\ref{lemmareprightfibs}, the $\infty$-category $\mathrm{Un}(f)\downarrow v$ has a terminal object if and only if this 
composite right fibration is representable. We hence want to show that $f$ is a representable natural transformation if and only if 
the right fibration $\mathrm{Un}(f)\downarrow v\twoheadrightarrow\mathcal{C}$ is representable for all $v\in Y$.

For every $v\in\mathrm{Un}(Y)$, the cartesian square (\ref{diagproprepcartmaps}) is homotopy-cartesian in the 
contravariant model structure on $\mathbf{S}_{/\mathcal{C}}$ essentially by an application of \cite[Proposition 4.1.2.18]{luriehtt}.
Furthermore, the right fibration $\mathrm{Un}(Y)\twoheadrightarrow\mathcal{C}$ induces a trivial Kan fibration
\begin{align}\label{diagproprepcartmaps2}
\begin{gathered}
\xymatrix{
\mathrm{Un}(Y)_{/v}\ar@{->>}[r]^{\sim}\ar@{->>}[d]_s & \mathcal{C}_{/q(v)}\ar@{->>}[d]^s\ar@{-->}[dl]^{\ulcorner v\urcorner} \\
\mathrm{Un}(Y)\ar@{->>}[r] & \mathcal{C}
}
\end{gathered}
\end{align}
of slice $\infty$-categories such that the outer square commutes. A dotted natural transformation
$\ulcorner v\urcorner\colon\mathcal{C}_{/q(v)}\rightarrow\mathrm{Un}(Y)$ over $\mathcal{C}$ as depicted in (\ref{diagproprepcartmaps2}) such 
that the top triangle commutes up to homotopy is hence given by any section of the trivial fibration on the top.
As straightening over $\mathcal{C}$ is an equivalence of $\infty$-categories and hence preserves pullback squares, it follows from 
(\ref{diagproprepcartmaps}) and (\ref{diagproprepcartmaps2}) that the induced square
\[
\xymatrix{
\mathrm{St}(\mathrm{Un}(f)\downarrow v)\ar[r]\ar[d] & y(q(v))\ar[d]^{\mathrm{St}(\ulcorner v\urcorner)} \\
X\ar[r]^{f} & Y
}\]
is cartesian in $\hat{\mathcal{C}}$.

Thus, whenever $f$ is a representable natural transformation, it follows that for all $v\in\mathrm{Un}(Y)$ the right fibration
$\mathrm{Un}(f)\downarrow v\twoheadrightarrow\mathcal{C}$ is representable, and so the functor
$\mathrm{Un}(f)\colon\mathrm{Un}(X)\rightarrow\mathrm{Un}(Y)$ has a right adjoint as stated above. Vice versa, suppose 
$\mathrm{Un}(f)$ has a right adjoint, and let $\ulcorner v\urcorner\colon y(C)\rightarrow Y$ be some natural transformation in
$\hat{\mathcal{C}}$ for $C\in\mathcal{C}$. We obtain a functor
$\mathrm{Un}(\ulcorner v\urcorner)\colon\mathcal{C}_{/C}\rightarrow\mathrm{Un}(Y)$ of right fibrations 
over $\mathcal{C}$ and hence a vertex $v:=\mathrm{Un}(\ulcorner v\urcorner)(\mathrm{id}_C)$ in the fiber over $C$ together with a trivial 
Kan fibration as in (\ref{diagproprepcartmaps2}). It follows that
$\mathrm{St}(\mathrm{Un}(f)\downarrow v)\simeq(\ulcorner v\urcorner)^{\ast}X$ in $\hat{\mathcal{C}}$. But the right fibration
$\mathrm{Un}(f)\downarrow v\twoheadrightarrow\mathcal{C}$ is representable since $\mathrm{Un}(f)$ is assumed to have a right adjoint, and 
hence so is the pullback $(\ulcorner v\urcorner)^{\ast}X$ in $\hat{\mathcal{C}}$.
\end{proof}

\begin{remark}\label{remclassobj}
To elaborate on the interpretation of representable natural transformations as general comprehension schemes in Definition~\ref{defgencomp},
let $\mathcal{C}$ be an $\infty$-category and $f\colon X\rightarrow Y$ be a representable natural transformation in $\hat{\mathcal{C}}$. 
Given objects $C\in\mathcal{C}$ and $v\in Y(C)$ represented by some $\ulcorner v\urcorner\colon y(C)\rightarrow Y$, we obtain a pullback 
square of the form
\[\xymatrix{
y(C.(f,v))\ar[r]^(.6){\ulcorner x.v \urcorner}\ar[d]_{\varepsilon_{(f,v)}}\ar@{}[dr]|(.3){\pbs} & X\ar[d]^f \\
y(C)\ar[r]_{\ulcorner v\urcorner} & Y
}\]
in $\hat{\mathcal{C}}$. Following the proof of Lemma~\ref{proprepcartmaps}, the object $x.v\in X(C.(f,v))$ together with the morphism
$\varepsilon_{(f,v)}\colon C.(f,v)\rightarrow C$ gives rise to a terminal object in the comma $\infty$-category
$\mathrm{Un}(f)\downarrow v$. Algebraically, this pair can be understood as the universal partial $X$-structure on $C$ which lifts the
$Y$-structure $v$ on $C$ along $f$. That is to say that whenever some $X$-structure $w\in X(D)$ on some 
object $D\in\mathcal{C}$ is such that the projection $f(w)\in Y(D)$ is (equivalent to) the restriction $g^{\ast}v$ for some morphism
$g\colon D\rightarrow C$, then it already is the restriction of the $X$-structure $x.v\in X(C.(f,v))$ along some essentially uniquely 
assigned morphism $D\rightarrow C.(f,v)$ in $\mathcal{C}_{/C}$. In type theoretic terms, the morphism
$\varepsilon_{(f,v)}\colon C.(f,v)\rightarrow C$ can be understood as the extension of the context $C$ by the type $v$ in as much as it is 
the universal context morphism along which the reindexing of the type $v\in Y(C)$ has a term $x.v\in X$ (given syntactically by the 
associated Variable Introduction rule $c:C, x:v\vdash x:v$).
\end{remark}

\begin{lemma}\label{lemmarepnattransfbasics}
Let $\mathcal{C}$ be an $\infty$-category.
\begin{enumerate}
\item The class of representable natural transformations in $\hat{\mathcal{C}}$ is closed under composition, is stable under pullbacks and contains all equivalences.
\item Whenever $\mathcal{C}$ has pullbacks, the class of representable natural transformations in $\hat{\mathcal{C}}$ (considered as a full
$\infty$-subcategory of $\mathrm{Fun}(\Delta^1,\hat{\mathcal{C}})$) furthermore is closed under finite limits. In particular, it has the left 
cancellation property.
\item Suppose $\mathcal{D}$ is an $\infty$-category with pullbacks and $F\colon\mathcal{D}\rightarrow\mathcal{C}$ is a functor with a 
right adjoint. Then the restriction functor $F^{\ast}\colon\hat{\mathcal{C}}\rightarrow\hat{\mathcal{D}}$ preserves representability of 
natural transformations.
\end{enumerate}
\end{lemma}
\begin{proof}
Part 1 is straightforward.
For Part 2, given a presheaf $X$ and a span of representable natural transformations 
$f_1\rightarrow f_3\leftarrow f_2$ in $\hat{\mathcal{C}}_{/X}$, the pullback of the limit $f_1\times_{f_3}f_2$ in
$\hat{\mathcal{C}}_{/X}$ along an element $x\colon y(C)\rightarrow X$ is represented by the according fiber product of the representatives in
$\mathcal{C}_{/C}$ of the elements $x^{\ast}f_i\in\hat{\mathcal{C}}_{/y(C)}$. Since equivalences are representable natural transformations, 
it follows that the class of representable natural transformations is fiberwise closed under finite limits.  As it is stable under pullbacks 
by Part 1 as well, it thus is closed under all finite limits (as in \cite[Proposition 3.3]{binimkelly}). The left cancellation property 
follows formally, see e.g.\ \cite[Lemma 3.1.15]{abfjsheavesI}. 

For Part 3, we first note that $F\colon\mathcal{D}\rightarrow\mathcal{C}$ has a right adjoint if and only if the restriction functor
$F^{\ast}\colon\hat{\mathcal{C}}\rightarrow\hat{\mathcal{D}}$ preserves representables. This is essentially
\cite[Proposition 3.1.2]{nrsadjoint} together with Lemma~\ref{lemmareprightfibs} and Example~\ref{explscartfibs}.3. Now suppose $\mathcal{D}$ 
has pullbacks and $g\colon X\rightarrow Y$ is a representable natural transformation in $\hat{\mathcal{C}}$. Let $D\in\mathcal{D}$ and 
$f\colon y(D)\rightarrow F^{\ast}Y$ be a natural transformation in $\hat{\mathcal{D}}$. 
By the Yoneda lemma we obtain a natural transformation $\bar{f}\colon y(FD)\rightarrow Y$ together with a diagram 
\begin{align}\label{diaglemmarepnattransfbasics}
\begin{gathered}
\xymatrix{
\bullet\ar[dd]\ar[rr]\ar@{}[ddrr]_(.2){\pbs}\ar[dr] & & F^{\ast}X\ar[dd]^{F^{\ast}g} \\
 & \bullet\ar[ur]\ar[dd]\ar@{}[dr]|(.3){\pbs} & \\
y(D)\ar[dr]_{\ulcorner 1_{FD}\urcorner}\ar[rr]^(.3)f|(.485)\hole & & F^{\ast}Y \\
 & F^{\ast}y(FD)\ar[ur]_{F^{\ast}\bar{f}} & \\
}
\end{gathered}
\end{align}
in $\hat{\mathcal{D}}$, where the bullet points denote the obvious fiber products. The factorization of $f$ on the bottom always exists, but 
one may note that it also follows directly from the existence of the left adjoint $F_!\colon\hat{\mathcal{D}}\rightarrow\hat{\mathcal{C}}$ of 
$F^{\ast}$ whenever $\mathcal{D}$ is small (as $F_!$ automatically preserves representables by construction). As the natural transformation 
$g$ is representable, the pullback $\bar{f}^{\ast}X$ is represented by some object $C\in\mathcal{C}$. 
And as $F$ has a right adjoint, both restrictions $F^{\ast}y(FD)$ and $F^{\ast}y(C)$ are representable over $\mathcal{D}$. Thus, the top left 
corner in Diagram~(\ref{diaglemmarepnattransfbasics}) is a fiber product of representable presheaves and hence is representable itself 
whenever $\mathcal{D}$ has pullbacks.
\end{proof}

This subsumes the basic notions of indexed and fibered $\infty$-category theory necessary for the subsequent sections to 
follow. Further constructions, e.g.\ concerning limits and colimits in this context can be found in
\cite[Section 4.3.1]{luriehtt} and in more detail in \cite[Section 6.1.1]{luriehtt} in the case $\mathcal{C}$ is presentable. Other 
constructions have been studied in \cite{barwicketalparahct}.	

We end this section with a short recap of the category $\mathbf{S}^+$ of marked simplicial sets and the cartesian model 
structure as defined in \cite[Section 3.1]{luriehtt}.
A marked simplicial set is a pair $X=(I,A)$ where $I$ is a simplicial set and $A\subseteq I_1$ is a subset which 
contains all degenerate edges. We say that $A$ is the set of marked edges in $X$. A map $f\colon (I,A)\rightarrow (J,B)$ 
of marked simplicial sets is a functor $f\colon I\rightarrow J$ of simplicial sets such that $f[A]\subseteq B$.
For every simplicial set $I$ we obtain two cases of special interest. One, the minimally marked simplicial set 
$I^{\flat}:=(I,s_0[I_0])$ and, two, the maximally marked simplicial set $I^{\sharp}:=(I,I)$.
The category $\mathbf{S}^{+}$ of marked simplicial sets is cartesian closed, the underlying simplicial set of its 
internal hom-objects $Y^X$ is denoted by $\mathrm{Map}^{\flat}(X,Y)$. The simplicial subset
$\mathrm{Map}^{\sharp}(X,Y)\subseteq\mathrm{Map}^{\flat}(X,Y)$ consists of the simplices whose 1-boundaries are 
all marked in $Y^X$. These two simplicial sets are characterized by the formulas
\begin{align}
\label{flatenrichment} \mathbf{S}(I,\mathrm{Map}^{\flat}(X,Y)) & \cong \mathbf{S}^{+}(I^{\flat}\times X,Y), \\
\label{sharpenrichment} \mathbf{S}(I,\mathrm{Map}^{\sharp}(X,Y)) & \cong \mathbf{S}^{+}(I^{\sharp}\times X,Y).
\end{align}
Whenever $p\colon\mathcal{E}\twoheadrightarrow\mathcal{C}$ is a cartesian fibration,
the object $\mathcal{E}^{\natural}\in\mathbf{S}^+$ denotes the marked simplicial set
$(\mathcal{E},\{p\text{-cartesian edges}\})$. Thus, $\mathcal{E}^{\times}$ is the largest subsimplicial set $S\subseteq\mathcal{E}$ 
such that $p\mid_S\colon S\twoheadrightarrow\mathcal{C}$ is a subfibration of $p$ and $S^{\natural}=S^{\sharp}$. It follows that a 
cartesian fibration $p\colon\mathcal{E}\twoheadrightarrow \mathcal{C}$ is a right fibration if and only if
$\mathcal{E}^{\natural}=\mathcal{E}^{\sharp}$.

As a special case of \cite[Section 2.1.4, Section 3.1.3]{luriehtt}, given an $\infty$-category $\mathcal{C}$, Lurie constructs two model 
structures on the slice category $\mathbf{S}^{+}_{/\mathcal{C}^{\sharp}}$. These are the contravariant model structure
$\mathrm{RFib}(\mathcal{C})$, whose fibrant objects are exactly maps of the form
$p\colon\mathcal{E}^{\sharp}\twoheadrightarrow\mathcal{C}^{\sharp}$ such that $p$ is a right fibration of underlying simplicial sets, and the 
cartesian model structure $\mathrm{Cart}(\mathcal{C})$, whose fibrant objects are exactly of the form
$p\colon\mathcal{E}^{\natural}\twoheadrightarrow\mathcal{C}^{\sharp}$ such that $p$ is a cartesian fibration on underlying simplicial sets.
The induced simplicial enrichment of the slice-category $\mathbf{S}^+_{/\mathcal{C}^{\sharp}}$ via the ``sharp'' hom-object construction 
(\ref{sharpenrichment}) endows both model categories with a simplicial enrichment over the model category $(\mathbf{S},\mathrm{Kan})$ for Kan 
complexes. The induced simplicial enrichment of the slice-category $\mathbf{S}^+_{/\mathcal{C}^{\sharp}}$ via the ``flat'' hom-object
construction (\ref{flatenrichment}) endows both model categories with a simplicial enrichment over the model category
$(\mathbf{S},\mathrm{QCat})$ for quasi-categories.

\section{Diagrammatic $(\infty,1)$-comprehension schemes}\label{seccomp}

In this section we generalize the notion of comprehension for fibered categories defined in \cite[Section B1.3]{elephant} to notions of 
comprehension for indexed $\infty$-categories, and study the most fundamental examples. We will do so first in terms of indexed $\infty$-categories over the base $\mathcal{C}$ rather than in terms of fibered
$\infty$-categories over $\mathcal{C}$, because some of the structural arguments are more straightforward in the indexed 
language. Additionally, this allows us to define comprehension schemes with 
respect to a further parameter in form of an abstract $\infty$-category $K$ which determines the sort of diagrammatic shapes we want to 
consider comprehension for. Non-standard choices of such $\infty$-categories $K$ will capture instances of ``generalized'' comprehension 
schemes such as well-poweredness or structured comprehension schemes which for example can capture the interpretation 
of (essentially) algebraic theories.
As concrete examples of indexed $\infty$-categories often have to be constructed in terms of their fibrational counterparts however, we will 
give an equivalent characterization of comprehension schemes for cartesian fibrations in Section~\ref{secsubcartfib} and relate it to 
Johnstone's original notion in Remark~\ref{remGcomp}.

\begin{notation}
We fix an $\infty$-category $\mathcal{C}$ for the entire section. It can be assumed to be large (and locally small) unless explicitly stated 
to be small otherwise. We recall that the latter means that $\mathcal{C}$ is an object of the large $\infty$-category
$\mathrm{Cat}_{\infty}$. 
\end{notation}

\begin{remark}
As long as a given (large) $\infty$-category $\mathcal{C}$ can be assumed to be an object of some $\infty$-category $\mathrm{CAT}_{\infty}$ 
of (large) $\infty$-categories, all results apply to such (large) $\infty$-categories $\mathcal{C}$ together with $\mathrm{Cat}_{\infty}$ 
replaced by $\mathrm{CAT}_{\infty}$ accordingly as well. This is just the common kind of universe polymorphism also referred to in 
Notation~\ref{notationglobal} and usually left implicit.
\end{remark}

\begin{definition}\label{defrectind}
Let $K$ be an $\infty$-category, and let $\mathcal{E}\colon\mathcal{C}^{op}\rightarrow K$ be a functor. For an object $k\in K$, let
$K(k,\mathcal{E})$ be the presheaf defined by the composition
\begin{align}\label{equdefrectind}
\mathcal{C}^{op}\xrightarrow{\mathcal{E}} K\xrightarrow{K(k,\phv)}\mathcal{S}.
\end{align}
\end{definition}

Given a morphism $G\colon k\rightarrow l$ in $K$, the restriction
$G^{\ast}\colon K(l,\phv)\rightarrow K(k,\phv)$ in $\mathrm{Fun}(K,\mathcal{S})$ induces for every functor
$\mathcal{E}\colon\mathcal{C}^{op}\rightarrow K$ a natural transformation 
$G^{\ast}\colon K(l,\mathcal{E})\rightarrow K(k,\mathcal{E})$ in $\hat{\mathcal{C}}$ by precomposition with $\mathcal{E}$. 

\begin{definition}\label{defGcomp}
Let $K$ be an $\infty$-category and $G\colon k\rightarrow l$ be a morphism in $K$. Say a functor
$\mathcal{E}\colon\mathcal{C}^{op}\rightarrow K$ has \emph{$(K,G)$-comprehension} if the 
natural transformation
\[G^{\ast}\colon K(l,\mathcal{E})\rightarrow K(k,\mathcal{E})\]
in $\hat{\mathcal{C}}$ is representable.
\end{definition}

\begin{lemma}[Homotopy invariance I]
Let $K$ be an $\infty$-category and $\mathcal{E}\colon\mathcal{C}^{op}\rightarrow K$ be a functor.
Suppose $G\colon k\rightarrow l$ and $G\sprime\colon k\sprime\rightarrow l\sprime$ are
equivalent morphisms in $K$. Then $\mathcal{E}$ has $(K,G)$-comprehension if and only if it has $(K,G\sprime)$-comprehension.
\end{lemma}

\begin{proof}
Every set $\Sigma$ of 2-cells in $K$ which witnesses that $G$ and $G\sprime$ are equivalent morphisms in $K$ induces a set $y(\Sigma)$ of
2-cells in $\mathrm{Fun}(K,\mathcal{S})$ which witnesses that $G^{\ast}=yG$ and $(G\sprime)^{\ast}=yG\sprime$ are equivalent. Precomposition 
with $\mathcal{E}$ is functorial and so we obtain a set $\mathcal{E}^{\ast}y(\Sigma)$ of 2-cells in $\hat{\mathcal{C}}$ which witnesses that 
the natural transformations $G^{\ast}$ and $(G\sprime)^{\ast}$ are equivalent. Hence, one is representable if and only if the 
other is.
\end{proof}

Comprehension schemes defined in this generality encompass a wide variety of examples. We will give three such general examples to illustrate 
the ubiquity, but will then restrict our attention to comprehension schemes with respect to a (not necessarily full) $\infty$-subcategory 
$K\subseteq\mathrm{Cat}_{\infty}$. The standard case is formed by $K=\mathrm{Cat}_{\infty}$ itself, which will be studied in most detail.

\begin{example}\label{exmplequivcomp}
Let $K$ be an $\infty$-category and $\mathcal{E}\colon\mathcal{C}^{op}\rightarrow K$ be a functor. Then $\mathcal{E}$ has
$(K,G)$-comprehension for every equivalence $G\colon k\rightarrow l$ in $K$. Indeed, for every such $G$ the restriction
\begin{align}\label{equexmplequivcomp}
G^{\ast}\colon K(l,\mathcal{E})\rightarrow K(k,\mathcal{E})
\end{align}
is again an equivalence in $\hat{\mathcal{C}}$, and thus is representable by Lemma~\ref{lemmarepnattransfbasics}.1.
\end{example}

\begin{example}
Let $\mathcal{C}$ be an $\infty$-category with finite products and let $C\in\mathcal{C}$ be an object. Via Example~\ref{explerepslenattransfoverpt} the functor
\[C\times(\cdot)\colon\mathcal{C}^{op}\rightarrow\mathcal{C}^{op}\]
has $(\mathcal{C}^{op},\emptyset\rightarrow D)$-comprehension for some object $D\in\mathcal{C}$ if and only if the presheaf
$\mathcal{C}(C\times(\cdot),D)\in\hat{\mathcal{C}}$ is representable. That is, if and only if the exponential $D^C$ exists in $\mathcal{C}$.
One computes that this in turn implies (and hence is equivalent to) $(\mathcal{C}^{op},\iota_2\colon E\rightarrow D\sqcup E) $-comprehension for all 
objects $E\in\mathcal{C}$.
\end{example}

\begin{example}\label{explegencasela}
More generally, a functor $\mathcal{E}\colon\mathcal{C}^{op}\rightarrow K$ into an $\infty$-category $K$ with an initial object has
$(K,\emptyset\rightarrow k)$-comprehension for all objects $k$ (contained in a full $\infty$-subcategory $K\sprime\subseteq K$) if and only 
if the functor $\mathcal{E}$ has a (partial) left adjoint (with respect to the inclusion $K\sprime\subseteq K$). Accordingly, the
``syntax-to-semantics'' functors $\mathrm{pt}\colon\mathrm{Thy}^{op}\rightarrow\mathrm{Mod}$ arising in the common Stone duality type 
theorems all have $(\mathrm{Mod},\emptyset\rightarrow M)$-comprehension for all objects $M\in\mathrm{Mod}$.
\end{example}

Let $\mathrm{Cat}_{\infty}^{\mathrm{pb}}\subseteq\mathrm{Cat}_{\infty}$ denote the $\infty$-category of small
$\infty$-categories with pullbacks and pullback-preserving functors. Let
$\mathrm{Cat}_{\infty}^{\mathrm{lex}}\subseteq\mathrm{Cat}_{\infty}$ denote the $\infty$-category of small left exact 
$\infty$-categories and left exact functors. Note that for example the $\infty$-categories $\Delta^0$ and $\Delta^1$ 
are both left exact, and that the endpoint-inclusion $d^0\colon\Delta^0\rightarrow\Delta^1$ is left exact as well. A 
pullback-preserving functor $\Delta^0\rightarrow\mathcal{D}$ into an $\infty$-category $\mathcal{D}$ with pullbacks 
is determined by an object of $\mathcal{D}$. The space of left exact functors $\Delta^0\rightarrow\mathcal{D}$ into a 
(left exact) $\infty$-category $\mathcal{D}$ with terminal object $\ast$ is contractible -- with point of contraction 
the morphism $\ast\mapsto\ast$. A pullback-preserving functor $\Delta^1\rightarrow\mathcal{D}$ into an
$\infty$-category $\mathcal{D}$ (with pullbacks) is determined by a monomorphism in $\mathcal{D}$. A left exact 
functor $\Delta^1\rightarrow\mathcal{D}$ into a (left exact) $\infty$-category $\mathcal{D}$ with terminal object is 
determined by a subterminal object in $\mathcal{D}$.

\begin{example}[Well-poweredness]
Let $\mathcal{E}\colon\mathcal{C}^{op}\rightarrow\mathrm{Cat}_{\infty}^{\mathrm{pb}}$ be an indexed $\infty$-category with pullbacks. 
Say that $\mathcal{E}$ is \emph{well-powered} if it has
$(\mathrm{Cat}_{\infty}^{\mathrm{pb}},d^0\colon\Delta^0\rightarrow\Delta^1)$-comprehension.
Thus, $\mathcal{E}$ is well-powered if for every pair of objects $C\in\mathcal{C}$, $E\in\mathcal{E}(C)$ there is a morphism
$\varepsilon_E\colon P_C(E)\rightarrow C$ in $\mathcal{C}$ together with a monomorphism
$\top_E\colon\mathrm{el}_E\hookrightarrow \varepsilon_E^{\ast}(E)$ in $\mathcal{E}(P_C(E))$ which for all $f\colon D\rightarrow C$ classifies the subobjects of $f^{\ast}E\in\mathcal{E}(D)$ in the sense that the natural transformation
\[((\varepsilon_E)_{\ast},\ulcorner \top_E\urcorner)\colon y(P_C(E))\rightarrow yC\times_{\mathcal{E}^{\simeq}}\mathrm{Cat}_{\infty}^{\mathrm{pb}}(\Delta^1,\mathcal{E})\]
is an equivalence (following the blueprint in Remark~\ref{remclassobj}).
The given comprehension scheme corresponds exactly to the ``generalized comprehension scheme'' for pointwise monomorphic functors
from \cite[Example B1.3.14]{elephant} in the ordinary categorical case.
In particular, if $\mathcal{C}$ is small and has pullbacks itself, the canonical indexing
$\mathcal{C}_{/(\cdot)}\colon\mathcal{C}^{op}\rightarrow\mathrm{Cat}_{\infty}^{\mathrm{pb}}$ is well-powered if and only if all slices
$\mathcal{C}_{/C}$ have (pullback-stable) power objects. 
\end{example}

\begin{remark}
One can define well-poweredness more generally for arbitrary indexed $\infty$-categories
$\mathcal{E}\colon\mathcal{C}^{op}\rightarrow\mathrm{Cat}_{\infty}$ as $((\mathrm{Cat}_{\infty},\mathcal{M}),d_0)$-comprehension, where 
$(\mathrm{Cat}_{\infty},\mathcal{M})\subset\mathrm{Cat}_{\infty}$ denotes the wide $\infty$-subcategory spanned by the
monomorphism-preserving functors.
\end{remark}

\begin{example}[Subterminal object classifiers]\label{explesubtermobclass}
Let $\mathcal{E}\colon\mathcal{C}^{op}\rightarrow\mathrm{Cat}_{\infty}^{\mathrm{lex}}$ be an indexed left exact $\infty$-category. 
Let us consider $(\mathrm{Cat}_{\infty}^{\mathrm{lex}},d^0\colon\Delta^0\rightarrow\Delta^1)$-comprehension.
It holds if for every object $C\in\mathcal{C}$ there is a morphism $\varepsilon\colon \Omega_C\rightarrow C$ in $\mathcal{C}$ together with 
a subterminal object $\top\colon\mathrm{el}\hookrightarrow \ast$ in $\mathcal{E}(\Omega_C)$ which classifies the subterminal objects in the 
fibers of $\mathcal{E}$. In that case let's say that $\mathcal{E}$ has a subterminal object classifier.
In particular, if $\mathcal{C}$ is small and has pullbacks itself, the canonical indexing
$\mathcal{C}_{/(\cdot)}\colon\mathcal{C}^{op}\rightarrow\mathrm{Cat}_{\infty}^{\mathrm{lex}}$ has a subterminal object classifier if and only 
if all slices $\mathcal{C}_{/C}$ have (pullback-stable) subobject classifiers. 

Even in this generality, for every indexed left exact $\infty$-category
$\mathcal{E}\colon\mathcal{C}^{op}\rightarrow\mathrm{Cat}_{\infty}^{\mathrm{lex}}$, one can show that well-poweredness of the pushforward
$\mathcal{E}\colon\mathcal{C}^{op}\rightarrow\mathrm{Cat}_{\infty}^{\mathrm{pb}}$ implies the existence of subterminal object classifiers in 
$\mathcal{E}$. Indeed, under the given assumptions there is a cartesian square
\begin{align}\label{diagexplesubtermclass}
\begin{gathered}
\xymatrix{
\mathrm{Cat}_{\infty}^{\mathrm{lex}}(\Delta^1,\mathcal{E})\ar[r]\ar[d] & \mathrm{Cat}_{\infty}^{\mathrm{pb}}(\Delta^1,\mathcal{E})\ar[d]^{d_0^{\ast}} \\
\Delta^0\ar[r]_{\ulcorner\ast_{\mathcal{E}}\urcorner} & \mathcal{E}^{\simeq}
}
\end{gathered}
\end{align}
in $\hat{\mathcal{C}}$. Here, the transformation on the top is given by the natural forgetful functor. We identify the constant terminal 
presheaf $\Delta^0$ in the lower left corner with $\mathrm{Cat}_{\infty}^{\mathrm{lex}}(\Delta^0,\mathcal{E})$, and the lower right corner 
with $\mathrm{Cat}_{\infty}^{\mathrm{pb}}(\Delta^0,\mathcal{E})$. The terminal objects in the fibers $\mathcal{E}(C)$ assemble to a natural 
transformation $\ulcorner\ast_{\mathcal{E}}\urcorner\colon\Delta^0\rightarrow\mathcal{E}^{\simeq}$ precisely because the transition functors 
of $\mathcal{E}$ are assumed to preserve terminal objects (see e.g.\ \cite[Proposition 12.2.6]{riehlverityelements}). The fact that the 
square (\ref{diagexplesubtermclass}) is cartesian follows from the characterization of the upper right corner as the presheaf of 
fiberwise monomorphisms, and the upper left corner as the presheaf of fiberwise subterminal objects. Thus, representability of the 
right vertical transformation in (\ref{diagexplesubtermclass}) implies representability of the left vertical transformation by 
Lemma~\ref{lemmarepnattransfbasics}.1.
\end{example}

\begin{remark}\label{remtoposcomp}
One might expect that the notion of elementary $\infty$-topos, however defined, should be recoverable from suitable comprehension schemes. 
Such a presentation would be useful to define and study elementary $\infty$-toposes over arbitrary base $\infty$-toposes. For example, in
1-category theory, a left exact 1-category $\mathcal{C}$ is an elementary 1-topos if and only if the (large) canonical indexing
$\mathcal{C}_{/(\cdot)}$ is well-powered, see \cite[Theorem 11.1]{streicherfibcats}. But well-poweredness of the canonical indexing
$\mathcal{C}_{/(\cdot)}\colon\mathcal{C}^{op}\rightarrow\mathrm{CAT}_{\infty}^{\mathrm{lex}}$ for left exact $\infty$-categories does not 
characterize $\infty$-toposes, as such require (among others) the existence of object classifiers rather than of subobject classifiers only.
``Super-poweredness'' in form of unrestricted $(\mathrm{Cat}_{\infty},d^0\colon\Delta^0\rightarrow\Delta^1)$-comprehension is however too 
strong, since it would require each slice of $\mathcal{C}$ to have an object classifier with no size restriction whatsoever.

Whenever $\mathcal{C}$ is presentable however, for any regular cardinal $\kappa$ one may consider the comprehension scheme of
super-poweredness restricted to pointwise relatively $\kappa$-compact diagrams. To make this precise we may consider the full
$\infty$-subcategory $\mathrm{Pres}\subset\mathrm{CAT}_{\infty}$ spanned by the presentable $\infty$-categories, and
the wide $\infty$-subcategory $\mathrm{Pres}^{\rightarrow_{\kappa}}\subset\mathrm{Pres}$ spanned by those functors which preserve relatively
$\kappa$-compact morphisms. For instance, for every regular cardinal $\kappa$, the endpoint-inclusion $d_0\colon\Delta^0\rightarrow\Delta^1$ 
is a functor in $\mathrm{Pres}^{\rightarrow_{\kappa}}$, and the canonical indexing
$\mathcal{C}_{/(\cdot)}\colon\mathcal{C}^{op}\rightarrow\mathrm{CAT}_{\infty}$ factors through $\mathrm{Pres}^{\rightarrow_{\kappa}}$ as 
well. Then one can show that the canonical indexing $\mathcal{C}_{/(\cdot)}$ has
$(\mathrm{Pres}^{\rightarrow_{\kappa}},d_0)$-comprehension if and only if all slices of $\mathcal{C}$ have pullback-stable object classifiers 
for relatively $\kappa$-compact morphisms (via \cite[Theorem 6.1.6.8]{luriehtt} and \cite[Proposition 6.3.5.1.(1)]{luriehtt}). One may say in 
this case that each $\mathcal{C}_{/(\cdot)}\colon\mathcal{C}^{op}\rightarrow\mathrm{Pres}^{\rightarrow_{\kappa}}$ is ``$\kappa$-super 
powered''. We will give a more elegant characterization of $\infty$-toposes and elementary $\infty$-toposes in the sense of
\cite[Section 3]{rasekheltops} in terms of comprehension schemes in the end of Section~\ref{secextuniv}. 
\end{remark}

\begin{example}[Generic models of theories]\label{explegenmodsofthies}
Let $T$ be a Lawvere theory -- such as, say, the theory of monoids, the theory of (abelian) groups or the theory of (commutative) rings -- 
and let $\mathrm{Cat}_{\infty}^{\Pi}\subset\mathrm{Cat}_{\infty}$ be the $\infty$-category of small $\infty$-categories with finite products 
and finite product preserving functors. Suppose $\mathcal{C}$ has finite products itself. An indexed $\infty$-category
$\mathcal{E}\colon\mathcal{C}^{op}\rightarrow\mathrm{Cat}_{\infty}^{\Pi}$ has
$(\mathrm{Cat}_{\infty}^{\Pi}, \ast\rightarrow T)$-comprehension if the natural transformation
$\mathrm{Fun}^{\Pi}(T,\mathcal{E})^{\simeq}\rightarrow\ast$ is representable. By Example~\ref{explerepslenattransfoverpt} this holds if and 
only if the presheaf $\mathrm{Fun}^{\Pi}(T,\mathcal{E})^{\simeq}$ of $T$-models in $\mathcal{E}$ is representable. Informally, that means 
there is an object $M\in\mathcal{C}$ together with a $T$-model $m_{\mathrm{gen}}\colon T\rightarrow\mathcal{E}(M)$ which is a generic such model in the sense that for every $T$-model of the form $m\colon T\rightarrow\mathcal{E}(C)$ for some $C\in\mathcal{C}$ there is an essentially unique morphism 
$\ulcorner m\urcorner\colon C\rightarrow M$ in $\mathcal{C}$ such that $\mathcal{E}(\ulcorner m\urcorner)\circ m_{\mathrm{gen}}\simeq m$.
In the case $\mathcal{C}$ is small and has all finite limits, and $\mathcal{E}=\mathcal{C}_{/(\cdot)}$ is the canonical indexing for example, 
this means that there is a universal ``$T$-bundle'' in $\mathcal{C}$. 
\end{example}

Example~\ref{explegenmodsofthies} also applies to all sorts of other theories such as, say, finite limit theories together with indexed
$\infty$-categories $\mathcal{E}\colon\mathcal{C}^{op}\rightarrow\mathrm{Cat}_{\infty}^{\mathrm{lex}}$, or geometric theories as follows.

\begin{example}[Classifying toposes]
Let $\mathrm{RTop}$ be the very large $\infty$-category of $\infty$-toposes and geometric morphisms. 
In \cite[Section 7]{rs_hgst} the author defines an $\infty$-category $\mathrm{GeoCAT}_{\infty}^{\kappa}\subset\mathrm{CAT}_{\infty}$ of large
higher $\kappa$-geometric $\infty$-categories and higher $\kappa$-geometric functors for any regular cardinal $\kappa$. Its objects are the 
(large) $\infty$-categories with finite limits and universal and effective $\kappa$-small colimits. Its morphisms are the finite limit 
preserving and $\kappa$-small colimit preserving functors. The forgetful functor
\[U\colon\mathrm{RTop}^{op}\rightarrow\mathrm{GeoCAT}_{\infty}^{\kappa}\]
has $(\mathrm{GeoCAT}_{\infty},\ast\rightarrow\mathcal{C})$-comprehension for every small higher $\kappa$-geometric
$\infty$-category $\mathcal{C}$ precisely because every such $\mathcal{C}$ has 
a classifying $\infty$-topos $\mathrm{Sh}_{\mathrm{geo}_{\kappa}}(\mathcal{C})$ such that the Yoneda embedding of
$\mathcal{C}$ factors through a higher $\kappa$-geometric functor
$y\colon\mathcal{C}\rightarrow\mathrm{Sh}_{\mathrm{Geo}_{\kappa}}(\mathcal{C})$. The latter is the generic
$\mathcal{C}$-model associated to $U$.

In the ordinary categorical case, the same holds for the $(2,1)$-category $\mathrm{GeoCAT}^{\kappa}_1$ of large
$\kappa$-geometric categories, the $(2,1)$-category $\mathrm{RTop}_1$ of ordinary Grothendieck toposes, and the 
according forgetful $2$-functor
$U\colon\mathrm{RTop}_1^{op}\rightarrow\mathrm{GeoCAT}^{\kappa}_1$. The fact that $U$ has
$(\mathrm{GeoCAT}^{\kappa}_1,\ast\rightarrow \mathcal{C})$-comprehension for all small $\kappa$-geometric categories 
$\mathcal{C}$ corresponds exactly to the fact that for any given small $\kappa$-geometric theory $T$ the composite 
pseudo-functor
\begin{align}\label{equexplegenmodsofthies}
\mathrm{RTop}_1^{op}\xrightarrow{T\text{-Mod}(\cdot)}\mathrm{CAT}\xrightarrow{(\cdot)^{\simeq}}\mathrm{GRPD}
\end{align}
of $T$-models is representable \cite[Remark 2.1.5]{caramellobook}. This means that the indexed $\infty$-category
$T\text{-Mod}(\cdot)\colon\mathrm{RTop}_1^{op}\rightarrow\mathrm{CAT}$ is globally small in the sense of 
Example~\ref{exmplsmallness}. The sequence (\ref{equexplegenmodsofthies}) is part of a very classic construction in 
case of the geometric theory $T_{BG}$ of $G$-torsors for a discrete group $G$. Here, precomposition of 
(\ref{equexplegenmodsofthies}) with the functor $\mathrm{Sh}\colon\mathrm{Top}\rightarrow\mathrm{RTop}_1$ which 
assigns to a topological space its associated localic topos of sheaves 
over it is naturally equivalent to the functor $PB_G\colon\mathrm{Top}^{op}\rightarrow\mathrm{GRPD}$ which assigns to 
a space the groupoid of principle $G$-bundles over it, see e.g.\ \cite[Section VIII.2]{mlmsheaves}.
\footnote{Presumably the picture can be completed $\infty$-categorically with respect to a notion of principle
$\infty$-bundle \cite{nssbundles}.
This would mean that there is a geometric $\infty$-category $T_{BG_{\infty}}$ for every $\infty$-group $G$
such that $EG\in\mathcal{S}_{/BG}$ is the universal $T_{BG_{\infty}}$-bundle in the sense of 
Example~\ref{explegenmodsofthies}. Or in other words, that the canonical indexing
$\mathcal{S}_{/(\cdot)}\colon\mathcal{S}^{op}\rightarrow\mathrm{CAT}_{\infty}$ has
$(\mathrm{GeoCAT}^{|G|},\ast\rightarrow T_{BG_{\infty}})$-comprehension.}
\end{example}

\begin{example}[Generic models of theory extensions]
To elaborate further on Example~\ref{explegenmodsofthies}, we also may consider a functor
$F\colon T\rightarrow T\sprime$ of Lawvere theories \cite[Section 3.7]{borceuxhandbook2}. Informally, an indexed
$\infty$-category $\mathcal{E}\colon\mathcal{C}^{op}\rightarrow\mathrm{Cat}_{\infty}^{\Pi}$ has
$(\mathrm{Cat}_{\infty}^{\Pi},F)$-comprehension if for every $C\in\mathcal{C}$ and every $T$-model
$M\colon T\rightarrow\mathcal{E}(C)$ there is an object $\varepsilon_{(F,M)}\in\mathcal{C}_{/C}$ together 
with a $T\sprime$-model $M\sprime\colon T\sprime\rightarrow\mathcal{E}(\mathrm{dom}(\varepsilon_{(F,M)}))$ which is the generic $F$-extension 
of $M$ in the sense that it is the universal such model such that $\mathcal{E}(\varepsilon_{(F,M)})\circ M\simeq M\sprime\circ F$. In the 
case of the canonical indexing $\mathcal{E}=\mathcal{C}_{/(\cdot)}$ over a small $\infty$-category $\mathcal{C}$ with pullbacks, and, say, 
for the canonical functor $F\colon\mathrm{Mnd}\rightarrow\mathrm{Grp}$ from the theory of 
monoids to the theory of groups, this means that for every monoid object $M\in\mathcal{C}_{/C}$ for some $C\in\mathcal{C}$ there is a 
universal pair consisting of an object $\varepsilon_M\in\mathcal{C}_{/C}$ and a group object
$G_M\in\mathcal{C}_{/\mathrm{dom}(\varepsilon_M)}$ such that the underlying monoid of $G$ is equivalent to the restriction of $M$ along
$\varepsilon_M$. In that sense, the object $\mathrm{dom}(\varepsilon_M)\in\mathcal{C}$ classifies pullbacks of the monoid $M$ which happen to 
be groups already. For example, the scheme holds whenever $\mathcal{C}$ is an $\infty$-topos, as for a given monoid $M$ in $\mathcal{C}_{/C}$  
it is witnessed by the colimit $\varepsilon_M\in\mathcal{C}_{/C}$ of all elements
$g\colon G\rightarrow C$ for a generator $G$ such that the monoid $g^{\ast}M\in\mathcal{C}_{/G}$ is a group.
\end{example}

\begin{notation}\label{notationcompstd}
Whenever $K\subseteq\mathrm{Cat}_{\infty}$ is a full $\infty$-subcategory and $G\colon k\rightarrow l$ is a morphism in $K$, a functor
$\mathcal{E}\colon\mathcal{C}^{op}\rightarrow K$ has $(K,G)$-comprehension if and only if its postcomposition
$\mathcal{E}\colon\mathcal{C}^{op}\rightarrow\mathrm{Cat}_{\infty}$ has $(\mathrm{Cat}_{\infty},G)$-comprehension. In this case we will 
suppress the parameter $K$, and simply say that an indexed $\infty$-category
$\mathcal{E}\colon\mathcal{C}^{op}\rightarrow K\subseteq\mathrm{Cat}_{\infty}$ has $G$-comprehension whenever it has
$(\mathrm{Cat}_{\infty},G)$-comprehension. This is exactly standard diagrammatic $G$-comprehension in the sense of 
Definition~\ref{defdiagcompsheme}. 
\end{notation}

In the following we focus on the standard case of Notation~\ref{notationcompstd}. 
Recall that the $\infty$-category $\mathrm{Cat}_{\infty}$ can be described as the homotopy-coherent nerve of the full simplicial category
$\mathbf{QCat}\subset\mathbf{S}$ spanned by the quasi-categories. Every simplicial set $I$ induces a simplicial functor
\[(\cdot)^I\colon\mathbf{QCat}\rightarrow\mathbf{QCat}\]
by exponentiation. Whenever $I$ itself is a quasi-category and $\mathcal{E}$ is a $\mathcal{C}$-indexed $\infty$-category, the presheaf
$\mathrm{Cat}_{\infty}(I,\mathcal{E})\in\hat{\mathcal{C}}$ is by construction the core
$(\mathcal{E}^{I})^{\simeq}$ of the pointwise induced exponential 
\[\mathcal{E}^I\colon\mathcal{C}^{op}\xrightarrow{\mathcal{E}}\mathrm{Cat}_{\infty}\xrightarrow{(\cdot)^I}\mathrm{Cat}_{\infty}.\]
In this case, we therefore can define comprehension more generally for all maps of simplicial sets as follows. 

\begin{notation}\label{notationGcompstrict}
Let $\mathcal{E}\colon\mathcal{C}^{op}\rightarrow\mathrm{Cat}_{\infty}$ be an indexed $\infty$-category, and let $G\colon I\rightarrow J$ be 
a map of simplicial sets. Say that $\mathcal{E}$ has $G$-comprehension if the 
natural transformation
\begin{align}\label{equnotationGcompstrict}
G^{\ast}\colon (\mathcal{E}^J)^{\simeq}\rightarrow(\mathcal{E}^I)^{\simeq}
\end{align}
in $\hat{\mathcal{C}}$ is representable.
\end{notation}

\begin{lemma}[Homotopy Invariance II]\label{lemmaGcompstrict}
Let $\mathcal{E}\colon\mathcal{C}^{op}\rightarrow\mathrm{Cat}_{\infty}$ be an indexed $\infty$-category.
If $G\colon I\rightarrow J$ and $G\sprime\colon I\sprime\rightarrow J\sprime$ in $\mathbf{S}$ are weakly equivalent maps in the Joyal 
model structure, then an indexed $\infty$-category $\mathcal{E}\colon\mathcal{C}^{op}\rightarrow\mathrm{Cat}_{\infty}$ has $G$-comprehension 
if and only if it has $G\sprime$-comprehension.
In particular, if $\mathbb{R}G\colon\mathbb{R}I\rightarrow\mathbb{R}J$ denotes a fibrant replacement of $G\colon I\rightarrow J$ in the Joyal 
model structure, then $\mathcal{E}$ has $G$-comprehension in the sense of Notation~\ref{notationGcompstrict} if and 
only if it has $\mathbb{R}G$-comprehension in the sense of Definition~\ref{defGcomp}.
\end{lemma}
\begin{proof}
Suppose we are given a commutative square
\[\xymatrix{
I\ar[r]^G\ar[d] & J\ar[d] \\
I\sprime\ar[r]_{G\sprime} & J\sprime
}\]
in $\mathbf{S}$ such that the two vertical maps are weak equivalences in the Joyal model structure. This induces a square
\[\xymatrix{
(\mathcal{E}^{J\sprime})^{\simeq}\ar[r]\ar[d]_{(G\sprime)^{\ast}}& (\mathcal{E}^{J})^{\simeq}\ar[d]^{G^{\ast}}  \\
(\mathcal{E}^{I\sprime})^{\simeq}\ar[r] & (\mathcal{E}^{I})^{\simeq}
}\]
in $\hat{\mathcal{C}}$ where the horizontal natural transformations are equivalences (given that the Joyal model structure is cartesian 
closed). It follows that $G^{\ast}$ is representable if and only if $(G\sprime)^{\ast}$ is representable, and so the two corresponding 
instances of comprehension are equivalent.
\end{proof}

Therefore, in the following we will not further distinguish whether we consider a given functor $G\colon I\rightarrow J$ of
$\infty$-categories as a morphism in $\mathrm{Cat}_{\infty}$ or as a map of underlying simplicial sets.

\begin{example}[Global smallness]\label{exmplsmallness}
Let $\mathcal{E}\colon\mathcal{C}^{op}\rightarrow\mathrm{Cat}_{\infty}$ be an indexed $\infty$-category. 
Say $p$ is \emph{globally small} if it has $(\emptyset\rightarrow\Delta^0)$-comprehension, i.e.\ if the natural transformation
$\mathcal{E}^{\simeq}\rightarrow\ast$ in $\hat{\mathcal{C}}$ is representable. Whenever $\mathcal{C}$ has finite products, that is 
if and only if the core $\mathcal{E}^{\simeq}\in\hat{\mathcal{C}}$ is representable by Example~\ref{explerepslenattransfoverpt}.
\end{example}

The choice of terminology in Example~\ref{exmplsmallness} follows \cite[Section 5.2]{jacobsttbook} (and hence ultimately \cite{pavlovic_thesis}) where the term is used to denote the existence of (split) generic objects in a (split) fibration of 1-categories.

\begin{remark}\label{reminftysmall}
Let $\mathcal{E}\colon\mathcal{C}\rightarrow\mathrm{Cat}$ be a pseudo-functor over an ordinary category $\mathcal{C}$ with products. By 
definition, $\mathcal{E}$ is globally small if and only if the $\mathcal{C}$-indexed groupoid $\mathcal{E}^{\simeq}$ is representable. As 
$\mathcal{C}$ is a 1-category however, its Yoneda embedding factors through the category $\mathrm{Fun}(\mathcal{C}^{op},\mathrm{Set})$ of 
set-valued presheaves.
Thus, the existence of a natural equivalence $\mathcal{E}^{\simeq}\simeq y(C)$ for some $C\in\mathcal{C}$ implies 
that the groupoids $\mathcal{E}(D)$ for $D\in\mathcal{C}$ are all essentially discrete. This for instance implies 
that the automorphism groups of all objects $E\in\mathcal{E}(D)$ for $D\in\mathcal{C}$ are trivial. 
This latter implication is exactly the observation in \cite[Example B1.3.13]{elephant} as
$(\emptyset\rightarrow\Delta^1)$-comprehension here really is a $(2,1)$-categorical comprehension scheme in the 
sense that it requires representability of the indexed groupoid $\mathcal{E}^{\simeq}$ up to equivalence. This 
observation leads to the introduction of a variety of generic objects
\cite{jacobsttbook, sterlinggenobj} in the ordinary categorical context to replace global smallness, and so to 
capture essential examples with non-trivial automorphism groups. To be more precise, whenever $\mathcal{E}$ is a strict functor, and comprehension is defined in terms 
of the underlying-set functor $\mathrm{Cat}\rightarrow\mathrm{Set}$ as discussed in the introduction, then global smallness corresponds 
exactly to the existence of a split generic object. This notion is not equivalence-invariant however, and to remedy the fact that the
$\mathcal{C}$-indexed underlying set of a general pseudo-functor does not even exist, global smallness for a general pseudo-functor (or a 
general fibration respectively) has to be defined in other terms. Such are given for example by the existence of 
(weak) generic objects \cite{jacobsttbook, sterlinggenobj}. These however go beyond the framework of diagrammatic 
comprehension schemes. 

In that sense, global smallness is a much more natural notion in $\infty$-category theory. We will further elaborate on this in 
Remark~\ref{reminftysmall2} in the context of the externalization functor. 
Indeed, in Section~\ref{secext} we will see examples of globally small indexed $\infty$-categories with highly non-trivial vertical 
automorphism spaces in their fibers.
For instance, such an example is given by the identity $\mathcal{S}\rightarrow\mathcal{S}$ considered as an $\mathcal{S}^{op}$-indexed
$\infty$-category. It is represented by the terminal object $\ast\in\mathcal{S}$ (considered as an object of $\mathcal{S}^{op}$) and hence it 
is globally small. The automorphism space of a point $x\in X$ for a space $X\in\mathcal{S}$ is the (generally non-contractible) loop space
$\Omega_X(x)$.
\end{remark}

\begin{example}\label{exmpleintdiags}
More generally, given an $\infty$-category $\mathcal{C}$ with finite products, for any given simplicial set $I$, a $\mathcal{C}$-indexed
$\infty$-category $\mathcal{E}$ has $(\emptyset\rightarrow I)$-comprehension if and only if the presheaf
$(\mathcal{E}^I)^{\simeq}$ of fiberwise $I$-indexed diagrams in $\mathcal{E}$ is representable. For instance, the identity
$1\colon\mathrm{Cat}_{\infty}\rightarrow\mathrm{Cat}_{\infty}$ considered as a $\mathrm{Cat}_{\infty}^{op}$-indexed $\infty$-category has
($\emptyset\rightarrow I$)-comprehension in trivial fashion.
Thus, just as global smallness of $\mathcal{E}$ expresses representability of the presheaf $\mathcal{E}^{\simeq}$ of objects in
$\mathcal{E}$, for instance $(\emptyset\rightarrow\Delta^0\sqcup\Delta^0)$-comprehension expresses representability of the presheaf
$\mathcal{E}^{\simeq}\times\mathcal{E}^{\simeq}$ of pairs of objects in $\mathcal{E}$.
\end{example}

\begin{example}\label{exmplecompschemesidcat}
In fact, the identity $1\colon(\mathrm{Cat}_{\infty}^{op})^{op}\rightarrow\mathrm{Cat}_{\infty}$ has $G$-comprehension for every functor
$G\colon\mathcal{I}\rightarrow\mathcal{J}$ between small $\infty$-categories $\mathcal{I}$, $\mathcal{J}$. Indeed, for a given span
$\mathrm{Cat}_{\infty}(\mathcal{C},\phv)\rightarrow\mathrm{Cat}_{\infty}(\mathcal{I},\phv)\xleftarrow{G^{\ast}}\mathrm{Cat}_{\infty}(\mathcal{J},\phv)$ in $\widehat{\mathrm{Cat}_{\infty}^{op}}$, the pullback is represented by the pushout of the associated 
cospan $\mathcal{C}\leftarrow\mathcal{I}\xrightarrow{G}\mathcal{J}$ in $\mathrm{Cat}_{\infty}$.
\end{example}

\begin{example}[Local smallness]
Say $\mathcal{E}\colon\mathcal{C}^{op}\rightarrow\mathrm{Cat}_{\infty}$  is \emph{locally small} if it has $\delta^1$-comprehension, 
where $\delta^1\colon\partial\Delta^1\rightarrow\Delta^1$ is the boundary inclusion. That is, if for all $C\in\mathcal{C}$ and all 
$e_1,e_2\colon y(C)\rightarrow\mathcal{E}^{\simeq}$ in $\hat{\mathcal{C}}$, the pullback
\[\xymatrix{
\bullet\ar[r]\ar[d]\ar@{}[dr]|(.3){\pbs} & (\mathcal{E}^{\Delta^1})^{\simeq}\ar[d]^{(\delta^1)^{\ast}} \\
y(C)\ar[r]_(.4){(e_1,e_2)} & \mathcal{E}^{\simeq}\times\mathcal{E}^{\simeq}
}
\]
is representable.
\end{example}

Recall that an $\infty$-category $\mathcal{C}$ is defined to be locally cartesian closed if and only if for every object
$C\in\mathcal{C}$, the slice $\mathcal{C}_{/C}$ is cartesian closed. That means that for every pair
$x,y\in\mathcal{C}_{/C}$, the presheaf
\[\mathcal{C}_{/C}(x\times_{C}(\cdot),y)\colon(\mathcal{C}_{/C})^{op}\rightarrow\mathcal{S}\]
is representable. We have the following characterization of local cartesian closedness, generalizing the corresponding 
result for 1-categories in \cite[Theorem 10.2]{streicherfibcats}.

\begin{proposition}\label{exmplelocsmall}
Suppose $\mathcal{C}$ is a small $\infty$-category and has pullbacks. Then the canonical indexing
$\mathcal{C}_{/(\cdot)}\colon\mathcal{C}^{op}\rightarrow\mathrm{Cat}_{\infty}$ is locally small if and only if $\mathcal{C}$ is locally 
cartesian closed.
\end{proposition}

\begin{proof}
We show that under the equivalence $\hat{\mathcal{C}}_{/y(C)}\simeq\widehat{\mathcal{C}_{/C}}$, the pullback
\begin{align}\label{diagexmplelocsmall1}
\begin{gathered}
\xymatrix{
P(x,z)\ar[r]\ar[d]\ar@{}[dr]|(.3){\pbs} & ((\mathcal{C}_{/(\cdot)})^{\Delta^1})^{\simeq}\ar[d]^{(d_1^{\ast},d_0^{\ast})} \\
y(C)\ar[r]_(.3){(x,z)} & (\mathcal{C}_{/(\cdot)})^{\simeq}\times(\mathcal{C}_{/(\cdot)})^{\simeq}
}
\end{gathered}
\end{align}
is equivalent to the presheaf $\mathcal{C}_{/C}(x\times_{C}(\cdot),z)\colon(\mathcal{C}_{/C})^{op}\rightarrow\mathcal{S}$ for all 
$C\in\mathcal{C}$ and all morphisms $x,z\in\mathcal{C}_{/C}$. Since under this equivalence a morphism over $y(C)$ with representable domain 
corresponds exactly to a representable presheaf in $\widehat{\mathcal{C}_{/C}}$, it follows that representability of one implies 
representability of the other. In particular, the slices $\mathcal{C}_{/C}$ are cartesian closed if and only if the canonical indexing of
$\mathcal{C}$ is locally small. 

To show equivalence of the two presheaves over $\mathcal{C}_{/C}$, it suffices to show that their associated right fibrations over
$\mathcal{C}_{/C}$ are equivalent to one another. Indeed, as the canonical indexing of $\mathcal{C}$ is only indirectly defined as the 
straightening of the target fibration over $\mathcal{C}$, it is much easier to work in the cartesian model structure over $\mathcal{C}$ using 
the observations made in Section~\ref{secpre}. 
In essence, this amounts to showing that the space of morphisms
$f\colon x\times_C w\rightarrow z$ in the slice $\mathcal{C}_{/C}$ is equivalent to the space of squares considered as morphisms
$w^{\ast}x\rightarrow z$ in the $\infty$-category $\mathrm{Fun}(\Delta^1,\mathcal{C})$. Such in turn are essentially given by tuples $(w,v)$ 
in the $\infty$-category $\mathrm{Un}((P(x,z))$ by factorization through the pullback 
$w^{\ast}z$.
\[\xymatrix{
 & &  & X\ar@/_1pc/[ddr]|(.42)\hole^(.3)x & & \\
X\times_C D\ar[urrr]\ar@/_1pc/[ddr]\ar[drr]_v\ar@{-->}[rrrrr]_f & & & & & Z\ar@/^/[dl]^z \\
 & & Z\times_C D\ar[urrr]\ar@/^/[dl] & & C & \\
 & D\ar[urrr]_w & & & & \\
}\]
Hence, the statement follows from the fact that the class of squares in $\mathcal{C}$ which are vertical up to equivalence, together with
the class of cartesian squares in $\mathcal{C}$, yields a factorization system on the domain $\mathrm{Fun}(\Delta^1,\mathcal{C})$.

Now to the implementation of the proof. The unstraightening of the pullback square (\ref{diagexmplelocsmall1}) yields a pullback square of 
the form
\begin{align}\label{diagexmplelocsmall2}
\begin{gathered}
\xymatrix{
\mathrm{Un}_{\mathcal{C}}(P(x,z))\ar[r]\ar@{->>}[d]_{p(x,z)}\ar@{}[dr]|(.3){\pbs} & (\mathrm{Fun}(\Delta^1,\mathcal{C})^{\Delta^1})^{\times}\ar@{->>}[d]^{(d_1,d_0)} \\
\mathcal{C}_{/C}\ar[r]_(.3){(x,z)} & \mathrm{Fun}(\Delta^1,\mathcal{C})^{\times}\times_{\mathcal{C}}\mathrm{Fun}(\Delta^1,\mathcal{C})^{\times}
}
\end{gathered}
\end{align}
in the contravariant model category $\mathrm{RFib}(\mathcal{C})$ over $\mathcal{C}$. Here, the upper right corner is the core of the 
cartesian fibration $t^{\Delta^1}$ of vertical morphisms in $t\colon\mathrm{Fun}(\Delta^1,\mathcal{C})\twoheadrightarrow\mathcal{C}$. It is 
computed as the according simplicial cotensor in the cartesian model category $\mathrm{Cart}(\mathcal{C})$. The fact that the vertical 
fibration on the right hand side of (\ref{diagexmplelocsmall2}) is equivalent to the unstraightening of the according natural transformation  
in (\ref{diagexmplelocsmall1}) is an instance of Proposition~\ref{propstrrect}.

We hence want to show that the fibration $p(x,z)$ is equivalent to the unstraightening
$\mathrm{Un}_{\mathcal{C}_{/C}}(\mathcal{C}_{/C}(x\times_{C}(\cdot),z))\twoheadrightarrow\mathcal{C}_{/C}$.
Using that every right fibration $p\colon\mathcal{E}\twoheadrightarrow\mathcal{C}$ induces trivial fibrations
$p_{/e}\colon\mathcal{E}_{/e}\twoheadrightarrow\mathcal{C}_{/p(e)}$ on overcategories for every $e\in\mathcal{E}$, we obtain a 
sequence of categorical equivalences over $\mathcal{C}$ as follows.
\begin{align}
\notag \mathrm{Un}_{\mathcal{C}}(P(x,z)) & \simeq \left(\mathcal{C}_{/C}\times_{\mathcal{C}_{/C}}\mathcal{C}_{/C}\right)\times_{\left(\mathrm{Fun}(\Delta^1,\mathcal{C}))^{\times}\times_{\mathcal{C}}\mathrm{Fun}(\Delta^1,\mathcal{C}))^{\times}\right)}(\mathrm{Fun}(\Delta^1,\mathcal{C})^{\Delta^1})^{\times} \\
\notag  &\simeq \left(\mathrm{Fun}(\Delta^1,\mathcal{C})^{\times}_{/x}\times_{\mathcal{C}_{/C}}\mathrm{Fun}(\Delta^1,\mathcal{C})^{\times}_{/z}\right)\times_{\left(\mathrm{Fun}(\Delta^1,\mathcal{C}))^{\times}\times_{\mathcal{C}}\mathrm{Fun}(\Delta^1,\mathcal{C}))^{\times}\right)}(\mathrm{Fun}(\Delta^1,\mathcal{C})^{\Delta^1})^{\times} \\
& \simeq\mathrm{Fun}(\Delta^1,\mathcal{C})^{\times}_{/x}\times_{\left(\mathcal{C}_{/C}\times_{\mathcal{C}}(\mathrm{Fun}(\Delta^1,\mathcal{C}))^{\times}\right)}\underbrace{\left(\mathrm{Fun}(\Delta^1,\mathcal{C})^{\times}_{/z})\times_{\mathrm{Fun}(\Delta^1,\mathcal{C}))^{\times}}(\mathrm{Fun}(\Delta^1,\mathcal{C})^{\Delta^1})^{\times}\right)}_{(\ast)}\label{equexmplelocsmall1}
\end{align}
Here, the last equivalence is simply given by a permutation of components. Using Joyal's  alternative join construction and its 
according alternative overcategories (\cite[Section 4.2.1, Proposition 4.2.1.2]{luriehtt}), the right component of the fiber product 
(\ref{equexmplelocsmall1}) can be in turn expressed over the base
$\mathcal{C}_{/C}\times_{\mathcal{C}}(\mathrm{Fun}(\Delta^1,\mathcal{C}))^{\times}$ as follows.
\begin{align}\label{equexmplelocsmall2}
(\ast) & \simeq \left(\{z\}\times_{\mathrm{Fun}(\Delta^1,\mathcal{C})^{\times}}\mathrm{Fun}(\Delta^1,\mathrm{Fun}(\Delta^1,\mathcal{C})^{\times})\right)\times_{\mathrm{Fun}(\Delta^1,\mathcal{C})^{\times}}(\mathrm{Fun}(\Delta^1,\mathcal{C})^{\Delta^1})^{\times} \notag\\ 
  & \simeq \left(\{z\}\times_{\mathrm{Fun}(\Delta^1,\mathcal{C})^{\times}}\mathrm{Fun}^{\sharp}((\Delta^1)^{\sharp},\mathrm{Fun}(\Delta^1,\mathcal{C})^{\natural})\right)\times_{\mathrm{Fun}(\Delta^1,\mathcal{C})^{\times}}\mathrm{Fun}^{\sharp}_v((\Delta^1)^{\flat},\mathrm{Fun}(\Delta^1,\mathcal{C})^{\natural}) \notag\\
  & \simeq \{z\}\times_{\mathrm{Fun}(\Delta^1,\mathcal{C})^{\times}}\left(\mathrm{Fun}^{\sharp}((\Delta^1)^{\sharp},\mathrm{Fun}(\Delta^1,\mathcal{C})^{\natural})\times_{\mathrm{Fun}(\Delta^1,\mathcal{C})^{\times}}\mathrm{Fun}^{\sharp}_v((\Delta^1)^{\flat},\mathrm{Fun}(\Delta^1,\mathcal{C})^{\natural})\right) \notag\\
  & \simeq \{z\}\times_{\mathrm{Fun}(\Delta^1,\mathcal{C})^{\times}}\mathrm{Fun}^{\sharp}_{hv}((\Lambda^2_1)^{\flat},\mathrm{Fun}(\Delta^1,\mathcal{C})^{\natural}) \notag\\
  & \simeq \{z\}\times_{\mathrm{Fun}(\Delta^1,\mathcal{C})^{\times}}\mathrm{Fun}^{\sharp}_{hv}((\Delta^2)^{\flat},\mathrm{Fun}(\Delta^1,\mathcal{C})^{\natural}) \notag\\
  & \simeq \{z\}\times_{\mathrm{Fun}(\Delta^1,\mathcal{C})^{\times}}\mathrm{Fun}^{\sharp}((\Delta^1)^{\flat},\mathrm{Fun}(\Delta^1,\mathcal{C})^{\natural}) \\
  & \simeq(\mathrm{Fun}(\Delta^1,\mathcal{C})_{/z})^{\mathrm{cart.}}. \notag
\end{align}
Here, $\mathrm{Fun}^{\sharp}_v((\Delta^1)^{\flat},\mathrm{Fun}(\Delta^1,\mathcal{C})^{\natural})\subset\mathrm{Fun}^{\sharp}((\Delta^1)^{\flat},\mathrm{Fun}(\Delta^1,\mathcal{C})^{\natural})$ denotes the full $\infty$-subcategory spanned by the vertical squares,
i.e.\ those whose target morphism in $\mathcal{C}$ is an equivalence. It is the essential image of the inclusion
$(\mathrm{Fun}(\Delta^1,\mathcal{C})^{\Delta^1})^{\times}\hookrightarrow\mathrm{Fun}^{\sharp}((\Delta^1)^{\flat},\mathrm{Fun}(\Delta^1,\mathcal{C})^{\natural})$. The $\infty$-category
$\mathrm{Fun}^{\sharp}_{hv}((\Lambda^2_1)^{\flat},\mathrm{Fun}(\Delta^1,\mathcal{C})^{\natural})$ denotes the full subcategory of
$\mathrm{Fun}^{\sharp}((\Lambda^2_1)^{\flat},\mathrm{Fun}(\Delta^1,\mathcal{C})^{\natural})$ given by those spans whose first leg is 
is a vertical square and whose second leg is a cartesian square. The $\infty$-category
$\mathrm{Fun}^{\sharp}_{hv}((\Delta^2)^{\flat},\mathrm{Fun}(\Delta^1,\mathcal{C})^{\natural})$ is defined accordingly as the fiber of
$\mathrm{Fun}^{\sharp}_{hv}((\Lambda^2_1)^{\flat},\mathrm{Fun}(\Delta^1,\mathcal{C})^{\natural})$ of the obvious restriction functor along 
the inner horn inclusion $h_1^2\colon(\Lambda^2_1)^{\flat}\hookrightarrow(\Delta^2)^{\flat}$. Then the equivalence (\ref{equexmplelocsmall2}) 
follows directly from the fact that the pair of vertical and cartesian morphisms in the domain $\mathrm{Fun}(\Delta^1,\mathcal{C})$ form a 
factorization system, see specifically \cite[Example 5.2.8.15]{luriehtt} and \cite[Proposition 5.2.8.17]{luriehtt}. 
Lastly, the $\infty$-category $(\mathrm{Fun}(\Delta^1,\mathcal{C})_{/z})^{\mathrm{cart.}}:=\{z\}\times_{\mathrm{Fun}(\Delta^1,\mathcal{C})^{\times}}\mathrm{Fun}^{\sharp}((\Delta^1)^{\flat},\mathrm{Fun}(\Delta^1,\mathcal{C})^{\natural})$ denotes the wide
$\infty$-subcategory of the slice $\mathrm{Fun}(\Delta^1,\mathcal{C})_{/z}$ spanned by the cartesian squares.
Hence, via (\ref{equexmplelocsmall1}), the unstraightening $\mathrm{Un}_{\mathcal{C}}(P(x,z))$ fits into a homotopy-cartesian square of
$\infty$-categories over $\mathcal{C}$ as follows.
\[\xymatrix{
\mathrm{Un}_{\mathcal{C}}(P(x,z))\ar[r]\ar[d] & (\mathrm{Fun}(\Delta^1,\mathcal{C})_{/z})^{\mathrm{cart.}}\ar[d]^{(t_{/z},s)}\\
\mathrm{Fun}(\Delta^1,\mathcal{C})^{\times}_{/x}\ar[r]_(.42){(t_{/x},s)} & \mathcal{C}_{/C}\times_{\mathcal{C}}\mathrm{Fun}(\Delta^1,\mathcal{C})^{\times}
}\]
The base $\mathcal{C}_{/C}\times_{\mathcal{C}}\mathrm{Fun}(\Delta^1,\mathcal{C})^{\times}$ is isomorphic to the $\infty$-category
$\mathrm{Fun}_{h\mathrm{all}}(\Lambda^2_1,\mathcal{C})\times_{\mathcal{C}}\{C\}$, where
$\mathrm{Fun}_{h\mathrm{all}}(\Lambda^2_1,\mathcal{C})$ is defined as the wide $\infty$-subcategory of
$\mathrm{Fun}(\Lambda^2_1,\mathcal{C})$ given by those spans of squares whose $d^{\{0,1\}}$-face square is cartesian. The fiber
$\mathrm{Fun}_{h\mathrm{all}}(\Lambda^2_1,\mathcal{C})\times_{\mathcal{C}}\{C\}$ at $C\in\mathcal{C}$ is taken with respect to the 
restriction of $\mathcal{C}$-valued functors along the end-point inclusion $\{2\}\colon\Delta^0\hookrightarrow\Lambda^2_1$. 
We obtain a composite diagram as follows.
\[\xymatrix{
\mathrm{Un}_{\mathcal{C}}(P(x,z))\ar[r]\ar[d]\ar@/_3pc/@{->>}[dd]_{p(x,z)} & (\mathrm{Fun}(\Delta^1,\mathcal{C})_{/z})^{\mathrm{cart.}}\ar[d]\ar@{=}[r] &  (\mathrm{Fun}(\Delta^1,\mathcal{C})_{/z})^{\mathrm{cart.}}\ar@{->>}[d]\ar[r]\ar@{}[dr]|(.3){\pbs} & (\mathcal{C}_{/C})_{/z}\ar@{->>}[d]^s\\
\mathrm{Fun}(\Delta^1,\mathcal{C})^{\times}_{/x}\ar[r]\ar@{->>}[d]_{\rotatebox[origin=c]{90}{$\sim$}} & \mathrm{Fun}_{h\mathrm{all}}(\Lambda^2_1,\mathcal{C})\times_{\mathcal{C}}\{C\}\ar@/_2pc/@{-->}[r] & \mathrm{Fun}_{h\mathrm{all}}(\Delta^2,\mathcal{C})\times_{\mathcal{C}}\{C\}\ar@{->>}[l]_{\sim}^{(h^2_1)^{\ast}}\ar[r]^(.7){d_1} & \mathcal{C}_{/C} \\
\mathcal{C}_{/C}\ar@{-->}@/_2pc/[urrr]_{x\times_C(\cdot)} & & & 
}\]
Here, the $\infty$-category
$\mathrm{Fun}_{h\mathrm{all}}(\Delta^2,\mathcal{C})$ is again the wide subcategory of
$\mathrm{Fun}(\Delta^2,\mathcal{C})$ given by those prisms $\Delta^2\times\Delta^1$ whose $(d_2\times 1)$-face is a cartesian square. 
The fact that the upper right hand square yields a cartesian square in $\mathrm{Cat}_{\infty}$ can be computed 
combinatorially, using that any strict pullback of the right fibration
$s\colon(\mathcal{C}_{/C})_{/z}\twoheadrightarrow\mathcal{C}_{/C}$ represents the according homotopy-pullback.
The dotted functor  $x\times_C(\cdot)\colon\mathcal{C}_{/C}\rightarrow\mathcal{C}_{/C}$ is the essentially unique functor making the resulting triangle for any given section of the trivial fibration $(h^2_1)^{\ast}$ commute up to homotopy. Since all three squares are
homotopy-cartesian (for any given section of $(h^2_1)^{\ast}$), it follows that the composite rectangle is homotopy-cartesian. Hence, 
as the right fibration $\mathrm{Un}_{\mathcal{C}_{/C}}(\mathcal{C}_{/C}(x\times_{C}(\cdot),z))\twoheadrightarrow\mathcal{C}_{/C}$ is the homotopy-pullback 
of the representable $(\mathcal{C}_{/C})_{/z}\twoheadrightarrow \mathcal{C}_{/C}$ along the fiber-product
$x\times_C(\cdot)\colon\mathcal{C}_{/C}\rightarrow\mathcal{C}_{/C}$ as well, we obtain an equivalence
\[p(x,z)\simeq\mathrm{Un}(\mathcal{C}_{/C}(x\times_{C}(\cdot),z))\]
over $\mathcal{C}_{/C}$.
\end{proof}

\begin{example}[Smallness of representables]\label{remlocsmallrep}
Let $C\in\mathcal{C}$ be an object and consider its associated representable presheaf $y(C)\in\hat{\mathcal{C}}$. On the one hand, the indexed 
$\infty$-category $y(C)\colon\mathcal{C}^{op}\rightarrow\mathcal{S}\hookrightarrow\mathrm{Cat}_{\infty}$ satisfies
$(\partial\Delta^1\rightarrow \Delta^1)$-comprehension if and only if the natural transformation
$y(C)^{\Delta^1}\rightarrow y(C)\times y(C)$ is representable. This natural transformation however is equivalent to the diagonal
$\Delta\colon y(C)\rightarrow y(C)\times y(C)$. Thus, $y(C)$ has $(\partial\Delta^1\rightarrow \Delta^1)$-comprehension if and only if for every 
object $D\in\mathcal{C}$ and every pair of morphisms $f,g\colon D\rightarrow C$, the pullback
\[\xymatrix{
\mathrm{Eq}(y(f),y(g))\ar[d]\ar[r]\ar@{}[dr]|(.3){\pbs} & y(C)\ar[d]^{\Delta}\\
y(D)\ar[r]_(.4){(y(f),y(g))} & y(C)\times y(C)\\
}\]
is representable. This pullback computes the equalizer of the pair $y(f),y(g)\colon y(D)\rightarrow y(C)$ and hence the limit of the diagram
\[\xymatrix{y(C)\ar@<.5ex>[r]^{y(f)}\ar@<-.5ex>[r]_{y(g)} & y(D)}\]
in $\hat{\mathcal{C}}$. As the Yoneda embedding reflects limits, we see that all 
such pullbacks $\mathrm{Eq}(y(f),y(g))$ are representable if and only if $\mathcal{C}$ has all 
equalizers of pairs of morphisms with codomain $C$. Thus, all representable presheaves of $\mathcal{C}$ are locally small if and only if the
$\infty$-category $\mathcal{C}$ has all equalizers. On the other hand, suppose that $\mathcal{C}$ has a terminal object. Then $y(C)$ is 
globally small if and only if products with $C$ in $\mathcal{C}$ exist (via Example~\ref{explerepslenattransfoverpt} and the fact that the 
Yoneda embedding reflects limits). Thus, all representable presheaves of $\mathcal{C}$ are globally small if and only if the
$\infty$-category $\mathcal{C}$ has all finite products. 
In particular, assuming $\mathcal{C}$ has a terminal object, it has all finite limits if 
and only if its representable presheaves are both globally and locally small over $\mathcal{C}$. In turn, evidently, $\mathcal{C}$ has a 
terminal object if and only if the terminal presheaf $\ast\in\hat{\mathcal{C}}$ is representable.
\end{example}

\begin{example}\label{exmplecntrblty}
Consider the indexed $\infty$-category $\{\mathcal{C}\}\colon\Delta^0\rightarrow\mathrm{Cat}_{\infty}$ given by the value
$\mathcal{C}\in\mathrm{Cat}_{\infty}$. By Example~\ref{exmplsmallness}, as the $\infty$-category $\Delta^0$ has products, the indexed
$\infty$-category $\{\mathcal{C}\}$ is globally small if and only if the presheaf $\{\mathcal{C}\}^{\simeq}$ is representable, or in other 
words, if and only if the core $\mathcal{C}^{\simeq}$ is a contractible space. It is locally small if and only if all hom-spaces of
$\mathcal{C}$ are contractible.
\end{example}

The two characterizations in Example~\ref{exmplecntrblty} allow easy constructions of cartesian fibrations which are globally small but not 
locally small. For instance, we may take the free monoid on one generator and consider it as an ordinary category $\mathcal{M}$ with one 
object. Its nerve yields a cartesian fibration $N(\mathcal{M})\twoheadrightarrow\Delta^0$ which is globally small but not locally small.
Vice versa, for instance the $\infty$-category $\mathcal{S}$ of spaces is locally small over itself (i.e.\ its canonical indexing is locally 
small). Since it is a locally small $\infty$-category as well, but its slice $\infty$-categories are generally large, one can show that it 
cannot be globally small over itself. Hence, Proposition~\ref{exmplelocsmall} and Example~\ref{exmplecntrblty} together show that, in 
general, global smallness and local smallness are mutually independent properties. This mutual independence can also directly be seen via
the following example.

\begin{example}[Smallness and size]\label{exmplesize}
The $\infty$-category $\mathrm{Cat}_{\infty}$ as introduced in Section~\ref{secpre} is the $\infty$-category of 
\emph{small} $\infty$-categories.
Accordingly, $\mathrm{CAT}_{\infty}$ denotes the $\infty$-category of large $\infty$-categories
(via assumption of an inaccessible cardinal for example). Consider the fully faithful inclusion
$\iota\colon\mathrm{Cat}_{\infty}\hookrightarrow\mathrm{CAT}_{\infty}$. Given a large $\infty$-category $\mathcal{C}$, we can consider its 
associated large indexing from Example~\ref{exmplebasicindcats}.1 restricted to small $\infty$-categories given as follows.
\begin{align}\label{equsize}
\xymatrix{
\mathrm{Cat}_{\infty}\ar@{^(->}[r]^{\iota} & \mathrm{CAT}_{\infty}\ar[rr]^{\mathrm{Fun}(\phv,\mathcal{C})} & &  \mathrm{CAT}_{\infty}
}
\end{align}
The core of the indexing $\mathrm{Fun}(\phv,\mathcal{C})$ over $\mathrm{CAT}_{\infty}$ is representable as explained 
in Example~\ref{explscartfibs}.1. That means, it is a globally small $\mathrm{CAT}_{\infty}$-indexed $\infty$-category. Let us show that 
the $\infty$-category $\mathcal{C}$ is essentially small if and only if the restricted 
indexing (\ref{equsize}) is a globally small $\mathrm{Cat}_{\infty}$-indexed $\infty$-category. Therefore, recall that
$\mathcal{C}$ is essentially small if and only if there is an $\infty$-category $\mathcal{C}\sprime\in\mathrm{Cat}_{\infty}$ together 
with an equivalence $\mathcal{C}\sprime\simeq\mathcal{C}$. Now, every such equivalence between $\mathcal{C}$ and some
$\mathcal{C}\sprime\in\mathrm{Cat}_{\infty}$ induces an equivalence between the presheaf $\mathrm{CAT}_{\infty}(\iota(\cdot),\mathcal{C})$ 
and the representable presheaf
$\mathrm{Cat}_{\infty}(\phv,\mathcal{C}\sprime)\simeq\mathrm{CAT}_{\infty}(\iota(\cdot),\iota(\mathcal{C}\sprime))$. Vice versa, if
$\mathrm{CAT}_{\infty}(\iota(\cdot),\mathcal{C})$ is represented by some $\mathcal{C}\sprime\in\mathrm{Cat}_{\infty}$, the equivalence
$\mathrm{Cat}_{\infty}(\phv,\mathcal{C})\xrightarrow{\simeq}\mathrm{CAT}_{\infty}(\iota(\cdot),\mathcal{C})$ induces an equivalence
\[\mathrm{Fun}(\Delta^{\bullet},\mathcal{C}\sprime)^{\simeq}\xrightarrow{\simeq}\mathrm{Fun}(\Delta^{\bullet},\mathcal{C})^{\simeq}\]
of simplicial spaces via precomposition with the canonical inclusion $\Delta\hookrightarrow\mathrm{Cat}_{\infty}$. 
This equivalence is a functor between the complete Segal spaces associated to the quasi-categories $\mathcal{C}$ and
$\mathcal{C}\sprime$ (constructed in \cite[Section 4]{jtqcatvsss}, see Section~\ref{secext}). Thus, it in turn yields an equivalence
$\mathcal{C}\sprime\xrightarrow{\simeq}\mathcal{C}$.

Similarly, recall from \cite[Section 5.4.1]{luriehtt} that an $\infty$-category $\mathcal{C}$ is \emph{locally small} if all its 
associated hom-spaces are essentially small. For example, the large $\infty$-category $\mathrm{Cat}_{\infty}$ itself is locally small. Then 
one can show that a large $\infty$-category $\mathcal{C}$ is locally small if and only if its 
restricted indexing (\ref{equsize}) is a locally small $\mathrm{Cat}_{\infty }$-indexed $\infty$-category. In order to show 
this, one computes that the latter holds if and only if for all small $\infty$-categories $\mathcal{D}$ 
and all pairs of functors $F,G\colon\mathcal{D}\rightarrow\mathcal{C}$ the pullback
\[\xymatrix{
\mathrm{Fun}(\iota(\cdot),F\downarrow G)^{\simeq}\ar[r]\ar[d] & \mathrm{Fun}(\iota(\cdot),\mathrm{Fun}(\Delta^1,\mathcal{C}))^{\simeq}\ar[d]^{(s,t)_{\ast}}\\
\mathrm{Fun}(\phv,\mathcal{D})^{\simeq}\ar[r]_(.4){(F,G)} & \mathrm{Fun}(\iota(\cdot),\mathcal{C}\times\mathcal{C})^{\simeq}\\
}\]
of presheaves over $\mathrm{Cat}_{\infty}$ is representable. This in turn holds for a given such pair
$F,G\colon\mathcal{D}\rightarrow\mathcal{C}$ if and only the comma-$\infty$-category $F\downarrow G$ is essentially small (via the first part 
of this example). Essential smallness of these comma $\infty$-categories directly implies essential smallness of the hom-space
$\mathcal{C}(C,D)\simeq C\downarrow D$ for all pairs of objects $C,D\colon\Delta^0\rightarrow\mathcal{C}$. 
For the converse it is easy to see that local smallness of $\mathcal{C}$ and smallness of $\mathcal{D}$ implies 
essential smallness of the comma $\infty$-category $F\downarrow G$ for all functors $F,G\colon\mathcal{D}\rightarrow\mathcal{C}$.

Alternatively, one can consider the naive large indexing
\begin{align}\label{equsize2}
\xymatrix{
\mathcal{S}\ar@{^(->}[r]^{\iota} & \mathcal{S}^+\ar[rr]^{\mathrm{Fun}(\phv,\mathcal{C})} & &  \mathrm{CAT}_{\infty}
}
\end{align}
where $\mathcal{S}^+$ denotes the $\infty$-category of large $\infty$-groupoids accordingly. Then one computes that the $\mathcal{S}$-indexed 
large $\infty$-category (\ref{equsize2}) is globally small if and only if the core $\mathcal{C}^{\simeq}$ is a small space. It is a locally 
small $\mathcal{S}$-indexed large $\infty$-category again if and only if $\mathcal{C}$ is a locally small $\infty$-category.
\end{example}

In order to deduce various comprehension schemes from a given set of comprehension schemes for a fixed $\mathcal{C}$-indexed
$\infty$-category $\mathcal{E}$ (as stated to be possible in principle in \cite[Paragraph 8.5]{benaboufibfound}), we observe that the class 
of functors
\[\mathrm{Comp}(\mathcal{E}):=\{(G\colon I\rightarrow J)\in\mathbf{S}\mid \mathcal{E}\text{ satisfies }G\text{-comprehension}\}\]
satisfies the following stability properties, where $G$-comprehension for a map $G\colon I\rightarrow J$ of simplicial sets was defined in
Notation~\ref{notationGcompstrict}.

\begin{lemma}\label{lemmacompclosureprops}
Let $\mathcal{E}$ be a $\mathcal{C}$-indexed $\infty$-category.
\begin{enumerate}
\item The class $\mathrm{Comp}(\mathcal{E})$ contains all weak categorical equivalences. It is closed under composition and under
homotopy-pushouts (in the Joyal model structure). In particular, it is closed under pushouts along monomorphisms.
\item The class of monomorphisms in $\mathrm{Comp}(\mathcal{E})$ contains all trivial cofibrations in the Joyal model structure and 
is closed under compositions and arbitrary pushouts.
\item If $\mathcal{C}$ has pullbacks, the class $\mathrm{Comp}(\mathcal{E})$ is furthermore closed under finite homotopy-colimits in the 
Joyal model structure and has the right cancellation property.
\end{enumerate}
\end{lemma}
\begin{proof}
Part 1 basically follows directly from Lemma~\ref{lemmarepnattransfbasics}.1. More precisely, the class $\mathrm{Comp}(\mathcal{E})$ 
contains the weak categorical equivalences by Example~\ref{exmplequivcomp} and Lemma~\ref{lemmaGcompstrict}. Furthermore, as the class of 
representable natural transformations is stable under pullbacks, and the Yoneda embedding
$\mathrm{Cat}_{\infty}^{op}\rightarrow\mathrm{Fun}(\mathrm{Cat}_{\infty},\mathcal{S})$ preserves small limits, it follows that the class 
$\mathrm{Comp}(\mathcal{E})$ is closed under homotopy-pushouts. As the Joyal model structure is left proper, it follows that
$\mathrm{Comp}(\mathcal{E})$ is closed under pushouts along monomorphisms as well.

For Part 2 we only are left to show that arbitrary pushouts of monomorphisms in $\mathrm{Comp}(\mathcal{E})$ are contained in
$\mathrm{Comp}(\mathcal{E})$ again. Therefore, let $j\colon A\hookrightarrow B$ be a monomorphism in $\mathrm{Comp}(\mathcal{E})$ and 
consider a pushout square of the form
\[\xymatrix{
A\ar@{^(->}[d]_j\ar[r]^f\ar@{}[dr]|(.7){\pos} & C\ar[d] \\
B\ar[r] & D
}\]
in $\mathbf{S}$. We can factor the map $f\colon A\rightarrow C$ into a monomorphism $i\colon A\rightarrow C\sprime$ 
followed by a trivial isofibration $q\colon C\sprime\twoheadrightarrow C$, and consider the induced factorization of 
pushouts
\[\xymatrix{
A\ar@{^(->}[d]_j\ar@{^(->}[r]^i\ar@{}[dr]|(.7){\pos} & C\sprime\ar@{^(->}[d]\ar@{->>}[r]^q_{\sim}\ar@{}[dr]|(.7){\pos}  & C\ar@{^(->}[d] \\
B\ar[r] & D\sprime\ar[r] & D.
}\]
Then the map $C\sprime\rightarrow D\sprime$ is contained in $\mathrm{Comp}(\mathcal{E})$ by Part 2. The map
$D\sprime\rightarrow D$ is a weak categorical equivalence, again because the Joyal model structure is left proper. By
Example~\ref{exmplequivcomp}, it follows that the map $C\sprime\rightarrow D\sprime$ is contained 
in $\mathrm{Comp}(\mathcal{E})$ if and only if $C\rightarrow D$ is contained in $\mathrm{Comp}(\mathcal{E})$.
Part 3 follows directly from Lemma~\ref{lemmarepnattransfbasics}.2.
\end{proof}

\begin{example}\label{explebasicdefschemes}
Following \cite[Lemma B1.3.15]{elephant}, we say that a $\mathcal{C}$-indexed $\infty$-category  $\mathcal{E}$ satisfies \emph{definability 
of invertibility (for morphisms)} if it has $\mathrm{inv}$-comprehension, where $\mathrm{inv}\colon\Delta^1\rightarrow I\Delta^1$ is the 
embedding of $\Delta^1$ into the free groupoid generated by it.

It has \emph{definability of identity (of parallel morphisms)} if it has $\nabla_1$-comprehension, where the map
$\nabla_1\colon(\Delta^1\cup_{\partial\Delta^1}\Delta^1)\rightarrow\Delta^1$ identifies the two parallel non-degenerate 
morphisms. The domain can be described as the nerve of the parallel morphisms category
$\xymatrix{\bullet\ar@<.5ex>[r]\ar@<-.5ex>[r] & \bullet}$.

Furthermore, say $\mathcal{E}$ has \emph{definability of identity of parallel $n$-morphisms} if it has
$\nabla_n$-comprehension, where the map
$\nabla_n\colon(\Delta^n\cup_{\partial\Delta^n}\Delta^n)\rightarrow\Delta^n$ is the codiagonal identifying the two 
parallel non-degenerate $n$-morphisms.
\end{example}

\begin{remark}
Definability of identity of $n$-morphisms for $n\geq 1$ is phrased as a principle of strict identity rather than of equivalence. However, 
it can equivalently be expressed as comprehension for the canonical inclusion
$\nabla_n^{\simeq}\colon(\Delta^n\cup_{\partial\Delta^n}\Delta^n)\hookrightarrow B^n$ where $B^n$ denotes the pushout
\begin{align}\label{diagremdefofequinc}
\begin{gathered}
\xymatrix{
\partial\Delta^n\times I\Delta^1\ar@{^(->}[r]^(.55){\delta^n\times 1}\ar[d]_{\pi_1}\ar@{}[dr]|(.7){\pos} & \Delta^n\times I\Delta^1\ar[d] \\
\partial\Delta^n\ar@{^(->}[r] & B^n.
}
\end{gathered}
\end{align}	
The simplicial set $B^n$ essentially consists of two $n$-cells glued together along their boundary and a homotopy pasted in between. Since the 
space $I\Delta^1$ is contractible as a quasi-category and the Joyal model structure is left proper, it follows that the canonical map 
$B^n\rightarrow \Delta^n$ induced by the projection $\pi_1\colon\Delta^n\times I\Delta^1\rightarrow\Delta^n$ is a categorical equivalence. 
The composition
$(\Delta^n\cup_{\partial\Delta^n}\Delta^n)\hookrightarrow B^n\xrightarrow{\sim}\Delta^n$ is exactly $\nabla_n$, and thus it follows that
$\nabla_n$-comprehension is equivalent to $\nabla_n^{\simeq}$-comprehension via Lemma~\ref{lemmacompclosureprops}.1.

Note that in the case $n=1$, the pushout (\ref{diagremdefofequinc}) is in fact a pushout of (nerves of) 1-categories. However, the according 
pushout computed in $\mathrm{Cat}$ (instead of $\mathbf{S}$) is just $\Delta^1$ itself. We hence obtain no such inclusion of 1-categories 
which represents definability of identity of parallel morphisms in an indexed ($\infty$-)category.
\end{remark}

The following lemma can be thought of as an $\infty$-categorical generalization of \cite[Lemma 1.3.15]{elephant} 
extended by its natural higher analogues.

\begin{proposition}\label{lemmalocsm->definv}
Let $\mathcal{E}$  be a $\mathcal{C}$-indexed $\infty$-category and $\delta^n\colon\partial\Delta^n\hookrightarrow\Delta^n$ be the standard 
boundary inclusion for $n\geq 0$.
\begin{enumerate}
\item Let $n\geq 1$. If $\mathcal{E}$ has definable identity of parallel $n$-morphisms, it has $\delta^{n+1}$-comprehension.
\item Suppose all slice $\infty$-categories of $\mathcal{C}$ have equalizers of global sections.
\begin{enumerate}
\item Let $n\geq 1$. If $\mathcal{E}$ has $\delta^n$-comprehension, it has definable identity for parallel $n$-morphisms.
\item If $\mathcal{E}$ is locally small, it has $\delta^n$-comprehension for all $n\geq 1$.
\item If $\mathcal{E}$ is locally small, it has definability of invertibility. 
\end{enumerate}
\end{enumerate}
\end{proposition}
\begin{proof}
For Part 1.\ let $n\geq 1$ and $0<i<n$. We may consider a factorization of the $n$-th boundary inclusion given by the 
bottom left square in the following diagram.
\[\xymatrix{
\Lambda^{n+1}_i\ar@{^(->}[r]^{h^{n+1}_i}\ar@{^(->}[d]_{\iota_2}\ar@{}[dr]|(.7){\pos} & \Delta^{n+1}\ar[d]^{\iota_2} & \\
\Delta^n\cup_{\partial\Delta^n}\Lambda^{n+1}_i\ar[r]_{(\mathrm{id},h^{n+1}_i)} & \Delta^n\cup_{\partial\Delta^n}\Delta^{n+1}\ar[d]_{(d^i,\mathrm{id})} & \Delta^n\cup_{\partial\Delta^n}\Delta^n\ar@{_(->}[l]_{(\mathrm{id},d^i)}\ar[d]^{\nabla_n}\ar@{}[dl]|(.7){\rotatebox[origin=c]{-90}{$\pos$}} \\ 
\partial\Delta^{n+1}\ar[r]_{\delta^{n+1}}\ar[u]^{\rotatebox[origin=c]{90}{$\cong$}} & \Delta^{n+1} & \Delta^n\ar[l]^{d_i}
}\]
The inner horn inclusion $h^{n+1}_i\colon\Lambda^{n+1}_i\rightarrow\Delta^{n+1}$ is contained in $\mathrm{Comp}(\mathcal{E})$ by
Example~\ref{exmplequivcomp}, the codiagonal $\nabla_n\colon\Delta^n\cup_{\partial\Delta^n}\Delta^n\rightarrow\Delta^n$
is contained in the class $\mathrm{Comp}(\mathcal{E})$ by assumption. It thus follows from Lemma~\ref{lemmacompclosureprops}.2 
that the composition $\delta^{n+1}\colon\partial\Delta^{n+1}\rightarrow\Delta^{n+1}$ is contained in the class
$\mathrm{Comp}(\mathcal{E})$ as well.

For Part 2.(a) assume that $\mathcal{E}$ has $\delta^n$-comprehension. 
The natural transformation $\nabla_n^{\ast}\colon(\mathcal{E}^{\Delta^n})^{\simeq}\rightarrow(\mathcal{E}^{\Delta^n\cup_{\partial\Delta^n}\Delta^n})^{\simeq}$ is equivalent to the diagonal
\[\Delta_{(\delta^n)^{\ast}}\colon(\mathcal{E}^{\Delta^n})^{\simeq}\rightarrow(\mathcal{E}^{\Delta^n})^{\simeq}\times_{(\mathcal{E}^{\partial\Delta^n})^{\simeq}}(\mathcal{E}^{\Delta^n})^{\simeq}.\]
Let $C\in\mathcal{C}$ and consider a natural transformation
$(\ulcorner\alpha\urcorner,\ulcorner\beta\urcorner)\colon y(C)\rightarrow(\mathcal{E}^{\Delta^n})^{\simeq}\times_{(\mathcal{E}^{\partial\Delta^n})^{\simeq}}(\mathcal{E}^{\Delta^n})^{\simeq}$ which represents $n$-simplices 
$\alpha,\beta\in\mathcal{E}(C)^{\Delta^n}$ together with an equivalence $e\colon(\delta^n)^{\ast}(\alpha)\simeq(\delta^n)^{\ast}(\beta)$ 
between their respective boundaries. By assumption there is a pullback square of the form
\begin{align}\label{diaglemmalocsm->definv1}
\begin{gathered}
\xymatrix{
y\bar{C}\ar[d]_{y(\varepsilon)}\ar[r]^(.45){\ulcorner\gamma\urcorner}\ar@{}[dr]|(.3){\pbs} & (\mathcal{E}^{\Delta^n})^{\simeq}\ar[d]^{(\delta^n)^{\ast}}\\
y(C)\ar[r]_(.4){(\delta^n)^{\ast}\ulcorner\alpha\urcorner} & (\mathcal{E}^{\partial\Delta^n})^{\simeq}
}
\end{gathered}
\end{align}
in $\hat{\mathcal{C}}$. As diagonals are stable under pullback (\cite[Lemma 3.4.12]{as_soa}), the square 
\begin{align}\label{diaglemmalocsm->definv2}
\begin{gathered}
\xymatrix{
y\bar{C}\ar[d]_{\Delta_{y(\varepsilon)}}\ar[r]^(.45){\ulcorner\gamma\urcorner}\ar@{}[dr]|(.3){\pbs} & (\mathcal{E}^{\Delta^n})^{\simeq}\ar[d]^{\Delta_{(\delta)^{\ast}}}\\
y\bar{C}\times_{y(C)}y\bar{C}\ar[r]_(.4){(\ulcorner\gamma\urcorner,\ulcorner\gamma\urcorner)} & (\mathcal{E}^{\Delta^n})^{\simeq}\times_{(\mathcal{E}^{\partial\Delta^n})^{\simeq}}(\mathcal{E}^{\Delta^n})^{\simeq} 
}
\end{gathered}
\end{align}

is cartesian as well. The natural transformation
$\ulcorner\alpha\urcorner\colon y(C)\rightarrow(\mathcal{E}^{\Delta^n})^{\simeq}$ factors via some
\[y(C)\xrightarrow{y(f_{\alpha})}y\bar{C}\xrightarrow{\ulcorner\gamma\urcorner}(\mathcal{E}^{\Delta^n})^{\simeq}\]
by (\ref{diaglemmalocsm->definv1}), and by virtue of the equivalence $e\colon(\delta^n)^{\ast}(\alpha)\simeq(\delta^n)^{\ast}(\beta)$ so 
does $\ulcorner\beta\urcorner\colon y(C)\rightarrow(\mathcal{E}^{\Delta^n})^{\simeq}$ via some
\[y(C)\xrightarrow{y(f_{\beta})}y\bar{C}\xrightarrow{(\ulcorner\gamma\urcorner}(\mathcal{E}^{\Delta^n})^{\simeq}.\]
We obtain a factorization of the pair
$(\ulcorner\alpha\urcorner,\ulcorner\beta\urcorner)\colon y(C)\rightarrow(\mathcal{E}^{\Delta^n})^{\simeq}\times_{(\mathcal{E}^{\partial\Delta^n})^{\simeq}}(\mathcal{E}^{\Delta^n})^{\simeq}$
via 
\[y(C)\xrightarrow{(f_{\alpha},f_{\beta})} y\bar{C}\times_{y(C)}y\bar{C}\xrightarrow{(\ulcorner\gamma\urcorner,\ulcorner\gamma\urcorner)}(\mathcal{E}^{\Delta^n})^{\simeq}\times_{(\mathcal{E}^{\partial\Delta^n})^{\simeq}}(\mathcal{E}^{\Delta^n})^{\simeq}.\]
Thus, as (\ref{diaglemmalocsm->definv2}) is cartesian as well, the pullback
\[\xymatrix{
\bullet\ar[d]\ar[r]\ar@{}[dr]|(.3){\pbs} & (\mathcal{E}^{\Delta^n})^{\simeq}\ar[d]^{\Delta_{(\delta)^{\ast}}}\\
y(C)\ar[r]_(.25){(\ulcorner\alpha\urcorner,\ulcorner\beta\urcorner)} & (\mathcal{E}^{\Delta^n})^{\simeq}\times_{(\mathcal{E}^{\partial\Delta^n})^{\simeq}}(\mathcal{E}^{\Delta^n})^{\simeq} 
}\]
can be computed as the pullback
\[\xymatrix{
\bullet\ar[d]\ar[r]\ar@{}[dr]|(.3){\pbs} & y\bar{C}\ar[d]^{\Delta_{\varepsilon}}\\
y(C)\ar[r]_(.3){(f_{\alpha},f_{\beta})} & y\bar{C}\times_{y(C)}y\bar{C}.
}\]
This pullback however computes the equalizer of the pair of global sections $f_{\alpha},f_{\beta}\colon y(\ast)\rightarrow y(\varepsilon)$ in 
the presheaf $\infty$-category $\hat{\mathcal{C}}_{/y(C)}\simeq\widehat{\mathcal{C}_{/C}}$. Thus, this pullback is representable whenever 
the equalizers of $f_{\alpha}$ and $f_{\beta}$ exists in the slice $\mathcal{C}_{/C}$.
This finishes Part (a).

Part 2.(b) follows immediately from Parts 1.\ and 2.(a) since local smallness is $\delta^1$-comprehension.

For Part 2.(c) let $J^{(2)}\subset I\Delta^1$ be the subsimplicial set given by exactly one of the two non-degenerate
2-simplices. It can be thought of as the free left (or right) invertible edge, depicted as follows.
\[\xymatrix{
& [0]\ar[dr]^{f} & \\
[1]\ar[ur]^{f^{-1}}\ar[rr]_{s_0([1])} & & [1]
}\]
The pushout $K:=J^{(2)}\cup_{\Delta^1} J^{(2)}$ along the boundaries $d^0$ and $d^2$ is the free biinvertible map, the 
inclusion $\mathrm{inv}\colon\Delta^1\rightarrow I\Delta^1$ factors through $K$. The natural map
$K\rightarrow I\Delta^1$ is a weak categorical equivalence (both are interval objects in the Joyal model structure), and 
hence by Example~\ref{exmplequivcomp} and Lemma~\ref{lemmacompclosureprops} we can reduce 
$\mathrm{inv}$-comprehension to $(\Delta^1\rightarrow K)$-comprehension. Since the boundaries
$d^i\colon \Delta^1\rightarrow J^{(2)}$ for $i=0,2$ are monomorphisms, we can reduce
$(\Delta^1\rightarrow K)$-comprehension further to $(d^0\colon\Delta^1\rightarrow J^{(2)})$-comprehension 
again by Lemma~\ref{lemmacompclosureprops}.

Therefore, we factor the inclusion $d^0\colon\Delta^1\rightarrow J^{(2)}$ through subobjects
\[\Delta^1\rightarrow\partial J^{(2)}\rightarrow (\partial J^{(2)})_+\rightarrow J^{(2)}\]
such that each component is the pushout of a map in $\mathrm{Comp}(\mathcal{E})$ along a monomorphism. Then we can use
Lemma~\ref{lemmacompclosureprops} once more to finish the proof. First, consider the circle given by the pushout
\begin{align}\label{equlocsm->definv1}
\begin{gathered}
\xymatrix{
\partial\Delta^1\ar[r]^{\delta^1}\ar@{^(->}[d]_{\bar{\delta^1}}\ar@{}[dr]|(.7){\pos} & \Delta^1\ar[d] \\
\Delta^1\ar[r] & \partial J^{(2)}
}
\end{gathered}
\end{align}

where $\bar{\delta}^1\colon\partial\Delta^1\rightarrow\Delta^1$ swaps the endpoints. Second, consider the pushout
\begin{align}\label{equlocsm->definv2}
\begin{gathered}
\xymatrix{
\Lambda^2_1\ar[r]^{h^2_1}\ar@{^(->}[d]^{}\ar@{}[dr]|(.7){\pos}  & \Delta^2\ar[d] \\
\partial J^{(2)}\ar[r] & (\partial J^{(2)})_+
}
\end{gathered}
\end{align}
where the left vertical map picks out the two non-degenerate edges added in (\ref{equlocsm->definv1}) concatenated at
$[0]$. Third, we obtain a pushout
\[\xymatrix{
\Delta^1\cup_{\partial\Delta^1}\Delta^1\ar[r]^(.6){\nabla_1}\ar@{^(->}[d]^{}\ar@{}[dr]|(.7){\pos}  & \Delta^1\ar[d] \\
(\partial J^{(2)})_+\ar[r] & J^{(2)}
}\]
where the left vertical map picks out the pair consisting of the morphism $[1]\rightarrow [1]$ freely added in (\ref{equlocsm->definv2}) and 
the identity on $[1]$.

The composition of the bottom maps of the three pushout squares is exactly the inclusion
$d^0\colon\Delta^1\rightarrow J^{(2)}$.
The $\mathcal{C}$-indexed $\infty$-category $\mathcal{E}$ has $\delta^1$-comprehension by assumption, and thus it has
$\nabla_1$-comprehension by Part 2.(a). It has $h^2_1$-comprehension by Example~\ref{exmplequivcomp} and thus the statement follows by
Lemma~\ref{lemmacompclosureprops} as claimed.
\end{proof}

\begin{corollary}\label{cormonocomp}
Suppose $\mathcal{C}$ has pullbacks. Then for a $\mathcal{C}$-indexed $\infty$-category $\mathcal{E}$ the following are equivalent.
\begin{enumerate}
\item $\mathcal{E}$ is both globally small and locally small.
\item $\mathcal{E}$ has both $(\emptyset\rightarrow\Delta^0)$-comprehension and $(\emptyset\rightarrow\Delta^1)$-comprehension.
\item $\mathcal{E}$ has $j$-comprehension for every monomorphism $j$ between simplicial sets with finitely many non-degenerate simplices. 
\item $\mathcal{E}$ has $j$-comprehension for all finite extensions of simplicial sets, i.e\ for all monomorphisms $j\colon I\rightarrow J$ 
of simplicial sets such that the complement $J\setminus I$ (considered as an $\mathbb{N}$-indexed collection of sets) contains only finitely 
many non-degenerate simplices of $J$. 
\end{enumerate}
\end{corollary}
\begin{proof}
The monomorphisms in $\mathbf{S}$ are exactly the free cofibrations generated by the class of boundary inclusions
$\partial\Delta^n\hookrightarrow \Delta^n$ for $n\geq 0$. In particular, every monomorphism between simplicial sets that adds only finitely 
many non-degenerate simplices is the composition of pushouts of boundary inclusions. Thus, the equivalence of 1 -- 4 follows from
Lemma~\ref{lemmacompclosureprops}.3 and Proposition~\ref{lemmalocsm->definv}.
\end{proof}

\begin{remark}
Recall that a quasi-category is finite if it is weakly categorically equivalent to a simplicial set with finitely many non-degenerate 
simplices. It thus follows from Corollary~\ref{cormonocomp} that an $\infty$-category $\mathcal{E}$ indexed over an $\infty$-category
$\mathcal{C}$ with pullbacks is globally and locally small if and only if it has $(\emptyset\rightarrow\mathcal{I})$-comprehension for all 
finite quasi-categories $\mathcal{I}$.
We will see in Corollary~\ref{corsmallallcompschemes} that the finiteness condition vanishes in case the base $\mathcal{C}$ has all small 
limits. We will furthermore provide an internal characterization of the globally small and locally small indexed $\infty$-categories over 
left exact bases in Section~\ref{secext}. 
\end{remark} 

The following corollary addresses the three examples of comprehension schemes listed in \cite[Paragraph 8.5]{benaboufibfound}, there 
considered to be implausible for a general category (in a non-univalent meta-theory) to satisfy. It shows that all three schemes hold for
example for the canonical indexing over any locally cartesian closed and left exact $\infty$-category.

\begin{corollary}\label{coridprinciples8.7}
Let $\mathcal{C}$ be an $\infty$-category and $\mathcal{E}$ be a $\mathcal{C}$-indexed $\infty$-category. 
\begin{enumerate}
\item Whenever $\mathcal{E}$ is locally small, it has $(\partial\Delta^1\rightarrow\Delta^0)$-comprehension. That means identity of objects 
in $\mathcal{E}$ is definable.
\item $\mathcal{E}$ has $(\mathrm{inv}\colon\Delta^1\rightarrow I\Delta^1)$-comprehension if and only if it has
$(\Delta^1\rightarrow \Delta^0)$-comprehension. That means invertibility of morphisms in $\mathcal{E}$ is 
definable (Examples~\ref{explebasicdefschemes}) if and only if identities in $\mathcal{E}$ are definable (in all morphisms in
$\mathcal{E}$).
\item $\mathcal{E}$ has $(\nabla\colon\Delta^1\sqcup\Delta^1\rightarrow\Delta^1)$-comprehension if and only if it has 
$((\mathrm{id},\{0\}),(\mathrm{id},\{1\}))\colon\Delta^1\sqcup\Delta^1\rightarrow\Delta^1\times I\Delta^1)$-comprehension.
In either case we say that identity of morphisms in $\mathcal{E}$ is definable.
Suppose $\mathcal{C}$ has finite limits, and suppose that identity of objects in $\mathcal{E}$ is definable. Then identity of parallel 
morphisms in $\mathcal{E}$ is definable (Examples~\ref{explebasicdefschemes}) if and only if identity of morphisms in $\mathcal{E}$ is 
definable. 
\end{enumerate}
In particular, whenever $\mathcal{E}$ is locally small and $\mathcal{C}$ has all finite limits, it follows that equality of objects, 
equality of morphisms and identities in $\mathcal{E}$ are definable altogether.
\end{corollary}
\begin{proof}
For Part 1 consider the diagram
\[\xymatrix{
\partial\Delta^1\ar[r]^{\delta^1}\ar[dr] & \Delta^1\ar[r]^{\mathrm{inv}} & I\Delta^1\ar[dl]^{\simeq}\\
 & \Delta^0 & 
}\]
of simplicial sets. If $\mathcal{E}$ is locally small, the map $\delta^1$ is contained in the class $\mathrm{Comp}(\mathcal{E})$ by 
definition, and so is the map $\mathrm{inv}$ by Proposition~\ref{lemmalocsm->definv}.2.(c). Now by Lemma~\ref{lemmacompclosureprops}.1, since 
every equivalence of quasi-categories is contained in $\mathrm{Comp}(\mathcal{E})$ as well, so is the composition of Part 1.
For Part 2 we just note that $\Delta^1\rightarrow\Delta^0$ factors through
$\Delta^1\xrightarrow{\mathrm{inv}}I\Delta^1\xrightarrow{\sim}\Delta^0$.

In Part 3, the two comprehension schemes are equivalent as the two respective maps commute over the categorical equivalence
$\pi_1\colon\Delta^1\times I\Delta^1\rightarrow\Delta^1$. Whenever $\mathcal{E}$ has
$(\nabla\colon\partial\Delta^1\rightarrow\Delta^0)$-comprehension, it has
$(\nabla\sqcup\nabla\colon\partial\Delta^1\sqcup\partial\Delta^1\rightarrow\Delta^0\sqcup\Delta^0)$-comprehension by Lemma~\ref{lemmacompclosureprops}.3. It follows that the right vertical map in the pushout
\[\xymatrix{
\partial\Delta^1\sqcup\partial\Delta^1\ar[d]_{\nabla\sqcup\nabla}\ar@{^(->}[rr]^(.55){((0_1,0_2),(1_1,1_2))}\ar@{}[drr]|(.7){\pos} & & \Delta^1\sqcup\Delta^1\ar[d] \\
\Delta^0\sqcup\Delta^0\ar@{^(->}[rr] & & \Delta^1\cup_{\partial\Delta^1}\Delta^1 
}\]
is contained in the class $\mathrm{Comp}(\mathcal{E})$ as well by Lemma~\ref{lemmacompclosureprops}.1. Lastly, the triangle
\[\xymatrix{
\Delta^1\cup_{\partial\Delta^1}\Delta^1\ar[dr]_{\nabla_1}\ar[r] & \Delta^1\sqcup\Delta^1\ar[d]^{\nabla} \\
 & \Delta^1
}\] 
of simplicial sets commutes. As we just have seen that the top map is contained in the class $\mathrm{Comp}(\mathcal{E})$, it follows again 
from Lemma~\ref{lemmacompclosureprops} that $\nabla$ is contained in $\mathrm{Comp}(\mathcal{E})$ if and only if $\nabla_1$ is so.
\end{proof}

We end this section with the following useful transition result for comprehension schemes which is an $\infty$-categorical generalization of 
\cite[Proposition 1.3.17]{elephant}.

\begin{lemma}\label{lemmachangeofbasecomp}
Let $\mathcal{E}$ be a $\mathcal{C}$-indexed $\infty$-category and suppose $\mathcal{D}$ is an $\infty$-category with pullbacks. If 
$F\colon\mathcal{D}\rightarrow\mathcal{C}$ is a functor with a right adjoint, then
$\mathrm{Comp}(F^{\ast}{\mathcal{E}})\subseteq\mathrm{Comp}(\mathcal{E})$.
\end{lemma}
\begin{proof}
This follows from Lemma~\ref{lemmarepnattransfbasics}.3 by virtue of the natural equivalence
\[F^{\ast}(\mathrm{Cat}_{\infty}(\phv,\mathcal{E}))\simeq\mathrm{Cat}_{\infty}(\phv,F^{\ast}\mathcal{E})\]
simply given by associativity of pre- and postcomposition of functors.
\end{proof}

\begin{remark}
Lemma~\ref{lemmachangeofbasecomp} applies more generally to all comprehension schemes in the generality discussed in 
the introduction. In particular, it applies to all $(K,G)$-comprehension schemes for any $\infty$-category $K$ 
and any morphism $G$ in $K$. The same remark applies to Lemma~\ref{lemmacompclosureprops}.
\end{remark}

We will see applications of Lemma~\ref{lemmachangeofbasecomp} in the coming sections.

\section{Standard diagrammatic comprehension for cartesian fibrations}\label{secsubcartfib}

For various reasons it may be useful at times to express comprehension schemes for a given cartesian fibration directly without having to 
compute its straightening first. One reason to do so is that often examples of indexed $\infty$-categories are only implicitly given as the 
straightening of certain cartesian fibrations as the latter are much easier to construct in practice (take the 
canonical indexing of an $\infty$-category with pullbacks for instance as in the proof of 
Proposition~\ref{exmplelocsmall}). Another reason is that theories of fibrations exist in many different contexts and 
can easily be translated from one to another, while the theory of indexed $\infty$-categories 
does not do so to the same extent. For instance, a theory of cartesian fibrations exists in every $\infty$-cosmos as 
defined by Riehl and Verity \cite{riehlverityelements} and it is furthermore preserved along cosmological functors. 
This enables us to give a simple proof of independence of our results about comprehension from the specific choice of 
model of $(\infty,1)$-category theory, and even allows us to express the notion of comprehension in other
$\infty$-cosmoses (of not necessarily $(\infty,1)$-categories). \\

Therefore, recall the flat simplicially enriched model categories
$(\mathbf{S}^{+}_{/\mathcal{C}^{\sharp}},\mathrm{RFib}(\mathcal{C}))$ for right fibrations and
$(\mathbf{S}^{+}_{/\mathcal{C}^{\sharp}},\mathrm{Cart}(\mathcal{C}))$ for cartesian fibrations over $\mathcal{C}$ from Section~\ref{secpre}.
The associated full simplicial subcategories of fibrant objects yield $\infty$-cosmoses $\mathbf{RFib}(\mathcal{C})$ and
$\mathbf{Cart}(\mathcal{C})$ in the sense of \cite{riehlverityelements}, respectively.
The following definition introduces a pinched variation of Johnstone's rectangular diagram 
categories \cite[Section B1.3]{elephant} that he uses to define comprehension schemes of Grothendieck fibrations.

\begin{definition}\label{defrect}
Let $p\colon\mathcal{E}\twoheadrightarrow\mathcal{C}$ be a cartesian fibration. For a simplicial set $I\in\mathbf{S}$, let $p^I$ be the 
simplicial cotensor of $p$ with $I$ in the $\infty$-cosmos $\mathbf{Cart}(\mathcal{C})$. The right fibration of 
\emph{pinched rectangular $I$-indexed diagrams in $p$} is given by the core $(p^I)^{\times}\in\mathbf{RFib}(\mathcal{C})$ 
in the sense of Definition~\ref{defcore}. We will denote the domain of $(p^I)^{\times}$ by
$\llbracket I,\mathcal{E}\rrbracket$.
\end{definition}

The right fibration $(p^I)^{\times}\colon\llbracket I,\mathcal{E}\rrbracket\twoheadrightarrow\mathcal{C}$ can explicitly be described 
as the composite pullback
\begin{align}\label{diagdefrect}
\begin{gathered}
\xymatrix{
\llbracket I,\mathcal{E}\rrbracket\ar@{^(->}[r]\ar@{^(->}[d]\ar@{}[dr]|(.3){\pbs} & (\mathcal{E}^I)^{\times}\ar@{^(->}[d] \\
\Delta^{\ast}(\mathcal{E}^I)\ar@{->>}[d]_{p^I}\ar@{^(->}[r]\ar@{}[dr]|(.3){\pbs} & \mathcal{E}^I\ar@{->>}[d]^(.4){p^I}\\
\mathcal{C}\ar@{^(->}[r]_{\Delta}& \mathcal{C}^I
}
\end{gathered}
\end{align}
in the simplicial category $\mathbf{QCat}$ of quasi-categories. Here, the exponential on the right hand side is computed in $\mathbf{QCat}$ 
and the exponential on the left hand side is computed in $\mathbf{Cart}(\mathcal{C})$. Indeed, the right vertical map $p^I\colon\mathrm{Map}^{\flat}(I^{\flat},\mathcal{E}^{\natural})\twoheadrightarrow\mathrm{Map}^{\flat}(I^{\flat},\mathcal{C}^{\sharp})$ 
is a cartesian fibration by \cite[Remark 3.1.1.10, Proposition 3.1.2.3]{luriehtt}, and hence so is its pullback
$p^I\colon\Delta^{\ast}(\mathcal{E}^I)\twoheadrightarrow\mathcal{C}$. Its cartesian morphisms are exactly the 
morphisms in the wide $\infty$-subcategory $(\mathcal{E}^I)^{\times}=\mathrm{Map}^{\sharp}(I^{\flat},\mathcal{E}^{\natural})$ by
\cite[Proposition 3.1.2.1]{luriehtt}. It follows that the vertical composite map $(\mathcal{E}^I)^{\times}\twoheadrightarrow\mathcal{C}^I$ in 
(\ref{diagdefrect}) is a right fibration, and so the pullback $(p^I)^{\times}\colon\llbracket I,\mathcal{E}\rrbracket\rightarrow\mathcal{C}$ 
is a right fibration as well. In particular, $\llbracket I,\mathcal{E}\rrbracket=\Delta^{\ast}(\mathcal{E}^I)^{\times}$ is an
$\infty$-category. Both pullbacks in (\ref{diagdefrect}) are homotopy pullbacks in the 
Joyal model structure for quasi-categories (because, for instance, the pullback of a span in which all vertices are fibrant and one leg is a 
fibration yields a homotopy pullback in any model category). 

In less formal terms, $\llbracket I,\mathcal{E}\rrbracket\subseteq\text{Fun}(I,\mathcal{E})$ is the $\infty$-category 
of vertical diagrams and jointly horizontal natural transformations over
$p\colon\mathcal{E}\twoheadrightarrow\mathcal{C}$. Its objects are the vertical diagrams contained in the
$\infty$-category $\coprod_{C\in\mathcal{C}}\text{Fun}(I,\mathcal{E}(C))$, its morphisms between two 
vertical diagrams $X\colon I\rightarrow\mathcal{E}(C)$ and $Y\colon I\rightarrow \mathcal{E}(C\sprime)$ are the 
natural transformations $\alpha\colon X\rightarrow Y$ in $\mathcal{E}$ which are pointwise cartesian over one single morphism
$f_{\alpha}\colon C\rightarrow C\sprime$ in $\mathcal{C}$. By virtue of this single apex morphism in $\mathcal{C}$ 
they can be thought as pinched rectangular or triangular prismatic. 

\begin{remark}\label{remGcomp}
The reason for the departure from rectangular shaped diagrams is the correspondence in Proposition~\ref{propstrrect}.
Following \cite{elephant} on the nose, one would define $\text{Rect}(I,\mathcal{E})\subseteq(\mathcal{E}^I)^{\times}$ 
to be the full $\infty$-subcategory generated by the diagrams which are vertices in
$\mathbf{S}^{+}(I^{\sharp},(\mathcal{E},\{\text{vertical edges in }\mathcal{E}\}))$.
This notion of rectangular diagram comes with a less strict notion of both vertical diagram and horizontal natural 
transformation than pinched rectangular diagrams do. 
Yet, the two derived notions of $G$-comprehension for a map $G\colon I\rightarrow J$ of simplicial sets are equivalent whenever the 
codomain $J$ is connected and the $\infty$-category $\mathcal{C}$ has $\pi_0(I)$-sized products. This in fact holds in all of the basic 
examples which are considered both here and in \cite{elephant}. Further elaborations on this divergence will be omitted, however more details 
including a proof of this equivalence can be found in an earlier draft of this paper \cite[Remark 3.3]{rs_comp_v1}.
\end{remark}

The pinched rectangular diagram construction associated to a cartesian fibration is pullback-stable in the following sense.

\begin{lemma}\label{lemmarectpbstable}
Let $p\colon\mathcal{E}\twoheadrightarrow\mathcal{C}$ be a cartesian fibration, $I$ be a simplicial set and suppose
$F\colon\mathcal{D}\rightarrow\mathcal{C}$ is a functor of $\infty$-categories. Then the right fibrations
$((F^{\ast}p)^I)^{\times}\colon\llbracket I,F^{\ast}\mathcal{E}\rrbracket\twoheadrightarrow\mathcal{D}$ and
$F^{\ast}((p^I)^{\times})\colon F^{\ast}(\llbracket I,\mathcal{E}\rrbracket)\twoheadrightarrow\mathcal{D}$ are isomorphic.
Furthermore, for every map $G\colon I\rightarrow J$ of simplicial sets, the according functors $((F^{\ast}p)^{G})^{\times}$ and
$F^{\ast}((p^G)^{\times})$ are isomorphic as well.
\end{lemma}
\begin{proof}
The pullback functor $F^{\ast}\colon\mathbf{Cart}(\mathcal{C})\rightarrow\mathbf{Cart}(\mathcal{D})$ is cosmological and hence preserves 
simplicial cotensors naturally. It furthermore commutes with the core construction as for instance can be seen from the explicit 
description of $\mathcal{E}^{\times}\subseteq\mathcal{E}$ as the wide $\infty$-subcategory spanned by the $p$-cartesian morphisms in 
Section~\ref{secpre}. Indeed, the $F^{\ast}p$-cartesian morphisms in the pullback $F^{\ast}\mathcal{E}$ are exactly the morphisms mapped to a
$p$-cartesian morphism in $\mathcal{E}$.
\end{proof}

\begin{proposition}\label{propstrrect}
Let $p\colon\mathcal{E}\twoheadrightarrow\mathcal{C}$ be a cartesian fibration with essentially small fibers.
\begin{enumerate}
\item Let $I$ be a simplicial set. Then the right fibration
$(p^I)^{\times}\colon\llbracket I,\mathcal{E}\rrbracket\twoheadrightarrow\mathcal{C}$ is (equivalent to) the unstraightening of the presheaf
\begin{align}\label{equrectind}
\mathcal{C}^{op}\xrightarrow{\mathrm{St}(p)}\mathrm{Cat}_{\infty}\xrightarrow{(\cdot)^I}\mathrm{Cat}_{\infty}\xrightarrow{(\cdot)^{\simeq}}\mathcal{S}. 
\end{align}
\item Let $G\colon I\rightarrow J$ be a map of simplicial sets. Then the natural restriction functor
\begin{align}\label{equstrrect}
G^{\ast}\colon\llbracket J,\mathcal{E}\rrbracket\rightarrow\llbracket I,\mathcal{E}\rrbracket
\end{align}
over $\mathcal{C}$ is (equivalent to) the unstraightening of the natural transformation (\ref{equnotationGcompstrict}). In particular, the
$\mathcal{C}$-indexed $\infty$-category $\mathrm{St}(p)$ has $G$-comprehension if and only if the restriction (\ref{equstrrect}) has a (not necessarily fibered) right adjoint.
\end{enumerate}
\end{proposition}

\begin{proof}
For every $\mathcal{C}$-indexed $\infty$-category $\mathcal{E}\colon\mathcal{C}^{op}\rightarrow\mathrm{Cat}_{\infty}$ and every simplicial 
set $I$, there is a natural equivalence
\begin{align}\label{equpropstrrectproof}
\mathrm{Un}(\mathcal{E}^I)\simeq\mathrm{Un}(\mathcal{E})^I
\end{align}
of cartesian fibrations over $\mathcal{C}$. Here, the left hand side is (up to equivalence) the unstraightening of the simplicial cotensor of 
(a projectively fibrant model of) $\mathcal{E}$ with $I$ in the simplicial category
$\mathbf{Fun}(\mathfrak{C}(\mathcal{C})^{op},\mathbf{QCat})$. It is computed as the pointwise exponential of simplicial sets. The right hand 
side is the simplicial cotensor of the unstraightening $\mathrm{Un}(\mathcal{E})$ computed in $\mathbf{Cart}(\mathcal{C})$. The existence of 
the natural equivalence (\ref{equpropstrrectproof}) 
hence follows from the fact that the unstraightening construction preserves simplicial cotensors up to binatural equivalence. A direct proof 
of this can be found in \cite[Theorem 1.1]{rs_unstraight}, but it essentially follows from the fact that unstraightening induces a 
biequivalence of according $(\infty,2)$-categories as first explicitly shown in \cite[Lemma 1.4.3]{ghl2fib}. 

In particular, for every cartesian fibration $p\colon\mathcal{E}\twoheadrightarrow\mathcal{C}$ there is an equivalence
\[\mathrm{Un}(\mathrm{St}(p)^I)\simeq\mathrm{Un}(\mathrm{St}(p))^I\simeq p^I.\]
By definition of the core construction in Definition~\ref{defcore}, this induces an equivalence
\[\mathrm{Un}((\mathrm{St}(p)^I)^{\simeq})\simeq(p^I)^{\times}\]
as stated in Part 1. Part 2 follows from the same argument by virtue of binaturality of the equivalence (\ref{equpropstrrectproof}).
\end{proof}

\begin{notation}
Given Proposition~\ref{propstrrect}, we say that a cartesian fibration $p\colon\mathcal{E}\twoheadrightarrow\mathcal{C}$ has
$G$-comprehension for a map $G\colon I\rightarrow J$ of simplicial sets whenever the functor
$G^{\ast}\colon\llbracket J,\mathcal{E}\rrbracket\rightarrow\llbracket I,\mathcal{E}\rrbracket$ of $\infty$-categories has a right adjoint.
\end{notation}

For a given map $G\colon I\rightarrow J$ of simplicial sets, a cartesian fibration $p\colon\mathcal{E}\twoheadrightarrow\mathcal{C}$ and a 
vertical diagram $X\in\llbracket I,\mathcal{E}\rrbracket$, the associated comma $\infty$-category $G^{\ast}\downarrow X$ given by 
the pullback
\begin{align}\label{equrelrep1pb}
\begin{gathered}
\xymatrix{
G^{\ast}\downarrow X\ar@{->>}[d]\ar[r]\ar@{}[dr]|(.3){\pbs} & \llbracket I,\mathcal{E}\rrbracket_{/X}\ar@{->>}[d]^{s}\\
\llbracket J,\mathcal{E}\rrbracket\ar[r]_{G^{\ast}} & \llbracket I,\mathcal{E}\rrbracket
}
\end{gathered}
\end{align}
may be thought of as the right fibration of horizontal $J$-extensions of $X$. Thus, Proposition~\ref{propstrrect} allows us to characterize
$G$-comprehension directly in terms of these comma $\infty$-categories as follows.

\begin{proposition}\label{lemmabenabou}
Let $G\colon I\rightarrow J$ be a map of simplicial sets and $p\colon\mathcal{E}\twoheadrightarrow\mathcal{C}$ be a cartesian fibration.
The fibration $p$ has $G$-comprehension if and only if for every vertical diagram $X\in\llbracket I,\mathcal{E}\rrbracket$ the $\infty$-category
$G^{\ast}\downarrow X$ of vertical $J$-structures extending $X$ horizontally has a terminal object.
\end{proposition}
\begin{proof}
The right fibration $G^{\ast}\downarrow X\twoheadrightarrow\mathcal{C}$ is exactly the unstraightening of the pullback of
$G^{\ast}\colon(\mathrm{St}(p)^J)^{\simeq}\rightarrow(\mathrm{St}(p)^I)^{\simeq}$ along
$\ulcorner X\urcorner\colon y(p(X))\rightarrow(\mathrm{St}(p)^I)^{\simeq}$ in $\hat{\mathcal{C}}$.
\end{proof}

\begin{examples}\label{exmpleunivcomptransfer}
The following are translations of some of the examples in Section~\ref{seccomp} to their fibrational counterparts.
\begin{enumerate}
\item A cartesian fibration $p\colon\mathcal{E}\twoheadrightarrow\mathcal{C}$ over an $\infty$-category $\mathcal{C}$ with finite products is 
globally small if and only if its core $\mathcal{E}^{\times}$ has a terminal object (Example~\ref{exmplsmallness}).
\item The canonical fibration $t\colon\mathrm{Fun}(\Delta^1,\mathcal{C})\twoheadrightarrow\mathcal{C}$ associated to a small
$\infty$-category $\mathcal{C}$ with pullbacks is locally small if and only if $\mathcal{C}$ is locally cartesian closed 
(Proposition~\ref{exmplelocsmall}).
\item The representable right fibrations $\mathcal{C}_{/C}\twoheadrightarrow\mathcal{C}$ are locally small if and only if $\mathcal{C}$ has 
equalizers. They are always globally small whenever $\mathcal{C}$ has finite products (Example~\ref{remlocsmallrep}).
\item Any $\infty$-category $\mathcal{C}$ considered as a cartesian fibration over the point is globally small if and only if its core
$\mathcal{C}^{\times}\simeq\mathcal{C}^{\simeq}$ is contractible. It is locally small if and only if its hom-spaces are contractible 
(Example~\ref{exmplecntrblty}).
\item Given a cartesian fibration $p\colon\mathcal{E}\twoheadrightarrow\mathcal{C}$ and an $\infty$-category $\mathcal{D}$ with pullbacks 
together with a left adjoint functor $F\colon\mathcal{D}\rightarrow\mathcal{C}$, we have $\mathrm{Comp}(F^{\ast}p)\subseteq\mathrm{Comp}(p)$ 
(Lemma~\ref{lemmachangeofbasecomp}).
\item The universal cartesian fibration $\pi^{op}\colon\mathrm{Dat}_{\infty}^{op}\twoheadrightarrow\mathrm{Cat}_{\infty}^{op}$ has
$G$-comprehension for every functor $G\colon\mathcal{I}\rightarrow\mathcal{J}$ between small $\infty$-categories $\mathcal{I}$,
$\mathcal{J}$ by Example~\ref{exmplecompschemesidcat}. In particular, the left fibration
$(\pi^{I})^{\times}\colon\llbracket I,\mathrm{Dat}_{\infty}\rrbracket\twoheadrightarrow\mathrm{Cat}_{\infty}$ is equivalent to 
the corepresentable $(\mathrm{Cat}_{\infty})_{I/}\twoheadrightarrow\mathrm{Cat}_{\infty}$ for every quasi-category $I$. 
Furthermore, the universal left fibration $\mathcal{S}_{\ast}\twoheadrightarrow\mathcal{S}$ is the pullback of the universal 
cocartesian fibration $\pi\colon\mathrm{Dat}_{\infty}\twoheadrightarrow\mathrm{Cat}_{\infty}$ along the canonical 
inclusion $\mathcal{S}\hookrightarrow\mathrm{Cat}_{\infty}$ \cite[Section 3.3.2]{luriehtt}. This inclusion comes with a left
adjoint $F\colon\mathrm{Cat}_{\infty}\rightarrow\mathcal{S}$ which assigns to every $\infty$-category the free
$\infty$-groupoid generated by it. Taking opposites all over this pullback square, it follows from
Example~\ref{exmpleunivcomptransfer}.5 that the universal right fibration $\mathcal{S}_{\ast}^{op}\twoheadrightarrow\mathcal{S}^{op}$  
satisfies all comprehension schemes satisfied by the universal cartesian fibration $\pi^{op}$.
\end{enumerate}
\end{examples}

\section{Externalization of internal $(\infty,1)$-categories}\label{secext}

In this section we generalize the externalization construction of \cite{elephant} (and \cite[Section 7.2]{jacobsttbook}) to
$\infty$-category theory. Given an $\infty$-category $\mathcal{C}$ with pullbacks, we show that externalization yields an equivalence between 
internal $\infty$-categories in $\mathcal{C}$ and the globally and locally small indexed $\infty$-categories over $\mathcal{C}$.
As an example we will see that the universal cartesian fibration
$\pi^{op}\colon\mathrm{Dat}_{\infty}^{op}\twoheadrightarrow\mathrm{Cat}_{\infty}^{op}$ is the (unstraightening of the) externalization of the 
internal $\infty$-category $\Delta^{\bullet}$ in $\mathrm{Cat}_{\infty}^{op}$.

\begin{notation}
For an $\infty$-category $\mathcal{C}$, we denote the $\infty$-category $\mathrm{Fun}(N(\Delta^{op}),\mathcal{C})$ of 
simplicial objects in $\mathcal{C}$ by $s\mathcal{C}$. For $n\geq 0$ and a subset $J\subseteq [n]$ of 
cardinality $j$, we denote by $d^J\colon [j]\rightarrow [n]$ the according inclusion of linear orders with image $J$, and for a simplicial 
object $X\in s\mathcal{C}$, by $d_J\colon X_n\rightarrow X_{j}$ the according simplicial operator. 
\end{notation}

We recall the definition of internal $\infty$-categories in an $\infty$-category $\mathcal{C}$ with pullbacks from
\cite[Section 3]{rasekhcart}, going only into as much technical detail as necessary to define and study their externalization.

\begin{definition}\label{defSegalobjects}
A \emph{Segal object} in an $\infty$-category $\mathcal{C}$ with pullbacks is a simplicial object $X\in s\mathcal{C}$ such that its 
associated Segal morphisms
\[\xi_n\colon X_n\rightarrow X_1\times_{X_0}\dots\times_{X_0}X_1\]
are equivalences in $\mathcal{C}$ \cite[Definition 1.1.1]{lgood}.
\end{definition}

To every Segal object $X$ in $\mathcal{C}$ (in fact to every $X\in s\mathcal{C}$) we may associate, first, the object
$\mathrm{Zig}\text{-}\mathrm{zag}(X)\simeq X_1\prescript{d_1}{}{\times}_{X_0}^{d_1}X_1\prescript{d_0}{}{\times}_{X_0}^{d_0}X_1$ of zig-zags 
in $X$, and second, the object $\mathrm{Equiv}(X)\subseteq X_3$ of internal equivalences in $X$ (or more precisely the object of edges
together with a left and a right inverse in $X$) defined as the pullback
\[\xymatrix{
\mathrm{Equiv}(X)\ar[r]\ar[d]\ar@{}[dr]|(.3){\pbs} & X_3\ar[d]^{(d_{\{0,2\}},d_{\{1,2\}},d_{\{1,3\}})} \\
X_1\ar[r]_(.35){(s_0d_0,1,s_0d_1)} & \mathrm{Zig}\text{-}\mathrm{zag}(X).
}\]

\begin{definition}\label{defcompleteness}
A Segal object $X$ in $\mathcal{C}$ is \emph{complete} if the factorized degeneracy
\[\xymatrix{
 & \mathrm{Equiv}(X)\ar[d]\\
X_0\ar[r]_{s_0}\ar@/^/[ur]^{s_0} & X_1
}\]
induced by the degenerated 3-simplex $X(!_{[3]})\colon X_0\rightarrow X_3$ is an equivalence in $\mathcal{C}$ \cite[Definition 3.3]{rasekhcart}. 
\end{definition}

The definition of completeness in Definition~\ref{defSegalobjects} is directly derived from Rezk's original definition 
of completeness of Segal spaces in \cite{rezk}. As complete Segal spaces are exactly the internal $\infty$-categories in the
$\infty$-category $\mathcal{S}$ of spaces, the complete Segal objects in an $\infty$-category $\mathcal{C}$ with pullbacks amount exactly to 
the internal $\infty$-categories in $\mathcal{C}$.

We obtain full $\infty$-subcategories $\mathrm{CS}(\mathcal{C})\subseteq\mathrm{S}(\mathcal{C})\subseteq s\mathcal{C}$ of Segal objects and 
complete Segal objects in $\mathcal{C}$, respectively.

\begin{definition}
A \emph{Segal groupoid} in an $\infty$-category $\mathcal{C}$ with pullbacks is a Segal object $X$ in $\mathcal{C}$ such that the natural 
morphism $d_{\{1,2\}}\colon\mathrm{Equiv}(X)\rightarrow X_1$ is an equivalence in $\mathcal{C}$. The full $\infty$-subcategory of Segal 
groupoids in $\mathrm{S}(\mathcal{C})$ is denoted by $\mathrm{SGpd}(\mathcal{C})$.
\end{definition}

\begin{remark}\label{remseglgrpd}
This definition of a Segal groupoid is equivalent to Lurie's definition of a groupoid object in an $\infty$-category $\mathcal{C}$ with 
pullbacks as given in \cite[Definition 6.1.2.7]{luriehtt}. This can be shown using \cite[Proposition 6.2.1.6.(3)]{luriehtt} and
\cite[Section 3]{rs_bspaces}.
Furthermore, in \cite[Proposition 1.1.14]{lgood}, Lurie constructs the core $X^{\simeq}$ of a Segal object $X$ in
$\mathcal{C}$ whenever $\mathcal{C}$ is left exact. That is a right adjoint
$(\cdot)^{\simeq}\colon\mathrm{S}(\mathcal{C})\rightarrow \mathrm{SGpd}(\mathcal{C})$ to the inclusion of Segal 
groupoids into Segal objects in $\mathcal{C}$. He then defines a Segal object $X$ to be complete if the degeneracy 
$s_0\colon X^{\simeq}_0\rightarrow X^{\simeq}_1$ is an equivalence in $\mathcal{C}$. That is, in Lurie's 
terms, if the groupoid object $X^{\simeq}$ is constant \cite[Remark 1.1.5]{lgood}. This definition is equivalent 
to the one given in Definition~\ref{defSegalobjects} essentially by \cite[Proposition 1.1.13.(2)]{lgood}.\\
\end{remark}

The externalization of a complete Segal object in an $\infty$-category $\mathcal{C}$ with pullbacks can be given as follows. The Yoneda 
embedding $y\colon\mathcal{C}\rightarrow\hat{\mathcal{C}}$ induces 
a functor $sy\colon s\mathcal{C}\rightarrow s\hat{\mathcal{C}}$ by postcomposition. The $\infty$-category $s\hat{\mathcal{C}}$ of simplicial 
objects in turn is equivalent to the $\infty$-category $\mathrm{Fun}(\mathcal{C}^{op},s\mathcal{S})$ of $\mathcal{C}$-indexed simplicial 
spaces simply by virtue of Currying. As the Yoneda embedding is left exact, it restricts to a functor
$sy\colon\mathrm{S}(\mathcal{C})\rightarrow\mathrm{S}(\hat{\mathcal{C}})$ and further to a functor
$sy\colon\mathrm{CS}(\mathcal{C})\rightarrow\mathrm{CS}(\hat{\mathcal{C}})$.
The equivalence $s\hat{\mathcal{C}}\simeq\mathrm{Fun}(\mathcal{C}^{op},s\mathcal{S})$ also restricts to 
equivalences $\mathrm{S}(\hat{\mathcal{C}})\simeq\mathrm{Fun}(\mathcal{C}^{op},\mathrm{S}(\mathcal{S}))$ and
$\mathrm{CS}(\hat{\mathcal{C}})\simeq\mathrm{Fun}(\mathcal{C}^{op},\mathrm{CS}(\mathcal{S}))$ between (complete) Segal objects 
in $\hat{\mathcal{C}}$ and $\mathcal{C}$-indexed (complete) Segal spaces, respectively.
Lastly, the $\infty$-category $\mathrm{CS}(\mathcal{S})$ of complete Segal spaces exhibits an equivalence to the $\infty$-category
$\mathrm{Cat}_{\infty}$ of (small) $\infty$-categories given by the right derived horizontal projection 
$\text{Ho}_{\infty}(p_h)\colon\mathrm{CS}(\mathcal{S})\xrightarrow{\simeq}\mathrm{Cat}_{\infty}$ \cite[Section 4]{jtqcatvsss}.

\begin{definition}\label{defext}
The \emph{externalization functor}
$\mathrm{Ext}\colon\mathrm{CS}(\mathcal{C})\rightarrow\mathrm{Fun}(\mathcal{C}^{op},\mathrm{Cat}_{\infty})$ 
associated to an $\infty$-category $\mathcal{C}$ with pullbacks is given by the composition
\[\mathrm{CS}(\mathcal{C})\xrightarrow{sy}\mathrm{Fun}(\mathcal{C}^{op},\mathrm{CS}(\mathcal{S}))\xrightarrow{\simeq}\mathrm{Fun}(\mathcal{C}^{op},\mathrm{Cat}_{\infty}).\]
\end{definition}

The externalization functor associated to an $\infty$-category $\mathcal{C}$ with pullbacks extends the Yoneda embedding associated to
$\mathcal{C}$ in the sense that the following diagram can be shown to commute in both directions.

\begin{align}\label{diagyonedaext}
\begin{gathered}
\xymatrix{
\mathcal{C}\ar@{^(->}[r]^(.35)y\ar@/_/@{^(->}[d] & \mathrm{Fun}(\mathcal{C}^{op},\mathcal{S})\ar@/_/@{^(->}[d] \\
\mathrm{CS}(\mathcal{C})\ar[r]_(.35){\mathrm{Ext}}\ar@/_/[u]_{(\cdot)_0} & \mathrm{Fun}(\mathcal{C}^{op},\mathrm{Cat}_{\infty})\ar@/_/[u]_{(\cdot)^{\simeq}}
}
\end{gathered}
\end{align}
Here, the left vertical inclusion assigns to an object $C\in\mathcal{C}$ the constant simplicial object in $\mathcal{C}$ with value
$C$, and the right vertical inclusion is the pushforward with the canonical inclusion 
$\mathcal{S}\hookrightarrow\mathrm{Cat}_{\infty}$.

For a given complete Segal object $X$ in $\mathcal{C}$, we denote the corresponding Unstraightened cartesian fibration by
$\mathrm{Ext}(X)\twoheadrightarrow\mathcal{C}$ as well. By construction, its fibers $\mathrm{Ext}(X)(C)$ are naturally 
equivalent to the composition $\text{Ho}_{\infty}(p_h)(sy(X)(C))$. They can be described as the
$\infty$-categories given by the $0$-th row of any Reedy fibrant representative of the complete Segal object
$sy(X)(C)\simeq\mathcal{C}(C,X_{(\cdot)})$ in the category $s\mathbf{S}$ of simplicial spaces.

\begin{definition}\label{defsmallness}
A $\mathcal{C}$-indexed $\infty$-category $\mathcal{E}$ is \emph{small} if there is a complete Segal object $X\in\mathrm{CS}(\mathcal{C})$ 
such that $\mathrm{Ext}(X)\simeq\mathcal{E}$.
\end{definition}

\begin{remark}\label{remmatchrasekh}
Under the equivalence of models of $(\infty,1)$-category theory induced by the Quillen equivalence
\[p_h\colon(s\mathbf{S},\mathrm{CS})\rightarrow(\mathbf{S},\mathrm{QCat})\]
of the Rezk model structure for complete Segal spaces on the one hand, and the Joyal model structure for quasi-categories on the other hand, 
the cartesian fibrations of the form $\mathrm{Ext}(X)\twoheadrightarrow\mathcal{C}$ for a 
complete Segal object $X$ in $\mathcal{C}$ are exactly the ``representable'' cartesian fibrations
$\mathcal{C}_{/X}\twoheadrightarrow\mathcal{C}$ in \cite[Definition 2.1]{rasekhcart} by \cite[Notation 2.5]{rasekhcart}. Note however that 
these are not the cartesian fibrations given by the $\infty$-overcategories $\mathcal{C}_{/X}\twoheadrightarrow\mathcal{C}$ as 
defined in \cite[Proposition 1.2.9.2]{luriehtt}.
In fact,	 whenever a cartesian fibration $p\colon\mathcal{E}\twoheadrightarrow\mathcal{C}$ is presented by a ``Reedy right fibration''
$\mathcal{R}\twoheadrightarrow\mathcal{C}$ \cite[Definition 4.2, Definition 2.2]{rasekhcartfibs}, then the right 
fibration $\mathcal{R}_n\twoheadrightarrow\mathcal{C}$ in $(s\mathbf{S},\mathrm{CS})$ translates exactly to the right 
fibration $\llbracket\Delta^n,\mathcal{E}\rrbracket\twoheadrightarrow\mathcal{C}$ in $(\mathbf{S},\mathrm{QCat})$ via the 
Quillen equivalence $p_h$ for every $n\geq 0$.

\end{remark}

\begin{remark}\label{remextgensegal}
Externalization can also be defined more generally for not necessarily complete Segal spaces (and even for general simplicial objects) in an
$\infty$-category $\mathcal{C}$, using that $\mathcal{C}$-indexed Segal spaces can be completed pointwise. More precisely, the inclusions
$\mathrm{CS}(\mathcal{S)}\hookrightarrow\mathrm{S}(\mathcal{S})\hookrightarrow s\mathcal{S}$ have a left adjoint each (given 
by the fact that the two inclusions are presented by left Bousfield localizations of the respective model structures). The left adjoint
$\rho\colon\mathrm{S}(\mathcal{S})\rightarrow\mathrm{CS}(\mathcal{S})$ maps a Segal space to its completion. We 
thus obtain a generalized externalization functor
\[\mathrm{S}(\mathcal{C})\xrightarrow{sy}\mathrm{Fun}(\mathcal{C}^{op},\mathrm{S}(\mathcal{S}))\xrightarrow{\rho_{\ast}}\mathrm{Fun}(\mathcal{C}^{op},\mathrm{CS}(\mathcal{S}))\xrightarrow{\simeq}\mathrm{Fun}(\mathcal{C}^{op},\mathrm{Cat}_{\infty})\]
which restricts on $\mathrm{CS}(\mathcal{C})$ (up to equivalence) exactly to the externalization functor from Definition~\ref{defext}.
Note that the definition of a Segal space in an $\infty$-category $\mathcal{C}$ does not require the existence of pullbacks in $\mathcal{C}$, 
as the Segal conditions only express that a certain specified cone in $\mathcal{C}$ is limiting.
\end{remark}

Suppose for the moment that $\mathcal{C}$ is a 1-category (with pullbacks). Then $\mathrm{S}(\mathcal{C})$ is the 1-category of 
internal categories in $\mathcal{C}$, and $\mathrm{CS}(\mathcal{C})$ is the 1-category of internal categories with a discrete core 
(Remark~\ref{remseglgrpd}).
The 1-categorical externalization functor as described in \cite[Section B.2.3]{elephant} and \cite[Section 7.3]{jacobsttbook} is a functor
\begin{align}\label{equdefext1}
\mathrm{Ext}_1\colon \mathrm{S}(\mathcal{C})\rightarrow\mathrm{Fun}(\mathcal{C}^{op},\mathrm{Cat}).
\end{align}
For an internal category $X\in\mathrm{S}(\mathcal{C})$ and an object $C\in\mathcal{C}$, the objects of the categories $\mathrm{Ext}_1(X)(C)$ 
are the morphisms
$\mathcal{C}(C,X_0)$. The morphisms between vertices $x,y\colon C\rightarrow X_0$ in $\mathrm{Ext}_1(X)(C)$ are explicitly given by
morphisms $f\colon C\rightarrow X_1$ such that $d_1f=x$ and $d_0f=y$. Given a morphism $f\colon C\rightarrow D$ in
$\mathcal{C}$, we obtain the functorial action $f^{\ast}\colon \mathrm{Ext}_1(X)(D)\rightarrow\mathrm{Ext}_1(X)(C)$ by precomposition 
with $f$. 

Although internal categories in a 1-category $\mathcal{C}$ are generally not complete, they automatically are internal Segal
\emph{categories}. 
Indeed, the Yoneda embedding of $\mathcal{C}$ factors through the $\infty$-category of discrete spaces, and so the functor
$sy\colon s\mathcal{C}\rightarrow s\mathcal{S}$ factors through the $\infty$-category $\mathrm{PCat}$ of precategories
\cite[Section 5]{jtqcatvsss} -- that is, the $\infty$-category of simplicial spaces $X$ such that $X_0$ is a discrete space. In fact, for 
any given internal category $X$ in $\mathcal{C}$ and any given object $C\in\mathcal{C}$, every column of the simplicial space
$sy(X)(C)=\mathcal{C}(C,X_{(\cdot)})$ is discrete. It follows in particular that the simplicial spaces $sy(X)(C)$ are Reedy fibrant. 
If we denote the $\infty$-category of Segal categories -- i.e.\ of Segal spaces $X$ such that $X_0$ is a discrete space -- by
$\mathrm{SC}(\mathcal{S})$, then the right derived horizontal projection
$\mathrm{Ho}_{\infty}(p_h)\colon\mathrm{SC}(\mathcal{S})\rightarrow\mathrm{Cat}_{\infty}$ is again an equivalence
\cite[Theorem 5.6]{jtqcatvsss}. Since every object in the image of $sy\colon\mathrm{S}(\mathcal{C})\rightarrow\mathrm{PCat}$ is already 
fibrant, it follows that the induced composition
\[\mathrm{S}(\mathcal{C})\xrightarrow{sy}\mathrm{Fun}(\mathcal{C}^{op},\mathrm{SC}(\mathcal{S}))\xrightarrow{\simeq}\mathrm{Fun}(\mathcal{C}^{op},\mathrm{Cat}_{\infty})\]
computes exactly the 1-categorical externalization (\ref{equdefext1}).

\begin{proposition}\label{prop1ext}
Let $\mathcal{C}$ be a 1-category. Then the externalization functor $\mathrm{Ext}\colon\mathrm{S}(\mathcal{C})\rightarrow\mathrm{Fun}(\mathcal{C}^{op},\mathrm{Cat}_{\infty})$ from Remark~\ref{remextgensegal} and the 1-categorical externalization functor
$\mathrm{Ext}_1\colon\mathrm{S}(\mathcal{C})\rightarrow\mathrm{Fun}(\mathcal{C}^{op},\mathrm{Cat}_{\infty})$ are naturally equivalent.
\end{proposition}
\begin{proof}
The inclusion $\iota\colon\mathrm{PCat}\rightarrow s\mathbf{S}$ of precategories in simplicial spaces is part of a left Quillen equivalence
$\iota\colon(\mathrm{PCat},\mathrm{SC})\rightarrow(s\mathbf{S},\mathrm{CS})$ between the Hirschowitz-Simpson model structure for Segal 
categories and the Rezk model structure for complete Segal spaces by \cite[Section 6]{bergner3models}. The induced (outer) triangle of 
equivalences
\[\xymatrix{
\mathrm{SC}(\mathcal{S})\ar[dr]^{\mathrm{Ho}_{\infty}(p_h)}\ar@/_2pc/[dd]_{\mathrm{Ho}_{\infty}(\iota)}\ar@{^(->}[d] & \\
\mathrm{S}(\mathcal{C})\ar[d]^{\rho} & \mathrm{Cat}_{\infty} \\
\mathrm{CS}(\mathcal{S})\ar[ur]_{\mathrm{Ho}_{\infty}(p_h)} & 
}\]
commutes up to equivalence again by \cite[Section 5]{jtqcatvsss}. Precomposition of this natural equivalence with the functor
$sy\colon \mathrm{S}(\mathcal{C})\rightarrow\mathrm{SC}(\mathcal{S})$ hence induces a natural 
equivalence between $\mathrm{Ext}_1$ and $\mathrm{Ext}$ as stated.
\end{proof}

Thus, the $\infty$-categorical externalization functor is a natural generalization of its 1-categorical analogue.

\begin{proposition}\label{remextgrpdchar}
A Segal object $X$ in an $\infty$-category $\mathcal{C}$ with pullbacks is a Segal groupoid if and only if its externalization is an indexed
$\infty$-groupoid. 
\end{proposition}
\begin{proof}
Let $X$ be a Segal object in $\mathcal{C}$. The Yoneda embedding is conservative and so the canonical morphism
$\mathrm{Equiv}(X)\rightarrow X_1$ is an equivalence in $\mathcal{C}$ if and only if for every object $C\in\mathcal{C}$, the presheaf
$sy(X(C))$ is a Segal space such that $\mathcal{C}(C,\mathrm{Equiv}(X))\rightarrow\mathcal{C}(C,X_1)$ is an equivalence of spaces. But each 
of these maps is naturally equivalent to the canonical map $\mathrm{Equiv}(sy(X)(C))\rightarrow(sy(X)(C))_1$ associated to the Segal space 
$sy(X)(C)$. Hence, $X$ is a Segal groupoid in $\mathcal{C}$ if and only if each $sy(X)(C)$ is a Segal groupoid in $\mathcal{S}$. By an 
application of \cite[Proposition 6.1.2.6.(3)]{luriehtt}, such are exactly the \emph{Bousfield-Segal spaces} in terms of
\cite[Corollary 5.4]{rs_bspaces}. Now, a Segal space is a Bousfield-Segal space if and only if its right derived horizontal projection
is a Kan complex by \cite[Remark 5.5]{rs_bspaces}. In summary, it follows that a Segal object $X$ in $\mathcal{C}$ is a Segal object if and 
only if for all $C\in\mathcal{C}$, the quasi-category $\mathrm{Ext}(X)(C)=\mathrm{Ho}_{\infty}(p_h)(sy(X)(C))$ is a Kan complex.
\end{proof}

For the following, recall that Joyal and Tierney constructed in fact two Quillen equivalences between the model structure for complete Segal 
spaces and the model structure for quasi-categories. On the one hand, we have already considered the right Quillen functor
\[p_h\colon(s\mathbf{S},\mathrm{CS})\rightarrow(\mathbf{S},\mathrm{QCat}).\]
On the other hand, in \cite[Section 4]{jtqcatvsss} the authors construct a left Quillen functor
\[t^!\colon(\mathbf{S},\mathrm{QCat})\rightarrow(s\mathbf{S},\mathrm{CS})\]
which is part of a Quillen equivalence as well. If by $I\Delta^n\in\mathbf{S}$ we denote the nerve of the free groupoid 
on $[n]$, the functor $t^!$ is given by the formula
\begin{align*}
t^!(J)_{mn} & =\mathbf{S}(\Delta^m\times I[\Delta^n],J) \\
 & =((J^{\Delta^m})^{\simeq})_n
\end{align*}
for simplicial sets $J$. In particular, for every quasi-category $\mathcal{C}$ we have
\[t^!(\mathcal{C})=\mathrm{Cat}_{\infty}(\Delta^{\bullet},\mathcal{C})=sy(\Delta^{\bullet})(\mathcal{C})\]
for $\Delta^{\bullet}$ considered as a simplicial object in $\mathrm{Cat}_{\infty}^{op}$.
In the proof of \cite[Theorem 4.12]{jtqcatvsss}
the authors show that the composition $p_h\circ t^!\colon\mathbf{S}\rightarrow\mathbf{S}$ is isomorphic to the identity, 
and so it follows that the composition
\[\mathrm{Ho}_{\infty}(p_h)\circ\mathrm{Cat}_{\infty}(\Delta^{\bullet},\phv)\colon\mathrm{Cat}_{\infty}\rightarrow\mathrm{Cat}_{\infty}\]
is naturally equivalent to the identity on $\mathrm{Cat}_{\infty}$.

\begin{example}\label{expleextqcat}
Via the equivalence $t^!\colon\mathrm{Cat}_{\infty}\rightarrow\mathrm{CS}(\mathcal{S})$ we can define the externalization of a quasi-category
$\mathcal{C}$ as $\mathrm{Ext}(t^!(\mathcal{C}))$. The fact that the composition
$\mathrm{Ho}_{\infty}(p_h)\circ t^!\colon\mathrm{Cat}_{\infty}\rightarrow\mathrm{Cat}_{\infty}$ recovers the identity up to equivalence 
induces for any given quasi-category $\mathcal{C}$ a natural equivalence
$\mathrm{Ext}(\mathcal{C})\simeq\mathrm{Fun}(\phv,\mathcal{C})$ in $\mathrm{Fun}(\mathbf{S}^{op},\mathrm{Cat}_{\infty})$.
This means that the externalization of an $\infty$-category considered as an internal category in spaces is just the naive indexing of 
$\mathcal{C}$ from Example~\ref{exmplebasicindcats}.2.
\end{example}

\begin{example}[The universal complete Segal object]\label{exmpleunivextpre}
The Yoneda embedding $y\colon\Delta\rightarrow\mathbf{S}$ yields a simplicial object
$\Delta^{\bullet}$  in $\mathrm{Cat}_{\infty}^{op}$ which is a complete Segal object precisely because the spine 
inclusions $S_n\rightarrow\Delta^n$ (i.e.\ the inclusions of the maximal chain of non-degenerate 1-simplices into
$\Delta^n$) and the endpoint inclusion $\Delta^0\rightarrow I\Delta^1$ are weak categorical equivalences in the Joyal 
model structure. We therefore can construct the externalization
\begin{align}\label{equexmpleunivextpre}
\mathrm{Ext}(\Delta^{\bullet})\colon\mathrm{Cat}_{\infty}\rightarrow\mathrm{Cat}_{\infty}
\end{align}
given exactly by $\mathrm{Ho}_{\infty}(p_h)\circ\mathrm{Cat}_{\infty}(\Delta^{\bullet},\phv)$ via Definition~\ref{defext}. By the above, 
it follows that the given externalization is naturally equivalent to the identity on
$\mathrm{Cat}_{\infty}$. In other words, the universal cocartesian fibration
$\pi\colon\mathrm{Dat}_{\infty}\rightarrow\mathrm{Cat}_{\infty}$ -- defined as the unstraightening of the identity on
$\mathrm{Cat}_{\infty}$ -- is equivalent to the externalization of $\Delta^{\bullet}\in\mathrm{Cat}_{\infty}^{op}$. In this sense, it 
is represented by the co-complete co-Segal object $\Delta^{\bullet}$ in $\mathrm{Cat}_{\infty}$.

This can be understood as the rather plain and by now folklore fact that the model category $(\mathbf{S},\mathrm{QCat})$ together with the 
cosimplicial object $\Delta^{\bullet}$ in $\mathbf{S}$ is a ``theory of $(\infty,1)$-categories'' in the sense of To\"{e}n 
\cite{toeninftycats}. Indeed, $\Delta^{\bullet}$ being a complete Segal object in $\mathrm{Cat}_{\infty}^{op}$ means dually that
$\Delta^{\bullet}$ considered as a ``weak co-category object'' in $\mathrm{Cat}_{\infty}$ is an interval 
\cite[Definition 3.4]{toeninftycats}. The functor
$t^!=sy(\Delta^{\bullet})\colon(\mathbf{S},\mathrm{QCat})\rightarrow(s\mathbf{S},\mathrm{CS})$ is exactly the comparison functor in
\cite[Theorem 5.1]{toeninftycats}.
\end{example}

Under the equivalence of models of $(\infty,1)$-category theory alluded to in Remark~\ref{remmatchrasekh}, the following proposition is 
equivalent to a corresponding sequence of results in \cite[Section 4.1]{rasekhcart}.

\begin{proposition}\label{lemmaextrect}
Let $\mathcal{C}$ be an $\infty$-category with pullbacks.
\begin{enumerate}
\item Let $X$ be a complete Segal object in $\mathcal{C}$. Then for every $n\geq 0$, the presheaf
$(\mathrm{Ext}(X)^{\Delta^n})^{\simeq}$ over $\mathcal{C}$ is represented by $X_n\in\mathcal{C}$.
In particular, the presheaf $\mathrm{Ext}(X)^{\simeq}\in\hat{\mathcal{C}}$ is represented by $X_0\in\mathcal{C}$.
\item Let $\mathcal{E}$ be a $\mathcal{C}$-indexed $\infty$-category such that the presheaves
$\mathcal{E}^{\simeq}$ and $(\mathcal{E}^{\Delta^1})^{\simeq}$ are represented each by some object $E_0,E_1\in\mathcal{C}$, respectively.
Then there is a complete Segal object $X$ in $\mathcal{C}$ such that $X_i\simeq E_i$ for $i=0,1$, and such that
$\mathcal{E}\simeq\mathrm{Ext}(X)$. In particular, $\mathcal{E}$ is small. 
\end{enumerate}
\end{proposition}

\begin{proof}
For Part 1 we want to show that for all $X\in\mathrm{CS}(\mathcal{C})$ and all $n\geq 0$ the composite
\begin{align*}
\xymatrix{
\mathcal{C}^{op}\ar[r]^{\mathrm{Ext}(X)} & \mathrm{Cat}_{\infty}\ar[rr]^{\mathrm{Fun}(\Delta^n,\phv)}\ar@/_1pc/[rrr]_{\mathrm{Cat}_{\infty}(\Delta^n,\phv)} & & \mathrm{Cat}_{\infty}\ar[r]^(.65){(\cdot)^{\simeq}} & \mathcal{S}
}
\end{align*}
is represented by $X_n$. By definition of the externalization functor, it therefore suffices to show that the following diagram of
$\infty$-categories commutes up to equivalence.
\begin{align}\label{diaglemmaextrect}
\begin{gathered}
\xymatrix{
\mathcal{C}^{op}\ar@/_1pc/[drr]_{y\circ X}\ar[r]^(.5){sy(X)} & \mathrm{CS}(\mathcal{S})\ar[r]^(.55){\mathrm{Ho}_{\infty}(p_h)}\ar[dr]^{\mathrm{id}} & \mathrm{Cat}_{\infty}\ar[d]^{\mathrm{Cat}_{\infty}(\Delta^{\bullet},\phv)} \\
 & & \mathrm{CS}(\mathcal{S})
}
\end{gathered}
\end{align}
Towards homotopy-commutativity of the top right triangle, we have seen that both non-identity components of (\ref{diaglemmaextrect}) are 
equivalences, and that their converse composition is equivalent to the identity. It follows that they are mutually inverse to each other, so 
that the composition
\begin{align*}
\mathrm{Cat}_{\infty}(\Delta^{\bullet},\phv)\circ\mathrm{Ho}_{\infty}(p_h)\colon\mathrm{CS}(\mathcal{S})\rightarrow\mathrm{CS}(\mathcal{S})
\end{align*}
is equivalent to the identity on $\mathrm{CS}(\mathcal{S})$ as well. The bottom left triangle of (\ref{diaglemmaextrect}) commutes trivially 
as $sy(X):=y_{\ast}X:=y\circ X$ is the functorial pushforward by definition. 

For Part 2, consider the composition
\[\mathcal{C}^{op}\xrightarrow{\mathcal{E}}\mathrm{Cat}_{\infty}\xrightarrow{\mathrm{Cat}_{\infty}(\Delta^{\bullet},\phv)}\mathrm{CS}(\mathcal{S})\hookrightarrow s\mathcal{S}.\]
Given the canonical equivalence $\mathrm{Fun}(\mathcal{C}^{op},s\mathcal{S})\simeq s\hat{\mathcal{C}}$ and its restriction
$\mathrm{Fun}(\mathcal{C}^{op},\mathrm{CS}(\mathcal{S}))\simeq \mathrm{CS}(\hat{\mathcal{C}})$, we can consider 
the composition $\mathrm{Cat}_{\infty}(\Delta^{\bullet},\phv)\circ\mathcal{E}=(\mathcal{E}^{\Delta^{\bullet}})^{\simeq}$ as a complete 
Segal object in $\hat{\mathcal{C}}$. We want to show that the simplicial object
$(\mathcal{E}^{\Delta^{\bullet}})^{\simeq}\in s\hat{\mathcal{C}}$ is contained in the essential image of
\[sy\colon s\mathcal{C}\rightarrow s\hat{\mathcal{C}}.\]
Indeed, if we obtain a simplicial object $X\in s\mathcal{C}$ such that
$sy(X)\simeq(\mathcal{E}^{\Delta^{\bullet}})^{\simeq}\in s\hat{\mathcal{C}}$, then $X$ is automatically a complete Segal object in
$\mathcal{C}$. That is, because the Yoneda embedding preserves and reflects both equivalences and all limits that exist in $\mathcal{C}$, 
and so the square
\[\xymatrix{
\mathrm{CS}(\mathcal{C})\ar[r]^{sy}\ar@{^(->}[d] & \mathrm{CS}(\hat{\mathcal{C}})\ar@{^(->}[d]  \\
s\mathcal{C}\ar[r]_{sy} & s\hat{\mathcal{C}}
}\]
is cartesian. Furthermore, in this case we obtain an equivalence
$\mathrm{Ho}_{\infty}(p_h)\circ sy(X)\simeq\mathrm{Ho}_{\infty}(p_h)\circ(\mathrm{Cat}_{\infty}(\Delta^{\bullet},\phv)\circ\mathcal{E})$, 
where the left hand side is $\mathrm{Ext}(X)$ by definition, and the right hand side is equivalent to $\mathcal{E}$ itself. Thus, to prove 
the statement we are left to construct a simplicial object $X$ in $\mathcal{C}$ such that
$sy(X)\simeq(\mathcal{E}^{\Delta^{\bullet}})^{\simeq}$ in $s\hat{\mathcal{C}}$. Therefore in turn, it suffices to show that for all
$n\geq 0$ there is an object $X_n\in\mathcal{C}$ such that $X_n$ represents the presheaf $(\mathcal{E}^{\Delta^n})^{\simeq}$. That is, 
because any such collection of objects $(X_n)_{n\geq 0}$ in $\mathcal{C}$ together with equivalences
$e_n\colon y(X_n)\xrightarrow{\simeq}(\mathcal{E}^{\Delta^n})^{\simeq}$ in $\hat{\mathcal{C}}$ induces a homotopy-commutative square of the 
form
\[\xymatrix{
N(\Delta^{op})_0\ar[r]^(.6){(X_n)_{n\geq 0}}\ar@{^(->}[d] & \mathcal{C}\ar[d]^y \\
N(\Delta^{op})\ar[r]_(.6){(\mathcal{E}^{\Delta^{\bullet}})^{\simeq}} & \hat{\mathcal{C}}.
}\]
Here, the left vertical functor is the canonical inclusion of the underlying set of $N(\Delta^{op})$ into $N(\Delta^{op})$. But the class of 
essentially surjective functors is left orthogonal to the class of fully faithful functors in $\mathrm{Cat}_{\infty}$, and the Yoneda 
embedding is fully faithful. We thus obtain a lift
$X\colon N(\Delta)^{op}\rightarrow\mathcal{C}$ to this square (such that both resulting 
triangles commute up to equivalence) which is precisely a simplicial object as required.

And indeed, such objects $X_n=E_n$ exist for $n=0,1$ by assumption. Since $\mathcal{C}$ has pullbacks, for $n\geq 2$ we may simply 
define $X_n:=X_1\times_{X_0}\dots\times_{X_0}X_1$, so that $y(X_n)\simeq(\mathcal{E}^{\Delta^n})^{\simeq}$ by virtue of the fact that
$(\mathcal{E}^{\Delta^{\bullet}})^{\simeq}$ is a Segal object in $\hat{\mathcal{C}}$.
\end{proof}

Generally, the existence of an internal presentation of a given $\mathcal{C}$-indexed $\infty$-category can be captured 
via $(\infty,1)$-comprehension in the following way.

\begin{theorem}\label{thmsmall=ext}
Let $\mathcal{C}$ be an $\infty$-category with finite limits and $\mathcal{E}$ be a $\mathcal{C}$-indexed $\infty$-category. Then
$\mathcal{E}$ is small if and only if it is both globally small and locally small.
\end{theorem}

\begin{proof}
The indexed $\infty$-category $\mathcal{E}$ is both globally and locally small if and only if it has both
$(\emptyset\rightarrow\Delta^0)$- and $(\emptyset\rightarrow\Delta^1)$-comprehension by Corollary~\ref{cormonocomp}.
By Example~\ref{explerepslenattransfoverpt}, that is if and only if both the presheaves $\mathcal{E}^{\simeq}$ and
$(\mathcal{E}^{\Delta^1})^{\simeq}$ are representable. This in turn is the case if and only if $\mathcal{E}$ is small 
Proposition~\ref{lemmaextrect}.2.
\end{proof}

\begin{remark}
Suppose the $\infty$-category $\mathcal{C}$ has finite limits and $\mathcal{E}$ itself is a presheaf on $\mathcal{C}$. In this case, 
by virtue of Proposition~\ref{remextgrpdchar}, Theorem~\ref{thmsmall=ext} states that $\mathcal{E}$ (as a $\mathcal{C}$-indexed
$\infty$-category) is small if and only if there is a complete Segal groupoid $X$ in $\mathcal{C}$ such that $\mathrm{Ext}(X)$ is equivalent 
to $\mathcal{E}$ over $\mathcal{C}$. 
As the square (\ref{diagyonedaext}) commutes, and the corestriction of the inclusion
$\mathcal{C}\hookrightarrow\mathrm{CS}(\mathcal{C})$ to the full $\infty$-subcategory of complete groupoid objects in $\mathcal{C}$ is an 
equivalence, that means that $\mathcal{E}$ is small if and only if the presheaf $\mathcal{E}$ is representable. This however is a fairly 
trivial matter; smallness of $\mathcal{E}$ implies representability of $\mathcal{E}$ by definition, and local smallness of representables is 
for free by Example~\ref{remlocsmallrep}. Thus, thinking of the externalization functor as a generalization of the Yoneda embedding via 
Diagram~(\ref{diagyonedaext}), Theorem~\ref{thmsmall=ext} can be thought of as generalization of the designating fact that the small 
presheaves over a left exact $\infty$-category are exactly the representable ones.
\end{remark}

\begin{remark}\label{reminftysmall2}
Theorem~\ref{thmsmall=ext} holds in ordinary category theory if one replaces global smallness with the existence of (split) generic objects 
\cite{jacobsttbook, sterlinggenobj}.
To construct a split fibration with split generic object (and hence an internal category) from a fibration with generic object as in 
\cite[Construction 39]{sterlinggenobj}, it may be interesting to note that one uses the (identity-on-objects, fully faithful)-factorization 
on $\mathrm{Cat}$, which falls into the category of ``god-given'' identity principles discussed in \cite[Paragraph 8.7]{benaboufibfound}, and 
does not exist in $\infty$-category theory. In Remark~\Ref{remaltdefgencomp}, we will see that this is of consequence.

We can think of global smallness as a generalization of the completeness condition in the following sense. Let
$c\colon\mathcal{C}\rightarrow\mathrm{S}(\mathcal{C})$ be the functor which assigns to an object $C$ the constant simplicial object with 
value $C$. Then a Segal object in $\mathcal{C}$ is complete exactly if the core $X^{\simeq}\in\mathrm{S}(\mathcal{C})$ is contained in the 
essential image of 
the functor $c$ (Remark~\ref{remseglgrpd}). Likewise, a $\mathcal{C}$-indexed $\infty$-category $\mathcal{E}$ is globally small by definition 
exactly if its core $\mathcal{E}^{\simeq}$ is contained in the essential image of the functor $y\simeq\mathrm{Ext}\circ c$ (by commutativity 
of Diagram~(\ref{diagyonedaext})). Indeed, completeness of a Segal object $X$ implies global smallness of $\mathrm{Ext}(X)$, and by virtue of 
Theorem~\ref{thmsmall=ext}, global smallness of a locally small $\mathcal{C}$-indexed $\infty$-category recovers a representing complete 
Segal object. Note however that for a given Segal object $X$ in $\mathcal{C}$, smallness of $\mathrm{Ext}(X)$ is equivalent to the 
existence of a Segal completion $X\rightarrow\rho(X)$ in $\mathcal{C}$ rather than to completeness of $X$ itself.
Thus, the fact that the externalization of an internal category in a 1-category $\mathcal{C}$ is generally not globally 
small (via Proposition~\ref{prop1ext}) corresponds exactly to the fact that it generally is not complete and can't be completed in
$\mathcal{C}$.
\end{remark}

\begin{corollary}\label{corfinitecomp}
Suppose $\mathcal{C}$ has finite limits. Then a $\mathcal{C}$-indexed $\infty$-category $\mathcal{E}$ has
any of the equivalent classes of comprehension schemes from Corollary~\ref{cormonocomp} if and only if it is small. In particular, for every 
$X\in\mathrm{CS}(\mathcal{C})$ and for every finite $\infty$-category 
$J$, there is a complete Segal object $X^J\in\mathrm{CS}(\mathcal{C})$ which represents the $\mathcal{C}$-indexed $\infty$-category
$\mathrm{Ext}(X)^J$ of $J$-indexed diagrams.
\end{corollary}
\begin{proof}
Given $X\in\mathrm{CS}(\mathcal{C})$ and a finite quasi-category $J$, the $\mathcal{C}$-indexed $\infty$-category $\mathrm{Ext}(X)$ has both 
$(\emptyset\rightarrow J)$-comprehension and $(\emptyset\rightarrow J\times\Delta^1)$-comprehension by Theorem~\ref{thmsmall=ext} and 
Corollary~\ref{cormonocomp}. This means that the  $\mathcal{C}$-indexed $\infty$-category $\mathrm{Ext}(X)^J$ has 
both $(\emptyset\rightarrow\Delta^0)$-comprehension and $(\emptyset\rightarrow\Delta^1)$-comprehension, and hence is small again by 
Theorem~\ref{thmsmall=ext}.
\end{proof}

\begin{corollary}\label{corextemb}
Suppose $\mathcal{C}$ has finite limits and let $\mathrm{icat}(\mathcal{C})\subset\mathrm{Fun}(\mathcal{C}^{op},\mathrm{Cat}_{\infty})$ be 
the full $\infty$-subcategory of globally small and locally small indexed $\infty$-categories over $\mathcal{C}$. Then the externalization 
construction
\begin{align}\label{equcorextemb}
\mathrm{Ext}\colon\mathrm{CS}(\mathcal{C})\rightarrow\mathrm{icat}(\mathcal{C})
\end{align}
is an equivalence of $\infty$-categories.
\end{corollary}
\begin{proof}
The externalization construction is fully faithful as noted in \cite[Theorem 2.7]{rasekhcart} (via Remark~\ref{remmatchrasekh}).
\end{proof}

\begin{remark}\label{remcorfinitecomp}
Via Corollary~\ref{corextemb} one can equip $\mathrm{CS}(\mathcal{C})$ with a canonical $(\infty,2)$-categorical structure such that the 
externalization construction becomes an embedding of $(\infty,2)$-categories. The last two corollaries then show that this
$(\infty,2)$-category $\mathrm{CS}(\mathcal{C})$ has cotensors with finite $\infty$-categories 
whenever $\mathcal{C}$ has finite limits, and that the externalization functor preserves them. This generalizes \cite[(4.4)]{streetintcat} to 
$\infty$-category theory and is developed in detail in \cite{rs_ext}. Accordingly, it follows from Corollary~\ref{corsmallallcompschemes} 
below that $\mathrm{CS}(\mathcal{C})$ is cotensored over $\mathrm{Cat}_{\infty}$ whenever $\mathcal{C}$ has all small limits.
\end{remark}

\begin{example}[The universal cartesian fibration and the universal right fibration]\label{exmpleuniversalext}
We have seen in Example~\ref{exmpleunivextpre} that the universal cartesian fibration
$\pi^{op}\colon\mathrm{Dat}_{\infty}^{op}\twoheadrightarrow\mathrm{Cat}_{\infty}^{op}$ is the 
externalization of the complete Segal object $\Delta^{\bullet}$ in $\mathrm{Cat}_{\infty}^{op}$.
Unfolding the definitions, local smallness 
of the fibration $\pi^{op}$ means that every small $\infty$-category in $\mathrm{Cat}_{\infty}$ can be freely extended by an 
edge between any two of its objects to give another small $\infty$-category (Example~\ref{exmplecompschemesidcat}). 
Furthermore, smallness of the universal cartesian fibration recovers the fact that the 
universal right fibration $(\mathcal{S}_{\ast})^{op}\twoheadrightarrow\mathcal{S}^{op}$ is represented by the terminal object in
$\mathcal{S}$ as well \cite[Proposition 3.3.2.6]{luriehtt}. Indeed, the universal right fibration is the pullback of the universal cartesian 
fibration $\pi^{op}$ along a left adjoint (Example~\ref{exmpleunivcomptransfer}.6). The following corollary then shows that the complete 
Segal object $I\Delta^{\bullet}\in\mathrm{CS}(\mathcal{S}^{op})$ yields an equivalence
$(\mathcal{S}_{\ast})^{op}\simeq\mathrm{Ext}(I\Delta^{\bullet})$ over $\mathcal{S}^{op}$. But every $I\Delta^n$ is 
contractible in $\mathcal{S}$, and so
$\mathrm{Ext}(I\Delta^{\bullet})\simeq\mathcal{S}^{op}_{/\ast}\simeq(\mathcal{S}_{\ast/})^{op}$.
\end{example}

The fact that the identity on $\mathrm{Cat}_{\infty}$ is small has useful applications. For instance, reindexing along left adjoints 
preserves smallness as the following generalization of \cite[Lemma 2.3.6]{elephant} shows.

\begin{lemma}\label{corextexchange}
Let $F\colon\mathcal{D}\rightarrow\mathcal{C}$ be a functor between $\infty$-categories with pullbacks. Suppose it has a 
right adjoint $R$. Then for every complete Segal object $X$ in $\mathcal{C}$, the restriction
$F^{\ast}(\mathrm{Ext}(X))\in\mathrm{Fun}(\mathcal{D}^{op},\mathrm{Cat}_{\infty})$ is equivalent to the externalization of the complete Segal 
object $sR(X)\in\mathrm{CS}(\mathcal{D})$.
\end{lemma}

\begin{proof}
We first note that $sR(X)\in s\mathcal{D}$ is a complete Segal object, because $R$ is a right adjoint and hence preserve limits.
Now consider the composition 
\begin{align}\label{equcorextexchange}
F^{\ast}(sy(X))\xrightarrow{\vec{R}} (RF)^{\ast}(sy(sR(X)))\xrightarrow{\eta^{\ast}} sy(sR(X))
\end{align}
of natural transformations in $\mathrm{Fun}(\mathcal{D}^{op},\mathrm{CS}(\mathcal{S}))$. Here, $\eta\colon 1\rightarrow RF$ denotes 
the unit of the adjunction $F\adj R$, and $\vec{R}$ denotes the action of the right adjoint $R$ on hom-spaces. The latter can be formally 
constructed  via the explicit definition of the Yoneda embedding in \cite[Section 5.1.3]{luriehtt}. To show that the composition 
(\ref{equcorextexchange}) is an equivalence, it suffices to do so pointwise. But for any given $D\in\mathcal{D}$ and every $n\geq 0$, the 
natural transformation (\ref{equcorextexchange}) evaluated at $D$ and at $n$ is exactly the functor
\[\mathcal{C}(F(D),X_n)\xrightarrow{\vec{R}}\mathcal{D}(R(F(D)),R(X_n))\xrightarrow{\eta^{\ast}}\mathcal{D}(D,R(X_n))\]
which is an equivalence by the fact that $(F,R,\eta)$ forms an adjunction \cite[Proposition 5.2.2.8]{luriehtt}. In particular, the 
natural equivalence (\ref{equcorextexchange}) induces an equivalence between the $\mathcal{D}$-indexed $\infty$-categories
\[\mathrm{Ho}_{\infty}(p_h)\circ(F^{\ast}(sy(X)))=F^{\ast}(\mathrm{Ho}_{\infty}(p_h)(sy(X)))=F^{\ast}\mathrm{Ext}(X)\]
and $\mathrm{Ho}_{\infty}(p_h)\circ sy(sR(X))=\mathrm{Ext}(sR(X))$.
\end{proof}

Whenever $\mathcal{C}$ is not only left exact but complete, we thus obtain the following characterization of small indexed 
$\infty$-categories over $\mathcal{C}$.

\begin{proposition}\label{proprightadj}
Suppose $\mathcal{C}$ is a complete $\infty$-category. Then a $\mathcal{C}$-indexed $\infty$-category $\mathcal{E}$ is small if and only if 
the functor $\mathcal{E}\colon\mathcal{C}^{op}\rightarrow\mathrm{Cat}_{\infty}$  has a left adjoint.
\end{proposition}
\begin{proof}
Suppose $\mathcal{E}\colon\mathcal{C}^{op}\rightarrow\mathrm{Cat}_{\infty}$ has a left adjoint $L$. Then its opposite
$\mathcal{E}^{op}\colon\mathcal{C}\rightarrow\mathrm{Cat}_{\infty}^{op}$ is itself left adjoint to $L^{op}$. As the identity
$\mathrm{id}\colon(\mathrm{Cat}_{\infty}^{op})^{op}\rightarrow\mathrm{Cat}_{\infty}$ is the externalization of the complete Segal object
$\Delta^{\bullet}\in\mathrm{CS}(\mathrm{Cat}_{\infty}^{op})$ (Example~\ref{exmpleuniversalext}), it follows from Lemma~\ref{corextexchange} 
that $\mathcal{E}=\mathcal{E}^{\ast}(\mathrm{id})$ is the externalization of the complete Segal object $L^{op}(\Delta^{\bullet})$.
Vice versa, let $X$ be a complete Segal object in $\mathcal{C}$ and consider its externalization
\[\mathcal{C}^{op}\xrightarrow{sy(X)}\mathrm{CS}(\mathcal{S})\xrightarrow{\mathrm{Ho}_{\infty}(p_h)}\mathrm{Cat}_{\infty}.\]
Since $\mathrm{Ho}_{\infty}(p_h)$ has a left adjoint, it suffices to show that $sy(X)=X^{\ast}\circ y$ has a left adjoint as well.
Therefore, we note that if we consider $X^{op}\colon\Delta\rightarrow\mathcal{C}^{op}$ as a cosimplicial object of $\mathcal{C}^{op}$, then
$sy(X)=N(X^{op})\colon\mathcal{C}^{op}\rightarrow s\mathcal{S}$ is but a nerve construction of $X^{op}$. We therefore can show that $sy(X)$ 
is the right adjoint to the left Kan extension of $X^{op}$ along the Yoneda embedding $y\colon\Delta\rightarrow s\mathcal{S}$.
To do so explicitly, we use \cite[Proposition 3.1.2]{nrsadjoint} and show that for all simplicial spaces $Z\in s\mathcal{S}$, the comma
$\infty$-category
\[\xymatrix{
Z\downarrow sy(X)\ar@{->>}[d]\ar[r]\ar@{}[dr]|(.3){\pbs} & (s\mathcal{S})_{Z/}\ar@{->>}[d]^t \\
\mathcal{C}^{op}\ar[r]_(.5){sy(X)} & s\mathcal{S}
}\]
has an initial object (or equivalently, that it is corepresentable over $\mathcal{C}^{op}$). First, if $Z$ is a representable $y([n])$, we 
obtain a diagram as follows.
\[\xymatrix{
\mathcal{C}^{op}_{X_n/}\ar[r]\ar@{->>}[d]  & \widehat{\mathcal{C}^{op}}_{y(X_n)/}\ar[r]\ar@{->>}[d]  & \widehat{N(\Delta)}_{y([n])/}\ar[r]\ar@{->>}[d]  & \mathcal{S}_{\ast}\ar@{->>}[d]  \\
\mathcal{C}^{op}\ar[r]_{y} & \widehat{\mathcal{C}^{op}}\ar[r]_{X^{\ast}} & \widehat{N(\Delta)}\ar[r]_{\mathrm{ev}([n])} & \mathcal{S}
}\]
The square on the right hand side is homotopy cartesian in the Joyal model structure by Examples~\ref{explscartfibs}.4. Since
$\mathrm{ev}([n])\circ X^{\ast}=\mathrm{ev}(X_n)$, the right composite 
rectangle is a homotopy pullback as well for the same reason. This implies that the center square is a homotopy 
pullback. The square on the left hand side is a homotopy pullback, because the Yoneda embedding is fully faithful. Thus, eventually, the left composite rectangle is a homotopy pullback, which yields the equivalence
\[\mathcal{C}^{op}_{X_n/}\simeq y([n])\downarrow sy(X)\]
over $\mathcal{C}^{op}$.
For general simplicial spaces $Z\in\widehat{N(\Delta)}$, one can now simply use that $Z$ is the colimit of 
representables. Indeed, it follows that the comma $\infty$-category $\widehat{N(\Delta)}_{Z/}$ is a homotopy limit of 
corepresentables, and hence so is the pullback along $sy(X)$ by what we just have shown, since 
homotopy pullbacks preserve homotopy limits. Since $\mathcal{C}$ is assumed to be complete, this limit of 
corepresentables over $\mathcal{C}^{op}$ is corepresented by the respective colimit in $\mathcal{C}^{op}$.
It follows that
$sy(X)\colon\mathcal{C}^{op}\rightarrow\mathrm{Fun}((\Delta)^{op},\mathcal{S})$ has a left adjoint, 
and so does the corestriction $sy(X)\colon\mathcal{C}^{op}\rightarrow\mathrm{CS}(\mathcal{S)}$ as the inclusion
$\mathrm{CS}(\mathcal{S})\subset\mathrm{Fun}(N(\Delta)^{op},\mathcal{S})$ is fully faithful. 
\end{proof}

\begin{remark}\label{remproprightadj}
In particular, whenever $\mathcal{C}$ is a presentable $\infty$-category, Proposition~\ref{proprightadj} states that a $\mathcal{C}$-indexed 
$\infty$-category is the externalization of a complete Segal object if and only if it preserves limits by the Adjoint Functor Theorem 
\cite{nrsadjoint}. In the special case of  presheaves over $\mathcal{C}$ this is exactly the basic fact that limit preserving presheaves are 
representable \cite[Proposition 5.5.2.2]{luriehtt}.
\end{remark}

\begin{remark}\label{remproprightadj2}
In Example~\ref{explegencasela} we saw that a functor $\mathcal{E}\colon\mathcal{C}^{op}\rightarrow K$ into an $\infty$-category $K$ with an 
initial object has $(K,\emptyset\rightarrow k)$-comprehension for all objects $k$ if and only if the functor $\mathcal{E}$ has a left 
adjoint. Thus, Proposition~\ref{proprightadj} (together with Corollary~\ref{cormonocomp}.2) states that in the case that $\mathcal{C}$ is 
complete and $K$ is the $\infty$-category $\mathrm{Cat}_{\infty}$, this criterion can be reduced to only
$(\emptyset\rightarrow\Delta^i)$-comprehension for $i=0,1$.
\end{remark}

\begin{corollary}\label{corsmallallcompschemes}
Suppose $\mathcal{C}$ is a complete $\infty$-category and $\mathcal{E}$ is a small $\mathcal{C}$-indexed $\infty$-category. Then
$\mathcal{E}$ satisfies all comprehension schemes which are satisfied by the identity
$\mathrm{id}\colon(\mathrm{Cat}_{\infty}^{op})^{op}\rightarrow\mathrm{Cat}_{\infty}$. Thus, $\mathcal{E}$ is small if and only if it has
comprehension for all functors between small $\infty$-categories. In particular, for every $X\in\mathrm{CS}(\mathcal{C})$ and for every small
$\infty$-category $J$, there is a complete Segal object $X^J\in\mathrm{CS}(\mathcal{C})$ which represents the $\mathcal{C}$-indexed
$\infty$-category $\mathrm{Ext}(X)^J$ of $J$-indexed diagrams.
\end{corollary}

\begin{proof}
Whenever $\mathcal{E}$ is small, the functor $\mathcal{E}^{op}\colon\mathcal{C}\rightarrow\mathrm{Cat}_{\infty}^{op}$ has a right adjoint by 
Proposition~\ref{proprightadj}. Thus, by Lemma~\ref{lemmachangeofbasecomp}, it follows that all comprehension 
schemes satisfied by the identity on $\mathrm{Cat}_{\infty}$ are transferred to $\mathcal{E}$. The second statement therefore follows 
from Example~\ref{exmplecompschemesidcat}. The last statement follows in the same way as the according claim in 
Corollary~\ref{corfinitecomp}.
\end{proof}

\section{Fibered notions of general comprehension and definability}\label{seccompcats}

In this section we briefly discuss the structural theory of general comprehension schemes (Definition~\ref{defgencomp}). We further introduce 
the notion of relative definability of full subfibrations following \cite[Paragraph 7.4]{benaboufibfound} in terms of general comprehension 
and in context of the notions studied in Section~\ref{secext} and \ref{secextuniv}. 

\begin{definition}
Let $p$ and $q$ be cartesian fibrations over 	$\mathcal{C}$ and $f\colon p\rightarrow q$ be a cartesian functor over $\mathcal{C}$. Say that
$\mathcal{C}$ has \emph{$f$-comprehension} if both natural transformations
$\mathrm{St}(f)^{\simeq}\colon\mathrm{St}(p)^{\simeq}\rightarrow\mathrm{St}(q)^{\simeq}$ and
$(\mathrm{St}(f)^{\Delta^1})^{\simeq}\colon(\mathrm{St}(p)^{\Delta^1})^{\simeq}\rightarrow(\mathrm{St}(q)^{\Delta^1})^{\simeq}$ in
$\hat{\mathcal{C}}$ are representable.
\end{definition}

In terms of Definition~\ref{defgencomp}, an $\infty$-category $\mathcal{C}$ has $f$-comprehension for a cartesian functor $f$ over
$\mathcal{C}$ exactly if $\mathcal{C}$ has both $\mathrm{St}(f)^{\simeq}$-comprehension and
$(\mathrm{St}(f)^{\Delta^1})^{\simeq}$-comprehension. Whenever $\mathcal{C}$ has pullbacks, $f$-comprehension implies
$(\mathrm{St}(f)^{\Delta^n})^{\simeq}$-comprehension for all $n\geq 0$. This follows directly from Lemma~\ref{lemmarepnattransfbasics}.2 and 
the fact that $(\mathrm{St}(f)^{\Delta^{\bullet}})^{\simeq}$ is a functor between (complete) Segal objects in $\hat{\mathcal{C}}$.

\begin{definition}\label{defdefsubfib}
Let $p\colon\mathcal{E}\twoheadrightarrow\mathcal{C}$ be a cartesian fibration and
$\mathcal{E}\sprime\subseteq\mathcal{E}$ be a full $\infty$-subcategory. If the restriction
\[\xymatrix{
\mathcal{E}\sprime\ar@{^(->}[rr]^{\iota}\ar@{->>}[dr]_{p\sprime} & & \mathcal{E}\ar@{->>}[dl]^	p \\
 & \mathcal{C} & 
}\]
is a cartesian functor of cartesian fibrations, we say that $p\sprime$ is a \emph{full subfibration} of $p$. A full subfibration $p\sprime$ 
of $p$ is \emph{definable} if $\mathcal{C}$ has $(p\sprime\subseteq p)$-comprehension.
\end{definition}

Given a full subfibration $p\sprime\colon\mathcal{E}\sprime\twoheadrightarrow\mathcal{C}$ of
$p\colon\mathcal{E}\twoheadrightarrow \mathcal{C}$, it is easy to see that a morphism of $\mathcal{E}\sprime$ is $p\sprime$-cartesian if and 
only if it is $p$-cartesian. In particular, $p\sprime$ is a full $\mathcal{C}$-subcategory of $p$ in the sense of
\cite[Section 4]{barwicketalparahct}. We say that $p\sprime$ is a full and replete subfibration of $p$ whenever the $\infty$-subcategory
$\mathcal{E}\sprime\subseteq\mathcal{E}$ is 
also replete. By \cite[Lemma 4.5]{barwicketalparahct}, given a cartesian fibration $p\colon\mathcal{E}\twoheadrightarrow\mathcal{C}$ 
and a full and replete $\infty$-subcategory $\mathcal{E}\sprime\subseteq\mathcal{E}$, the restriction of $p$ to
$\mathcal{E}\sprime$ yields a (full and replete) subfibration of $p$ if and only if for every $p$-cartesian morphism $f$ in 
$\mathcal{E}$, the morphism $f$ is contained in $\mathcal{E}\sprime$ whenever its codomain is contained in $\mathcal{E}\sprime$.
It is easy to see that whenever $p\sprime\subseteq p$ is full, the square
\begin{align}\label{diagfullsubfib}
\begin{gathered}
\xymatrix{
\llbracket\Delta^1,\mathcal{E}\sprime\rrbracket\ar[d]\ar@{->>}[r]^{(\delta^1)^{\ast}} \ar@{}[dr]|(.3){\pbs} & \llbracket\partial\Delta^1,\mathcal{E}\sprime\rrbracket\ar[d]\\
\llbracket\Delta^1,\mathcal{E}\rrbracket\ar@{->>}[r]_{(\delta^1)^{\ast}} & \llbracket\partial\Delta^1,\mathcal{E}\rrbracket
}
\end{gathered}
\end{align}
is a pullback of right fibrations.
We further note that inclusions of full subfibrations $p\sprime\subseteq p$ are monomorphisms in the $\infty$-category
$\mathrm{Cart}(\mathcal{C})$, and hence they induce (pointwise fully faithful) monomorphisms
$\mathrm{St}(p\sprime)\hookrightarrow\mathrm{St}(p)$ of associated $\mathcal{C}$-indexed $\infty$-categories as well. Thus, the term 
definability is to be understood as comprehension of substructures.

\begin{remark}
In \cite{benaboufibfound} (and in \cite{streicherfibcats} accordingly), definability is a property of classes of objects 
in a cartesian fibration $\mathcal{E}\twoheadrightarrow\mathcal{C}$. The conditions thereby imposed on such classes
$K\subseteq\mathcal{E}$ of objects are equivalent to the assertion that the full subcategory $\mathcal{K}\subseteq\mathcal{E}$ generated 
by $K$ yields a full and replete subfibration of $p$. Then it is easy to show that the class $K$ is definable in the sense of 
\cite{benaboufibfound} if and only if the full subfibration $\mathcal{K}^{\times}\subseteq\mathcal{E}^{\times}\twoheadrightarrow\mathcal{C}$ 
is definable in the sense of Definition~\ref{defdefsubfib}. 
\end{remark}

As alluded to in the introduction, a full subfibration of the canonical fibration over a 1-category $\mathcal{C}$ 
with pullbacks is essentially the same thing as a set-valued general comprehension scheme over $\mathcal{C}$ (Definition~\ref{defgencomp}).
With the ordinary categorical framework in mind, we make the following definitions.
\begin{definition}\label{defcompcatswatts}
Let $\mathcal{C}$ be an $\infty$-category. A \emph{comprehension $\infty$-category} over $\mathcal{C}$ is a cartesian fibration
$p\colon\mathcal{E}\twoheadrightarrow\mathcal{C}$ together with a functor
\[\xymatrix{
\mathcal{E}\ar@/_1pc/@{->>}[dr]_p\ar[rr]^(.4){f} & & \mathrm{Fun}(\Delta^1,\mathcal{C})\ar@/^1pc/[dl]^t\\
 & \mathcal{C} & 
}\]
over $\mathcal{C}$ which preserves cartesian morphisms. A comprehension $\infty$-category is \emph{full} if its underlying functor
$\mathcal{E}\rightarrow\mathrm{Fun}(\Delta^1,\mathcal{C})$ is fully faithful.

An \emph{$\infty$-category with attributes} over $\mathcal{C}$ is a comprehension $\infty$-category $(p,f)$ over $\mathcal{C}$ such that $p$ 
is a right fibration.
\end{definition}

\begin{definition}
Let $\mathrm{CmpCat}_{\infty}(\mathcal{C})$ be the $\infty$-subcategory of the slice $((\mathrm{Cat}_{\infty})_{/\mathcal{C}})_{/t}$ spanned 
by the comprehension $\infty$-categories over $\mathcal{C}$ and those diagrams $F\colon(p,f)\rightarrow (q,g)$ such that
$F\colon p\rightarrow q$ is a cartesian functor. Let $\mathrm{FCmpCat}_{\infty}(\mathcal{C})\subset\mathrm{CmpCat}_{\infty}(\mathcal{C})$ 
denote the full $\infty$-subcategory spanned by the full comprehension $\infty$-categories.

Let $\mathrm{CwA}_{\infty}(\mathcal{C})\subset\mathrm{CompCat}_{\infty}(\mathcal{C})$ be the full $\infty$-subcategory spanned by the
$\infty$-categories with attributes over $\mathcal{C}$.
 
For a presheaf $Y\in\hat{\mathcal{C}}$, let $\mathrm{Rep}(\mathcal{C})_{/Y}\subset\hat{\mathcal{C}}_{/Y}$ be the full 
$\infty$-category spanned by the general comprehension schemes with base $Y\in\hat{\mathcal{C}}$ on $\mathcal{C}$.
\end{definition}

Whenever $\mathcal{C}$ has pullbacks, the target functor $t\colon\mathrm{Fun}(\Delta^1,\mathcal{C})\twoheadrightarrow\mathcal{C}$ is a 
cartesian fibration itself, and so $\mathrm{CmpCat}_{\infty}(\mathcal{C})$ is thus just the slice $\mathrm{Cart}(\mathcal{C})_{/t}$. 
Analogously, in this case $\mathrm{CwA}_{\infty}(\mathcal{C})$ is the slice $\mathrm{RFib}(\mathcal{C})_{/t^{\times}}$.

\begin{proposition}\label{propaltdefgencomp}
Let $\mathcal{C}$ be a small $\infty$-category and $Y\in\hat{\mathcal{C}}$. Then the $\infty$-categories $\mathrm{Rep}(\mathcal{C})_{/Y}$
and the full $\infty$-subcategory of $\mathrm{CwA}_{\infty}(\mathcal{C})$ spanned by tuples $(p,f)$ such that $p\simeq\mathrm{Un}(Y)$ are 
equivalent.
\end{proposition}
\begin{proof}
For any given presheaf $Y\in\hat{\mathcal{C}}$, the unstraightening $U\colon\mathrm{Un}(Y)\twoheadrightarrow\mathcal{C}$ is the
$\infty$-category of elements of $Y$ in as much as the colimit of the composition
\[\mathrm{Un}(Y)\xrightarrow{U}\mathcal{C}\xrightarrow{y}\hat{\mathcal{C}}\]
is equivalent to $Y$ itself \cite[Lemma 5.1.5.3]{luriehtt}. Hence, by virtue of descent of the $\infty$-topos $\hat{\mathcal{C}}$, the 
restriction functor
\begin{align}\label{equpropaltdefgencomp}
\mathrm{res}\colon\hat{\mathcal{C}}_{/Y}\rightarrow\mathrm{Fun}(\mathrm{Un}(Y),\hat{\mathcal{C}})\downarrow^{\times}yU
\end{align}
given componentwise by pullback along the elements $y(C)\rightarrow Y$ is part of an equivalence (with inverse given by the colimit 
functor), where the right hand side denotes the $\infty$-category of cartesian natural transformations over the functor $yU$
\cite[Section 3.3]{aneljoyaltopos}. The $\infty$-category $\mathrm{Fun}(\mathrm{Un}(Y),\hat{\mathcal{C}})_{/yU}$ of natural transformations 
over $yU$ is canonically equivalent to the $\infty$-category of functors from $yU\colon\mathrm{Un}(Y)\rightarrow\hat{\mathcal{C}}$ to
$t\colon\hat{\mathcal{C}}^{\Delta^1}\twoheadrightarrow\hat{\mathcal{C}}$ over $\hat{\mathcal{C}}$. Accordingly, the codomain of 
(\ref{equpropaltdefgencomp}) as a full $\infty$-subcategory of $\mathrm{Fun}(\mathrm{Un}(Y),\hat{\mathcal{C}})_{/yU}$ is in turn equivalent 
to the $\infty$-category of functors from $yU\colon\mathrm{Un}(Y)\rightarrow\hat{\mathcal{C}}$ to the right fibration
$t^{\times}\colon\mathrm{Fun}(\Delta^1,\hat{\mathcal{C}})^{\times}\twoheadrightarrow\hat{\mathcal{C}}$ over $\hat{\mathcal{C}}$.
\begin{align}\label{diagpropaltdefgencomp}
\begin{gathered}
\xymatrix{
 & & \mathrm{Fun}(\Delta^1,\mathcal{C})^{\times}\ar@/^/[ddl]|(.47)\hole^(.6)t\ar@{^(->}[rrr]^y  & & & \mathrm{Fun}(\Delta^1,\hat{\mathcal{C}})^{\times}\ar@{->>}@/^/[ddl]^t\\
\mathrm{Un}(Y)\ar@{=}[rrr]\ar@{->>}@/_/[dr]_{U}\ar@{-->}[urr] & & & \mathrm{Un}(Y)\ar@/_/[dr]^{yU}\ar@{-->}[urr] & &  \\
 & \mathcal{C}\ar@{^(->}[rrr]_y & & & \hat{\mathcal{C}} & 
}
\end{gathered}
\end{align}
A natural transformation $f\colon X\rightarrow Y$ in $\hat{\mathcal{C}}$ is representable if and only if the associated functor
$\mathrm{Un}(Y)\rightarrow\mathrm{Fun}(\Delta^1,\mathcal{C})^{\times}$ as depicted by a dotted arrow on the right of diagram 
(\ref{diagpropaltdefgencomp}) lifts along
$y\colon\mathrm{Fun}(\Delta^1,\mathcal{C})^{\times}\rightarrow\mathrm{Fun}(\Delta^1,\hat{\mathcal{C}})^{\times}$ to a functor over
$\mathcal{C}$ as depicted by a dotted arrow on the left of the diagram. We hence obtain an equivalence between the $\infty$-category of 
general comprehension schemes with base $Y$ on $\mathcal{C}$ on the one hand, and $\infty$-categories with attributes of the form
$(\mathrm{Un}(Y),f)$ over $\mathcal{C}$ on the other. 
\end{proof}

Clearly, we furthermore have a functor
$(\cdot)^{\times}\colon\mathrm{CmpCat}_{\infty}(\mathcal{C})\rightarrow\mathrm{CwA}_{\infty}(\mathcal{C})$. Its restriction to 
$\mathrm{FCmpCat}_{\infty}(\mathcal{C})$ comes with a retraction $r$ given by factorization of an underlying functor
$\mathcal{E}\rightarrow\mathrm{Fun}(\Delta^1,\mathcal{C})$ through an essentially surjective functor followed by a fully faithful and replete 
functor. The pair 
\begin{align}\label{equcompcatcore}
\xymatrix{
(\cdot)^{\times}\colon\mathrm{FCmpCat}_{\infty}(\mathcal{C})\ar@<.5ex>@{^(->}[r] & \mathrm{CwA}_{\infty}(\mathcal{C})\colon r \ar@<.5ex>[l]
}
\end{align}
however does not form an equivalence of $\infty$-categories; rather, full comprehension $\infty$-categories are particularly simple
$\infty$-categories with attributes in which the entire higher structure of the homotopy types of the attributes are fully determined by
$\mathcal{C}$. We elaborate on this in the following remark and the subsequent proposition.

\begin{remark}\label{remaltdefgencomp}
In 1-category theory, the abstract equivalence of categories with attributes over a category $\mathcal{C}$ (together with a fixed discrete 
fibration $p\colon\mathcal{E}\twoheadrightarrow\mathcal{C}$) and full comprehension categories over $\mathcal{C}$ (together with a fixed 
cartesian fibration $p\colon\mathcal{E}\twoheadrightarrow\mathcal{C}$) is a formal consequence of the existence of the
(bijective-on-objects, fully faithful)-factorization system on $\mathrm{Cat}$. See e.g.\ \cite[Lemma 4.9, Example 4.10]{jacobscompcats}. 
Such a factorization system (with respect to an accordingly univalent notion of ``identity-on-objects'') does not exist in
$\mathrm{Cat}_{\infty}$\footnote{As already noted in Remark~\ref{reminftysmall}, this factorization system on $\mathrm{Cat}$
relies exactly on the very specific kind of equality principle which in \cite[Paragraph 8.7]{benaboufibfound} is considered to be entirely 
meta-theoretical rather than category theoretical.}; and in fact it is 
not hard to show that $\infty$-categories with attributes and full comprehension $\infty$-categories over a given $\infty$-category
$\mathcal{C}$ are generally non-equivalent structures.
\end{remark}

\begin{proposition}\label{propccatcwanot}
The $\infty$-categories $\mathrm{FCmpCat}_{\infty}(\mathcal{C})$ and $\mathrm{CwA}_{\infty}(\mathcal{C})$ are generally not equivalent.
\end{proposition}
\begin{proof}
Consider the left exact $\infty$-category $\mathcal{C}=\ast$. The $\infty$-category
$\mathrm{CwA}_{\infty}(\ast)\simeq\mathrm{RFib}(\ast)_{/1_{\ast}}$ is just the $\infty$-category $\mathcal{S}$ of spaces.
The full $\infty$-subcategory of $\mathrm{CompCat}_{\infty}(\ast)\simeq(	\mathrm{Cat}_{\infty})_{/1_{\ast}}$ spanned by the fully faithful 
functors however is the $\infty$-category of locally contractible $\infty$-categories (which are exactly the locally small cartesian 
fibrations over the point by Example~\ref{exmplecntrblty}). But a locally contractible $\infty$-category is automatically either empty or a 
contractible space, and so the homotopy category of locally contractible $\infty$-categories has only two objects up to isomorphism.
Hence, the two structures over a point are not equivalent.
\end{proof}

Proposition~\ref{propccatcwanot} can be shown even more concretely over a fixed base $Y\in\hat{\mathcal{C}}$: let $\mathcal{C}=\ast$ and $Y$ 
be any non-empty and non-contractible space considered as an $\infty$-category with attributes over the point. Then the full subcategory of
$\mathrm{CwA}_{\infty}(\ast)$ spanned by pairs $(X,f)$ such that $X\simeq Y$ is contractible. But the $\infty$-category of locally 
contractible $\infty$-categories such that the core $\mathcal{C}^{\simeq}=\mathcal{C}$ is equivalent to $Y$ is empty.

\begin{remark}
The proof of Proposition~\ref{propccatcwanot} in ordinary category theory computes up to equivalence the category of sets on the one hand, 
and the category of complete graphs on the other hand. These however are indeed isomorphic.
\end{remark}

In fact, via Proposition~\ref{propaltdefgencomp} one can characterize the full comprehension $\infty$-categories among all
$\infty$-categories with attributes exactly as those representable natural transformations that are univalent (as discussed in 
Section~\ref{secextuniv}). 
Therefore, we note that the proof of Proposition~\ref{propaltdefgencomp} in fact shows that the ``universe of sections'' associated to a 
small $\infty$-category $\mathcal{C}$ with pullbacks is the universal comprehension scheme over $\mathcal{C}$ as we show in the following
$\infty$-categorical version of \cite[Proposition 2.7]{shulmanuniverses}.

\begin{proposition}\label{propunivreptransf}
Suppose $\mathcal{C}$ is a small $\infty$-category with pullbacks. Then there is a representable natural transformation
$\pi_{\mathcal{C}}\colon(\mathcal{C}_{/(\cdot)})^{\simeq}_{\ast}\rightarrow(\mathcal{C}_{/(\cdot)})^{\simeq}$ in $\hat{\mathcal{C}}$ such that for every $Y\in\hat{\mathcal{C}}$, the functor
\[(\cdot)^{\ast}\pi_{\mathcal{C}}\colon\hat{\mathcal{C}}(Y,(\mathcal{C}_{/(\cdot)})^{\simeq})\rightarrow(\mathrm{Rep}(\mathcal{C})_{/Y})^{\simeq}\]
which maps a natural transformation $f\colon Y\rightarrow(\mathcal{C}_{/(\cdot)})^{\simeq}$ to the pullback
$f^{\ast}\pi_{\mathcal{C}}\in\hat{\mathcal{C}}_{/Y}$ is an equivalence of $\infty$-categories.
\end{proposition}
\begin{proof}
By definition, the canonical indexing $(\mathcal{C}_{/(\cdot)})^{\simeq}\in\hat{\mathcal{C}}$ is the straightening of the right fibration
$t^{\times}\colon(\mathrm{Fun}(\Delta^1,\mathcal{C})^{\times}\twoheadrightarrow\mathcal{C}$. It thus follows again from
\cite[Lemma 5.1.5.3]{luriehtt} that the colimit of the composition
\[\mathrm{Fun}(\Delta^1,\mathcal{C})^{\times}\overset{t^{\times}}{\twoheadrightarrow}\mathcal{C}\xrightarrow{y}\hat{\mathcal{C}}\]
is equivalent to the presheaf $(\mathcal{C}_{/(\cdot)})^{\simeq}$ in $\hat{\mathcal{C}}$. In particular, the colimit of the functor
\begin{align}\label{equpropunivreptransf}
\mathrm{Fun}(\Delta^1,y)\colon\mathrm{Fun}(\Delta^1,\mathcal{C})^{\times}\rightarrow\mathrm{Fun}(\Delta^1,\hat{\mathcal{C}})^{\times}
\end{align}
is a natural transformation over $(\mathcal{C}_{/(\cdot)})^{\simeq}$. We denote this colimit by
$\pi_{\mathcal{C}}\colon(\mathcal{C}_{/(\cdot)})^{\simeq}_{\ast}\rightarrow(\mathcal{C}_{/(\cdot)})^{\simeq}$. Verification of all the 
claims is now a routine exercise in the application of descent properties of $\infty$-toposes. Indeed, we first observe that the functor
\begin{align}\label{equpropunivreptransf1}
(\cdot)^{\ast}\pi_{\mathcal{C}}\colon\hat{\mathcal{C}}(Y,(\mathcal{C}_{/(\cdot)})^{\simeq})\rightarrow(\hat{\mathcal{C}}_{/Y})^{\simeq}
\end{align}
induced by pullback of the natural transformation $\pi_{\mathcal{C}}$ is fully faithful (i.e.\ $\pi_{\mathcal{C}}$ is a univalent morphism in 
$\hat{\mathcal{C}}$ as discussed in Section~\ref{secextuniv}). By virtue of descent of $\hat{\mathcal{C}}$ and the fact that the class of 
fully faithful functors between spaces is closed under small limits, it suffices to show this for $Y=y(C)$ a representable presheaf. In this 
case the functor (\ref{equpropunivreptransf1}) is simply the (core of the) Yoneda embedding
\[y\colon(\mathcal{C}_{/C})^{\simeq}\rightarrow(\widehat{\mathcal{C}_{/C}})^{\simeq}\]
under the canonical equivalence $\hat{\mathcal{C}}_{/y(C)}\simeq\widehat{\mathcal{C}_{/C}}$ given by the Yoneda lemma. In particular, it is 
fully faithful.

Second, the natural transformation $\pi_{\mathcal{C}}$ is representable as for every $C\in\mathcal{C}$ and every natural transformation
$\ulcorner f\urcorner \colon y(C)\rightarrow(\mathcal{C}_{/(\cdot)})^{\simeq}$, the canonical square
\[\xymatrix{
y(\mathrm{dom}f)\ar[r]\ar[d]_{yf} & (\mathcal{C}_{/(\cdot)})^{\simeq}_{\ast}\ar[d]^{\pi_{\mathcal{C}}}\\
y(C)\ar[r]_(.4){\ulcorner f\urcorner} & (\mathcal{C}_{/(\cdot)})^{\simeq}
}\]
given by the cocone that defines $\pi_{\mathcal{C}}$ as the colimit of (\ref{equpropunivreptransf}) is cartesian by virtue of the fact that 
(\ref{equpropunivreptransf}) is a cartesian natural transformation in $\hat{\mathcal{C}}$, and by the fact that colimits in
$\hat{\mathcal{C}}$ are effective. Thus, since representability of natural transformations is a pullback-stable property 
(Lemma~\ref{lemmarepnattransfbasics}.1), the fully faithful functor (\ref{equpropunivreptransf1}) factors through the full subspace
$(\mathrm{Rep}(\mathcal{C})_{/Y})^{\simeq}$. We are left to show that this functor is essentially surjective; but this is exactly the content 
of Proposition~\ref{propaltdefgencomp}. More precisely, given a representable natural transformation $f\colon X\rightarrow Y$ in
$\hat{\mathcal{C}}$, let $\mathrm{res}(f)\colon \mathrm{Un}(Y)\rightarrow\mathrm{Fun}(\Delta^1,\mathcal{C})^{\times}$ over $\mathcal{C}$ be 
its associated $\infty$-category with attributes from Proposition~\ref{propaltdefgencomp} defined via the equivalence 
(\ref{equpropaltdefgencomp}). Diagram (\ref{diagpropaltdefgencomp}) induces a 
square in $\hat{\mathcal{C}}$ from the colimit of the composition
$\mathrm{res}(f)\colon\mathrm{Un}(Y)\rightarrow\mathrm{Fun}(\Delta^1,\hat{\mathcal{C}})^{\times}$ to the colimit of the inclusion
$\mathrm{Fun}(\Delta^1,y)\colon \mathrm{Fun}(\Delta^1,\mathcal{C})^{\times})\rightarrow\mathrm{Fun}(\Delta^1,\hat{\mathcal{C}})^{\times})$.
That is, a square in $\hat{\mathcal{C}}$ from $f$ to $\pi_{\mathcal{C}}$.	This square is cartesian again by virtue of descent of
$\hat{\mathcal{C}}$.
\end{proof}

\begin{corollary}\footnote{The observation of this corollary is due to Jonas Frey; I personally had missed this.}
Let $\mathcal{C}$ be a small $\infty$-category and $Y\in\hat{\mathcal{C}}$. Then the equivalence of Proposition~\ref{propaltdefgencomp} 
restricts to an equivalence between the full $\infty$-subcategory spanned by the full comprehension $\infty$-categories via 
(\ref{equcompcatcore}) on the one hand, and the full $\infty$-subcategory spanned by the univalent representable natural transformations over 
$Y$ on the other hand.
\end{corollary}
\begin{proof}
For any given representable natural transformation $f\colon X\rightarrow Y$ in $\hat{\mathcal{C}}$ there is an essentially unique classifying 
morphism $\ulcorner f\urcorner\colon Y\rightarrow(\mathcal{C}_{/(\cdot)})^{\simeq}$ by Proposition~\ref{propunivreptransf}. This classifying 
morphism is by construction the value of the colimit functor applied to the morphism
\[\mathrm{res}_f\colon yU\rightarrow yt\]
of functors over $\hat{\mathcal{C}}$, where $U\colon\mathrm{Un}(Y)\twoheadrightarrow\mathcal{C}$ denotes the unstraightening of $Y$. This 
colimit computes the straightening $\mathrm{St}(\mathrm{res}_f)\colon Y\rightarrow(\mathcal{C}_{/(\cdot)})^{\simeq}$, essentially because the 
Yoneda embedding left Kan extends to the identity along itself \cite[Lemma 5.1.5.3]{luriehtt}. It follows in turn that the fibered functor
$\mathrm{res}_f\colon U\rightarrow t$ over $\mathcal{C}$ is the unstraightening of the classifying morphism $\ulcorner f\urcorner$. 

Now, as the universal representable natural transformation $\pi_{\mathcal{C}}$ is univalent, the transformation 
$f$ itself is univalent if and only if its classifying morphism $\ulcorner f\urcorner$ is monic \cite[Proposition 2.5]{rasekhunivalence}.  
This in turn holds if and only if its 
unstraightening $\mathrm{res}_f\colon U\rightarrow t$ is monic, which is to say that the associated $\infty$-category with attributes
$\mathrm{res}_f\colon U\rightarrow t$ is (the core of) a full comprehension $\infty$-category.
\end{proof}

We return to the notion of definability introduced in Definition~\ref{defdefsubfib}. We first note that definability of a full subfibration 
can be reduced to definability of its associated right fibration of objects in the following sense.

\begin{lemma}\label{lemmadefreduct}
Let $p\colon\mathcal{E}\twoheadrightarrow\mathcal{C}$ and $p\sprime\colon\mathcal{E}\sprime\twoheadrightarrow\mathcal{C}$ be cartesian 
fibrations and let $f\colon p\sprime\rightarrow p$ be a cartesian functor over $\mathcal{C}$. 
\begin{enumerate}
\item Suppose $\mathcal{C}$ has pullbacks and $f\colon p\sprime\rightarrow p$ is the inclusion of a full subfibration. Then $p\sprime$ is 
definable if and only if the inclusion $\mathrm{St}(p\sprime)\hookrightarrow\mathrm{St}(p)$ is representable.
\item Suppose $p$ and $p\sprime$ are right fibrations. Then $\mathcal{C}$ has $f$-comprehension if and only if the natural transformation
$\mathrm{St}(p\sprime)\rightarrow\mathrm{St}(p)$ is representable.
\end{enumerate}
\end{lemma}
\begin{proof}
For Part 1 we note that the pullback (\ref{diagfullsubfib}) induces a pullback
\[\xymatrix{
(\mathrm{St}(p\sprime)^{\Delta^1})^{\simeq}\ar[d]\ar[r]^(.4){(d_1,d_0)^{\ast}}\ar@{}[dr]|(.3){\pbs}& (\mathrm{St}(p\sprime))^{\simeq}\times (\mathrm{St}(p\sprime))^{\simeq}\ar[d] \\
(\mathrm{St}(p)^{\Delta^1})^{\simeq}\ar[r]_(.4){(d_1,d_0)^{\ast}} & (\mathrm{St}(p))^{\simeq}\times(\mathrm{St}(p))^{\simeq} 
}\]
of presheaves in $\hat{\mathcal{C}}$. Thus, to show that the vertical natural transformation on the left hand side is representable, it 
suffices to show that the right hand side is representable. By assumption however the natural transformation
$(\mathrm{St}(p\sprime))^{\simeq}\rightarrow(\mathrm{St}(p))^{\simeq}$ is representable indeed. It follows that its product is representable  
whenever $\mathcal{C}$ has pullbacks by Lemma~\ref{lemmarepnattransfbasics}.2.
For Part 2, if $p$ and $p\sprime$ are  right fibrations the vertical morphisms in the diagram
\[\xymatrix{
\mathrm{St}(p\sprime)^{\Delta^1}\ar[r]& \mathrm{St}(p)^{\Delta^1} \\
\mathrm{St}(p\sprime)\ar[r]\ar[u]^{(s^0)^{\ast}} & \mathrm{St}(p)\ar[u]_{(s^0)^{\ast}} 
}\]
are natural equivalences, and it follows that the top horizontal natural transformation is representable if and only if the bottom one is.
\end{proof}

In Example~\ref{explerepslenattransfoverpt} we stated that an $\infty$-category $\mathcal{C}$ has finite products if and only if 
representability of a presheaf $X\in\hat{\mathcal{C}}$ is equivalent to representability of the natural transformation $X\rightarrow\ast$. 
This is a special case of the equally straightforward fact that $\mathcal{C}$ has pullbacks (with base $C\in\mathcal{C}$) if and only if 
representability of a natural transformation in $\hat{\mathcal{C}}$ over a representable presheaf (represented by $C$) is equivalent to 
representability of its domain. In the rest of this section, we prove a direct generalization of (one half of) this observation to 
$\mathcal{C}$-indexed $\infty$-categories: if $\mathcal{C}$ has finite limits, then a comprehensible cartesian functor over a small
$\mathcal{C}$-indexed $\infty$-category of the form $\mathrm{Ext}(Y)$ is the same thing as an internal functor over $Y$.
This will be formally stated in Proposition~\ref{propintfunctorscomp}.

\begin{lemma}\label{propdeffibs}
Let $\mathcal{C}$ be an $\infty$-category. Let $p\colon\mathcal{E}\twoheadrightarrow\mathcal{C}$ and
$p\sprime\colon\mathcal{E}\sprime\twoheadrightarrow\mathcal{C}$ be cartesian fibrations and let $f\colon p\sprime\rightarrow p$ be a 
cartesian functor over $\mathcal{C}$ such that $\mathcal{C}$ has $f$-comprehension.
\begin{enumerate}
\item  If $p$ is a globally small fibration, then so is $p\sprime$.
\item Suppose $\mathcal{C}$ has pullbacks. If $p$ is a small fibration, then so is $p\sprime$.
\end{enumerate}
\end{lemma}
\begin{proof}
In Part 1 the natural transformations $\mathrm{St}(p)^{\simeq}\rightarrow\ast$ and
$\mathrm{St}(p\sprime)^{\simeq}\rightarrow\mathrm{St}(p)^{\simeq}$ in $\hat{\mathcal{C}}$ are representable by assumption. This implies that 
the composition $\mathrm{St}(p\sprime)^{\simeq}\rightarrow\ast$ is representable by Lemma~\ref{lemmarepnattransfbasics}.1. Hence, the 
fibration $p\sprime$ is globally small by definition. For Part 2 consider the square
\begin{align}\label{diagpropdeffibs}
\begin{gathered}
\xymatrix{
(\mathrm{St}(p\sprime)^{\Delta^1})^{\simeq}\ar[d]_{\delta^{\ast}}\ar[r] & (\mathrm{St}(p)^{\Delta^1})^{\simeq}\ar[d]^{\delta^{\ast}}\\
(\mathrm{St}(p\sprime)^{\partial\Delta^1})^{\simeq}\ar[r] & (\mathrm{St}(p)^{\partial\Delta^1})^{\simeq} 
}
\end{gathered}
\end{align}
in $\hat{\mathcal{C}}$. The right vertical natural transformation and the top horizontal natural transformation are representable by 
assumption. The bottom horizontal natural transformation is representable by the assumption of $f$-comprehension and by the fact the class of 
representable natural transformation in $\hat{\mathcal{C}}$ is closed under finite limits (Lemma~\ref{lemmarepnattransfbasics}.2). Since 
representable natural transformations are furthermore left cancellable, it follows that the left vertical natural transformation is 
representable as well.
\end{proof}

In fact, whenever $p\sprime\subseteq p$ is full, the assumption of comprehension can be dropped in the case of local smallness alone.

\begin{lemma}\label{lemmadeffibs}
Let $\mathcal{C}$ be an $\infty$-category. A full subfibration of a locally small fibration over $\mathcal{C}$ is again locally small.
\end{lemma}
\begin{proof}
Given a fibration $p$ over $\mathcal{C}$ together with some full subfibration $p\sprime\subseteq p$, the statement follows directly from 
Lemma~\ref{lemmarepnattransfbasics}.1 and the fact that the square (\ref{diagpropdeffibs}) is the straightening of the cartesian square 
(\ref{diagfullsubfib}) by Proposition~\ref{propstrrect} and hence is cartesian itself.
\end{proof}

In the converse direction, we have the following.

\begin{lemma}\label{propdeffibs2}
Let $\mathcal{C}$ be an $\infty$-category with pullbacks. Suppose $f\colon p\sprime\rightarrow p$ is a cartesian functor between small 
cartesian fibrations over $\mathcal{C}$. Then $\mathcal{C}$ has $f$-comprehension.
\end{lemma}

\begin{proof}
First, consider the sequence
\[\mathrm{St}(p\sprime)^{\simeq}\rightarrow\mathrm{St}(p)^{\simeq}\rightarrow\ast\]
of natural transformations in $\hat{\mathcal{C}}$. The composition and the second natural transformation are representable by assumption. It 
follows that so is the first by Lemma~\ref{lemmarepnattransfbasics}.2. Second, consider the square (\ref{diagpropdeffibs}) in
$\hat{\mathcal{C}}$ again. The two vertical natural transformations are representable by assumption. By virtue of global smallness, we 
have seen that the natural transformation $\mathrm{St}(p\sprime)\rightarrow\mathrm{St}(p)$ is representable as well. It follows that
the bottom vertical transformation is representable, too, since the class of representable natural transformation in $\hat{\mathcal{C}}$ is 
closed under finite limits (again by Lemma~\ref{lemmarepnattransfbasics}.2). Thus, the top vertical morphism in the square is representable 
by the left cancellation property stated in the same lemma.
\end{proof}

\begin{proposition}\label{propintfunctorscomp}
Let $\mathcal{C}$ be an $\infty$-category with finite limits, let $Y$ be a complete Segal object in $\mathcal{C}$, and let
$f\colon p\rightarrow\mathrm{Ext}(Y)$ be a cartesian functor over $\mathcal{C}$. Then $p$ is small if and only if $\mathcal{C}$ has
$f$-comprehension. In particular, $\mathcal{C}$ has $\mathrm{Ext}(f)$-comprehension for all internal functors $f\colon X\rightarrow Y$ in
$\mathrm{CS}(\mathcal{C})$.
\end{proposition}

\begin{proof}
Immediate by Lemma~\ref{propdeffibs}, Lemma~\ref{propdeffibs2} and Theorem~\ref{thmsmall=ext}.
\end{proof}

\section{Univalence and smallness}\label{secextuniv}

We apply the results of the last two sections to the special case of full comprehension $\infty$-categories over locally cartesian closed
$\infty$-categories $\mathcal{C}$, and find that such are globally small if and only if they are the externalization of the nerve of a 
univalent morphism in $\mathcal{C}$. Against this background, we will end the section with a short discussion of higher elementary 
toposes in terms of comprehension schemes.\\

For the rest of this section suppose $\mathcal{C}$ is a left exact locally cartesian closed $\infty$-category. For every morphism
$q\colon E\rightarrow B$ in $\mathcal{C}$ we can associate the full and replete comprehension $\infty$-category
$F_q\subseteq\mathrm{Fun}(\Delta^1,\mathcal{C})\overset{t}{\twoheadrightarrow}\mathcal{C}$ where the fiber $F_q(C)\subseteq\mathcal{C}_{/C}$ 
consists of those morphisms with codomain $C$ which arise as pullback of $q$ along some morphism $C\rightarrow B$.

\begin{lemma}\label{lemmaunivalentcomprehensioncats}
For every morphism $q\in\mathcal{C}$, the cartesian fibration $F_q\twoheadrightarrow\mathcal{C}$ is locally small.
\end{lemma}
\begin{proof}
Immediate by Proposition~\ref{exmplelocsmall} and Proposition~\ref{propdeffibs}.2.
\end{proof}

A morphism $q\colon E\rightarrow B$ in $\mathcal{C}$ is \emph{univalent} if it is a $(-1)$-truncated object in the 
core $\mathrm{Fun}(\Delta^1,\mathcal{C})^{\times}$ \cite[Definition 2.1]{rasekhunivalence}. Equivalently, that is whenever the object
$q\in (F_q)^{\times}$ is terminal. By Lemma~\ref{lemmareprightfibs} in turn, this means that $q$ is univalent if and only 
if the functor $\ulcorner q\urcorner\colon \mathcal{C}_{/B}\rightarrow(F_q)^{\times}$ of right fibrations given by $1_B\mapsto q$ via the 
Yoneda lemma is an equivalence.

\begin{proposition}\label{propunivalentcomprehensioncats}
Given a full and replete comprehension $\infty$-category $J$ over $\mathcal{C}$, the following are equivalent.
\begin{enumerate}
\item $J\twoheadrightarrow\mathcal{C}$ is globally small.
\item There is a complete Segal space $\mathcal{N}(q)\in\mathrm{CS}(\mathcal{C})$ such that
$\mathrm{Ext}(\mathcal{N}(q))\simeq J$.
\item There is a univalent morphism $q\colon E\rightarrow B$ in $\mathcal{C}$ such that $J=F_q$.
\end{enumerate}
\end{proposition}

\begin{proof}
Parts 1 and 2 are equivalent by Theorem~\ref{thmsmall=ext} and Lemma~\ref{lemmaunivalentcomprehensioncats}.
If $J$ is globally small, its core $J^{\times}\twoheadrightarrow\mathcal{C}$ is representable by definition and hence has a terminal object 
$(q\colon E\rightarrow B)\in J^{\times}$. It follows that $J=F_q$ since $J\subseteq\mathrm{Fun}(\Delta^1,\mathcal{C})$ is full and replete. 
In particular, $q$ is terminal in $(F_q)^{\times}$ and so $q$ is univalent. Vice versa, whenever $J=F_q$ for a univalent morphism $q$ in
$\mathcal{C}$, then $q\in J^{\times}$ is terminal by definition, and so $\mathrm{St}(J^{\times})$ is representable. That means $J$ is 
globally small.
\end{proof}

The equivalence of Parts 2 and 3 has a non-univalent 1-categorical analogon as well, see \cite[Section 6]{streetintcat} and
\cite[Proposition 7.3.6]{jacobsttbook}. In the case that $\mathcal{C}$ is a presentable $\infty$-category, Gepner and Kock have characterized 
univalence of morphisms $q$ in $\mathcal{C}$ in terms of a sheaf property of the $\mathcal{C}$-indexed $\infty$-category $\mathrm{St}(F_q)$
\cite[Proposition 3.8]{gepnerkock}. This can be directly generalized as follows.

\begin{corollary}\label{corcharunivalencela}
Suppose $\mathcal{C}$ is locally cartesian closed and complete and $q\colon E\rightarrow B$ is a morphism in $\mathcal{C}$. Then 
$q\in\mathrm{Fun}(\Delta^1,\mathcal{C})^{\times}$ has a retract $p$ which is a univalent morphism in $\mathcal{C}$ if and only if the indexed
$\infty$-category $\mathrm{St}(F_q)\colon\mathcal{C}^{op}\rightarrow\mathrm{Cat}_{\infty}$ has a left adjoint. In particular, the functor
$\mathrm{St}(F_q)\colon\mathcal{C}^{op}\rightarrow\mathrm{Cat}_{\infty}$ preserves limits whenever $q$ has a univalent retract in
$\mathrm{Fun}(\Delta^1,\mathcal{C})^{\times}$.
\end{corollary}

\begin{proof}
If $q$ has a retract $p$ in $\mathrm{Fun}(\Delta^1,\mathcal{C})^{\times}$, then $F_q= F_p$. Furthermore, whenever $p$ is univalent, the 
straightening $\mathrm{St}(F_p)\colon\mathcal{C}^{op}\rightarrow\mathrm{Cat}_{\infty}$ has a left adjoint by 
Proposition~\ref{propunivalentcomprehensioncats} and Proposition~\ref{proprightadj}. Vice versa, whenever
$\mathrm{St}(F_q)\colon\mathcal{C}^{op}\rightarrow\mathrm{Cat}_{\infty}$ has a left adjoint, it is globally small again by 
Proposition~\ref{proprightadj}. By Proposition~\ref{propunivalentcomprehensioncats} it follows that there is a univalent morphism $p$ in
$\mathcal{C}$ such that $F_{q}=F_{p}$. Thus, both morphisms $p$ and $q$ are mutually pullbacks of one another, which yield morphisms
$q\rightarrow p$ and $p\rightarrow q$ in $(F_p)^{\times}$. As $p\in(F_p)^{\times}$ is terminal, it follows that $p$ is a retract of $q$
in $(F_q)^{\times}\subset\mathrm{Fun}(\Delta^1,\mathcal{C})^{\times}$.
\end{proof}

\begin{remark}
Given a cartesian fibration $\mathcal{E}\twoheadrightarrow\mathcal{C}$, one may define in this generality an object $x\in\mathcal{E}$ to 
be univalent if it is $(-1)$-truncated in the core $\mathcal{E}^{\times}$. Accordingly, one can generalize various related constructions. For 
instance, the notion of univalent completion as defined in \cite{univalentcompletion} and studied in more generality in
\cite[Section 5]{rs_uc} can be defined for any object $x\in\mathcal{E}$ as its $(-1)$-truncation in the core $\mathcal{E}^{\times}$ whenever 
it exists. 
\end{remark}

\begin{remark}
The fact that smallness of $F_q$ only implies the existence of a univalent completion $q\rightarrow p$ (with a section) in
$\mathcal{C}$ rather than univalence of $q$ itself is a special case of the fact that smallness of $\mathrm{Ext}(X)$ for a given Segal object 
$X$ in $\mathcal{C}$ only implies the existence of a Segal completion $X\rightarrow\rho(X)$ in $\mathcal{C}$ rather than completeness of $X$ 
itself (Remark~\ref{reminftysmall2}). Indeed, for every morphism $q$ in $\mathcal{C}$ one can construct a Segal object $\mathcal{N}(q)$ in
$\mathcal{C}$ which is complete if and only if $q$ is univalent \cite{rasekhunivalence}.
\end{remark}

In Remark~\ref{remtoposcomp} we discussed criteria towards a potential classification of elementary $\infty$-toposes via 
comprehension schemes. Against the background of the results in this section, let us finish this paper with a review of this discussion. 
In Remark~\ref{remtoposcomp} we argued that a presentable $\infty$-category is an $\infty$-topos if and only if its associated canonical 
indexing is locally small and $\kappa$-super powered for all sufficiently large regular cardinals $\kappa$. To generalize this criterion to
non-presentable $\infty$-categories as well, we can re-parametrize the associated class of comprehension schemes imposed on the canonical 
indexing as a single comprehension scheme imposed on a class of canonically associated full comprehension $\infty$-categories as follows. 
Assume for the moment being again that $\mathcal{C}$ is presentable. Then $\mathcal{C}$ is an $\infty$-topos if and only if colimits in
$\mathcal{C}$ are universal and its canonical indexing
\[\mathcal{C}_{/(\cdot)}\colon\mathcal{C}^{op}\rightarrow\mathrm{CAT}_{\infty}\]
preserves limits \cite[Theorem 6.1.3.9]{luriehtt}. If this canonical indexing in fact had a left adjoint, it would be small by 
Proposition~\ref{proprightadj}. Vice versa, every presentable $\infty$-category $\mathcal{C}$ with a small canonical indexing is an
$\infty$-topos. However, although the canonical indexing may be limit preserving, it cannot have a left adjoint if $\mathcal{C}$ is large as 
this would generate a single univalent universe in $\mathcal{C}$ via Proposition~\ref{propunivalentcomprehensioncats} (given that such a
$\mathcal{C}$ is automatically locally cartesian closed).
Yet, the underlying idea that an $\infty$-topos is a presentable $\infty$-category which is small over itself holds up to local 
approximations in the following sense.

Given a regular cardinal $\kappa$, let $\mathrm{Fun}(\Delta^1,\mathcal{C})_{\kappa}\subset\mathrm{Fun}(\Delta^1,\mathcal{C})$ be the full 
and replete $\infty$-subcategory generated by the relative $\kappa$-compact morphisms in $\mathcal{C}$ and consider its associated full 
comprehension $\infty$-category $t_{\kappa}\colon\mathrm{Fun}(\Delta^1,\mathcal{C})_{\kappa}\twoheadrightarrow\mathcal{C}$.

\begin{corollary}\label{chartopssmall}
A presentable $\infty$-category $\mathcal{C}$ is an $\infty$-topos if and only if the following two conditions hold.
\begin{enumerate}
\item The canonical fibration $t\colon\mathrm{Fun}(\Delta^1,\mathcal{C})\twoheadrightarrow\mathcal{C}$ is locally small.
\item For all sufficiently large regular cardinals $\kappa$, the full comprehension $\infty$-category
$t_{\kappa}\colon\mathrm{Fun}(\Delta^1,\mathcal{C})_{\kappa}\twoheadrightarrow\mathcal{C}$
is small.
\end{enumerate}
\end{corollary}
\begin{proof}
Immediate by Proposition~\ref{propunivalentcomprehensioncats} and \cite[Theorem 6.1.6.8]{luriehtt}.
\end{proof}

Whenever $\mathcal{C}$ is an $\infty$-topos, we can furthermore show relative smallness of all its $\infty$-subtoposes 
(in fact, of all its modalities in the sense of  \cite[Definition 3.2.1]{abfjsheavesI}). A version for 1-toposes of this fact can be found in
\cite[Proposition 9.6.9]{jacobsttbook}. Therefore, we note that whenever $\mathcal{C}$ is an $\infty$-category with pullbacks, then
pullback-stable classes $S$ of morphisms in $\mathcal{C}$ stand in 1-1 correspondence with full comprehension $\infty$-categories
$s\subseteq t$ over $\mathcal{C}$. A pullback-stable class $S$ is said to be \emph{local} if the induced $\mathcal{C}$-indexed
$\infty$-category $S_{/(\cdot)}:=\mathrm{St}(s)\colon\mathcal{C}^{op}\rightarrow\mathrm{CAT}_{\infty}$ preserves small limits
\cite[Definition 6.1.3.8]{luriehtt}.

\begin{corollary}\label{proplocaldefty}
Suppose $\mathcal{C}$ is an $\infty$-topos and $S$ is a pullback-stable class of morphisms in $\mathcal{C}$. Then its associated
full comprehension $\infty$-category $s\subseteq t$ is definable if and only if $S$ is a local class.
In particular, every $\infty$-subtopos of $\mathcal{C}$ (in fact every modality) considered as a 
subfibration of $t\colon\mathrm{Fun}(\Delta^1,\mathcal{C})\twoheadrightarrow\mathcal{C}$ is definable.
\end{corollary}
\begin{proof}
On the one hand, the full subfibration $s\colon S\twoheadrightarrow\mathcal{C}$ of
$t\colon\mathrm{Fun}(\Delta^1,\mathcal{C})\twoheadrightarrow\mathcal{C}$ is definable if and only if the full inclusion
\begin{align}\label{equproplocaldefty}
S^{\times}\hookrightarrow\mathrm{Fun}(\Delta^1,\mathcal{C})^{\times}
\end{align}
fibered over $\mathcal{C}$ has a (non-fibered) right adjoint by Lemma~\ref{lemmadefreduct}. On the other 
hand, the class $S$ is local if and only if the inclusion $S^{\times}\hookrightarrow\mathrm{Fun}(\Delta^1,\mathcal{C})$ creates colimits via 
\cite[Lemma 6.1.3.7]{luriehtt}. As the class of all morphisms in $\mathcal{C}$ is local by assumption, this holds if and only if the full 
inclusion (\ref{equproplocaldefty}) creates colimits. Now, if $s\subseteq t$ is definable, then the full inclusion (\ref{equproplocaldefty}) 
is a left adjoint and 
hence creates colimits. In particular, definability of $s\subseteq t$ implies locality of $S$. Vice versa, suppose $S$ is local. To 
show that the inclusion (\ref{equproplocaldefty}) has a right adjoint, we are to show for any given
$f\in\mathrm{Fun}(\Delta^1,\mathcal{C})^{\times}$ that the comma $\infty$-category $S^{\times}\downarrow f$ has a terminal object. The 
morphism $f$ however is relative $\kappa$-compact for all regular cardinals $\kappa$ large enough. If we denote by $S_{\kappa}$ the 
intersection of $S$ and $\mathrm{Fun}(\Delta^1,\mathcal{C})_{\kappa}$, we observe that
$S_{\kappa}^{\times}\downarrow f = S^{\times}\downarrow f$ for any such $\kappa$. Hence, the latter has a terminal object if the inclusion
\[\iota_{\kappa}\colon S_{\kappa}^{\times}\hookrightarrow\mathrm{Fun}(\Delta^1,\mathcal{C})_{\kappa}^{\times}\]
has a right adjoint. If we denote by $s_{\kappa}\subseteq s$ the respective full subfibration, this holds whenever 
$\mathcal{C}$ has $(s_{\kappa}\subseteq t_{\kappa})$-comprehension. The class $S_{\kappa}$ is local again, and the continuous functor
$(S_{\kappa})_{/(\cdot)}\colon\mathcal{C}^{op}\rightarrow\mathrm{CAT}_{\infty}$ factors through $\mathrm{Cat}_{\infty}$.
It follows that the $\mathcal{C}$-indexed $\infty$-category $(S_{\kappa})_{/(\cdot)}$ is small by Remark~\ref{remproprightadj}. 
Thus, the inclusion $s_{\kappa}\subseteq t_{\kappa}$ is a cartesian functor between small fibrations, and so $\mathcal{C}$ has
$(s_{\kappa}\subseteq t_{\kappa})$-comprehension by Lemma~\ref{propdeffibs2}. 

Lastly, any given $\infty$-subtopos (or any modality) of $\mathcal{C}$ can be presented as a full subfibration of $t$ over $\mathcal{C}$, see 
e.g.\ \cite[Theorem A.7]{rss_hottmod}. The fact that the domain of this fibration is a local class of morphisms in $\mathcal{C}$ is shown for 
instance in \cite[Proposition 3.2.7]{abfjsheavesI}.
\end{proof}

For general $\infty$-categories $\mathcal{C}$ with pullbacks, a pullback-stable class $S$ of morphisms in $\mathcal{C}$  is ``closed'' 
\cite[Definition 3.4]{rasekheltops} if for all $C\in\mathcal{C}$ the fiber $s(C)$ of its associated full comprehension $\infty$-category 
$s\subseteq t$ is closed under all finite limits and colimits. We obtain a characterization of elementary $\infty$-toposes entirely in terms 
of comprehension schemes as follows.

\begin{corollary}\label{chareltopssmall}
An $\infty$-category $\mathcal{C}$ has finite limits and colimits if and only if
\begin{enumerate}
\item the identity over $\mathcal{C}$ and the identity over $\mathcal{C}^{op}$ are both representable right fibrations,
\item and all representable right fibrations over $\mathcal{C}$ and all representable right fibrations over $\mathcal{C}^{op}$ are small.
\end{enumerate}
Furthermore, an $\infty$-category $\mathcal{C}$ with finite limits and colimits is an elementary higher 
topos in the sense of \cite[Definition 3.5]{rasekheltops} if and only if the following two conditions hold.
\begin{enumerate}
\item[3.] The canonical fibration $t\colon\mathrm{Fun}(\Delta^1,\mathcal{C})\twoheadrightarrow\mathcal{C}$ has a subterminal object 
classifier (Example~\ref{explesubtermobclass}).
\item[4.] There is a (generally proper class-sized) collection $\mathcal{A}=\{S_i\mid i\in I\}$ of closed subclasses
$S_i\subset\mathrm{Fun}(\Delta^1,\mathcal{C})$ such that $\bigcup_{i\in I}S_i=\mathrm{Fun}(\Delta^1,\mathcal{C})$ and such that the 
associated full comprehension $\infty$-categories $s_i\colon S_i\twoheadrightarrow\mathcal{C}$ are small for every $i\in I$.
\end{enumerate}
\end{corollary}
\begin{proof}
The identity over $\mathcal{C}$ is the unstraightening of the terminal presheaf. It is representable if and only if $\mathcal{C}$ has a 
terminal object. Given a terminal object in $\mathcal{C}$, we have seen in Example~\ref{remlocsmallrep} that smallness of all representables 
over $\mathcal{C}$ is equivalent to the existence of all finite limits in $\mathcal{C}$. The same applies to $\mathcal{C}^{op}$. 
Given Proposition~\ref{propunivalentcomprehensioncats} and the fact that the slice of an elementary $\infty$-topos is again an elementary
$\infty$-topos \cite[Theorem 3.10]{rasekheltops}, the second part of the corollary is a mere reformulation of the axioms of an 
elementary $\infty$-topos.
\end{proof}

The first part of Corollary~\ref{chareltopssmall} can be formulated entirely in terms fibered over $\mathcal{C}$ by 
characterizing (representable) right fibrations over $\mathcal{C}^{op}$ as (corepresentable) left fibrations over $\mathcal{C}$. That means, 
$\mathcal{C}$ has finite colimits if and only if the identity on $\mathcal{C}$ is a corepresentable left fibration and all corepresentable 
left fibrations over $\mathcal{C}$ are small in the dual sense that they are equivalent to the nerve of an interval object in $\mathcal{C}$ 
in the sense of To\"en \cite{toeninftycats}.

The second part of the corollary can be interpreted to state that while a finitely bi-complete (and large) $\infty$-category $\mathcal{C}$ 
with a well-powered canonical indexing generally cannot be globally small over itself, it instead can be covered by an atlas $\mathcal{A}$ of 
small neighbourhoods in $\mathcal{C}$ exactly if it is an elementary $\infty$-topos. Local smallness of an elementary $\infty$-topos 
over itself then follows from the definition via \cite[Theorem 3.11]{rasekheltops} and Proposition~\ref{exmplelocsmall}. Vice versa, if on 
top of well-poweredness also local smallness of $\mathcal{C}$ over itself is assumed explicitly, then the existence of a cover of small 
neighbourhoods is equivalent to a cover of globally small neighbourhoods by Proposition~\ref{propunivalentcomprehensioncats}. 
Whenever the $\infty$-category $\mathcal{C}$ is presentable as well, 
Condition 2 in Corollary~\ref{chartopssmall} is exactly Condition 4 in Corollary~\ref{chareltopssmall} for the atlas
$\mathcal{A}=\{S_{\kappa}\mid \kappa\in\mathrm{Card}\text{ sufficiently large}\}$. 

\section{Model independence}\label{secmodindep}

Lastly, in this section we show that our constructions and main results regarding comprehension are invariant under the choice of any of the 
common models of $(\infty,1)$-category theory that the constructions and results are formulated in. As announced in the beginning of 
Section~\ref{secsubcartfib}, we therefore use Riehl and Verity's framework of $\infty$-cosmoses \cite{riehlverityelements}. Here, an
$\infty$-cosmos $\mathbf{V}$ is said to be an \emph{$\infty$-cosmos of $(\infty,1)$-categories} whenever the global sections functor
$(\cdot)_0\colon\mathbf{V}\rightarrow\mathbf{QCat}$ into the $\infty$-cosmos of quasi-categories is a cosmological biequivalence
\cite[Definition 1.3.10, Proposition 10.2.1]{riehlverityelements}. Every known model of $(\infty,1)$-category theory exhibits 
the structure of an $\infty$-cosmos together with such a cosmological biequivalence. In this section, we will stay 
faithful to the terminology used in the book \cite{riehlverityelements} for referential purposes. \\

In the following we will define comprehension schemes for fibrations and functors in general $\infty$-cosmoses $\mathbf{V}$ (with additional 
structure) which generalize the notions introduced in Section~\ref{seccomp} as well as in the introduction to this framework. Technically, we 
will prove that, first, all (suitable) cosmological functors between (suitable) $\infty$-cosmoses preserve satisfaction of all 
such comprehension schemes\footnote{The condition of being ``suitable'' will only be necessary for the case of basic diagrammatic 
comprehension. It means that the $\infty$-cosmos $\mathbf{V}$ has a fibered core construction and that the cosmological functor preserves 
fibered core constructions, see Definition~\ref{defGcompcosmos}.}, and second, that cosmological biequivalences both preserve and reflect 
satisfaction of all such comprehension schemes. We will state our results for the cosmological global sections functors
$(\cdot)_0\colon\mathbf{V}\rightarrow\mathbf{QCat}$ only however to not divert too far from the subject at hand.
It follows that the notions of comprehension in this paper are invariant under the choice of model of $(\infty,1)$-category theory. It  
further iteratively follows that all results and constructions of this paper which are otherwise comprised of formal $\infty$-cosmological 
structures only are model independent.

\begin{definition}
Let $\mathbf{V}$ be an $\infty$-cosmos and $p\colon E\twoheadrightarrow C$, $q\colon F\twoheadrightarrow C$ be a pair of discrete fibrations 
in $\mathbf{V}$. Say that a functor $f\colon p\rightarrow q$ in $\mathbf{V}_{/C}$ is \emph{representable} in $\mathbf{V}$ if
$f\colon E\rightarrow F$ as a functor in $\mathbf{V}$ has a right adjoint.
\end{definition}

\begin{proposition}[Invariance of representability]
Let $\mathbf{V}$ be an $\infty$-cosmos of $(\infty,1)$-categories and $C\in\mathbf{V}$ and object. Then the cosmological biequivalence
$(\cdot)_0\colon\mathbf{V}\rightarrow\mathbf{QCat}$ induces a cosmological biequivalence
\begin{align}\label{equequivdiscfib}
\mathbf{DiscCart}(\mathbf{V})_{/C}\rightarrow\mathbf{DiscCart}(\mathbf{QCat})_{/C_0}
\end{align}
between the corresponding $\infty$-cosmoses of discrete fibrations \cite[Proposition 6.3.15]{riehlverityelements}. In particular,
\begin{enumerate}
\item we obtain a 1-1 correspondence between homotopy classes of discrete fibrations in $\mathbf{V}$ over $C$, and presheaves over the
quasi-category $C_0$.
\item given two discrete fibrations $p\colon E\twoheadrightarrow C$ and $q\colon F\twoheadrightarrow C$ in $\mathbf{V}$ over $C$,
we obtain a 1-1 correspondence between homotopy classes of functors between $p$ and $q$ in $\mathbf{V}_{/C}$, and 
natural transformations between the presheaves $\mathrm{St}(p_0)$ and $\mathrm{St}(q_0)$ over $C_0$.
\end{enumerate}  
Furthermore, a functor $f\colon p\rightarrow q$ of right fibrations in $\mathbf{V}_{/C}$ is representable in $\mathbf{V}$ if and only if the 
induced natural transformation $\mathrm{St}(f_0)\colon\mathrm{St}(p_0)\rightarrow\mathrm{St}(q_0)$ of presheaves over $C_0$ is representable.
\end{proposition}
\begin{proof}
The existence of the induced cosmological biequivalence between the two respective $\infty$-cosmoi of discrete fibrations is proven in
\cite[Corollary 10.3.7]{riehlverityelements}. The $\infty$-cosmos of discrete fibrations in $\mathbf{QCat}$ over $C_0$ is exactly the 
underlying $\infty$-cosmos of the model category for right fibrations over $C_0$ via \cite[Proposition F.4.9]{riehlverityelements}. The 
cosmological biequivalence (\ref{equequivdiscfib}) hence induces a composite equivalence of underlying $\infty$-categories
\[\mathrm{DiscCart}(\mathbf{V})_{/C}\simeq\mathrm{RFib}(C_0)\simeq \widehat{C_0}\]
via the straightening/unstraightening construction for quasi-categories from Section~\ref{secpre}. Parts 1 and 2 are hence immediate. The 
fact that a functor $f$ of discrete fibrations in $\mathbf{V}_{/C}$ is representable if and only if $f_0$ in $\mathbf{QCat}_{/C_0}$ is so 
follows from \cite[Proposition 10.3.6]{riehlverityelements}. Equivalence of this notion of representability of $f_0$ to representability of 
its associated natural transformation of presheaves in turn follows from Proposition~\ref{proprepcartmaps}.
\end{proof}

\begin{corollary}
A general comprehension scheme over an $\infty$-category $C$ in form of a representable functor of discrete fibrations over $C$ is a model 
independent notion.\qed
\end{corollary}

Let us move on to the more specific diagrammatic comprehension schemes of Section~\ref{seccomp}.
A general $\infty$-cosmos $\mathbf{V}$ does not necessarily come equipped with a $(\cdot)^{op}$-construction. As this however is the 
only obstruction to define $(K,G)$-comprehension of $C$-indexed objects in $K$ internal to $\mathbf{V}$, we instead define
$(K^{op},G)$-comprehension formally as follows.

\begin{definition}
Let $\mathbf{V}$ be an $\infty$-cosmos and $E\colon C\rightarrow K$ be a functor in $\mathbf{V}$. Every morphism
$\{G\}\colon\ast\rightarrow K^{\Delta^1}$, depicted as $G\colon l\rightarrow k$ in $K$, induces a functor
$G_{\ast}\colon K_{/l}\rightarrow K_{/k}$ of representable discrete fibrations over $K$ in $\mathbf{V}$. 
We say that $E$ has \emph{$(K^{op},G^{op})$-comprehension} if the functor $E^{\ast}(G_{\ast})\colon E^{\ast}K_{/l}\rightarrow E^{\ast}K_{/k}$ 
of discrete cartesian fibrations over $C$ has a right adjoint in $\mathbf{V}$.
\end{definition}

An $\infty$-cosmos of $(\infty,1)$-categories more specifically however does automatically come equipped with a suitable
$(\cdot)^{op}$-construction \cite[Definition 12.1.4]{riehlverityelements} for its objects. In this case we have the following  
characterization.

\begin{proposition}[Invariance of $(K,G)$-comprehension]
Let $\mathbf{V}$ be an $\infty$-cosmos of $(\infty,1)$-categories, $E\colon C\rightarrow K$ be a functor in $\mathbf{V}$, and
$G\colon l\rightarrow k$ be a morphism in $K$. Then $E$ has $(K^{op},G^{op})$-comprehension if and only if the associated functor
$E_0^{op}\colon C_0^{op}\rightarrow K_0^{op}$ of quasi-categories has $(K_0^{op},G^{op})$-comprehension in the sense of 
Definition~\ref{defGcomp}.

In particular, a functor $E\colon C\rightarrow K^{op}$ in $\mathbf{V}$ has $((K^{op})^{op},(G^{op})^{op})$-comprehension if and only if the 
functor $E_0^{op}\colon C_0^{op}\rightarrow K_0$ of quasi-categories has $(K,G)$-comprehension.
\end{proposition}

\begin{proof}
Let $E\colon C\rightarrow K$ and $G\colon l\rightarrow k$ be as above. Then, again by \cite[Proposition 10.3.6]{riehlverityelements}, the 
functor $E^{\ast}(G_{\ast})\colon E^{\ast}K_{/l}\rightarrow E^{\ast}K_{/k}$ of discrete cartesian fibrations over $C$ has a right adjoint in
$\mathbf{V}$ if and only if the functor
\begin{align}\label{equpropinvKGcomp}
E_0^{\ast}(G_{\ast})\colon E_0^{\ast}((K_0)_{/l})\rightarrow E_0^{\ast}((K_0)_{/k})
\end{align}
of right fibrations over $C_0$ has a right adjoint in $\mathbf{QCat}$. 
The latter however is the unstraightening of the natural transformation
\begin{align}\label{equpropinvKGcomp1}
(E_0^{op})^{\ast}(G_{\ast})\colon K_0(E_0^{op},l)\rightarrow K_0(E_0^{op},k)
\end{align}
in $\widehat{C_0}$ essentially by Example~\ref{explscartfibs}.3. Thus, the functor (\ref{equpropinvKGcomp}) has a right adjoint if and only 
if the natural transformation (\ref{equpropinvKGcomp1}) is representable by Proposition~\ref{proprepcartmaps}. This natural transformation is 
canonically equivalent to
\[(E_0^{op})^{\ast}((G^{op})^{\ast})\colon K_0^{op}(l,E_0^{op})\rightarrow K_0^{op}(k,E_0^{op}),\]
which is representable if and only if $E_0^{op}$ has $(K_0^{op},G^{op})$-comprehension by definition.
The second part is a trivial consequence of idempotency of taking opposites.
\end{proof}

Lastly, we consider the special case of $G$-comprehension for a map $G\colon I\rightarrow J$ between simplicial sets. To define it, we 
require a fibered core construction over every base of our $\infty$-cosmos $\mathbf{V}$ so to be able to construct the discrete fibrations of 
pinched rectangular diagrams. Thus, we say that an $\infty$-cosmos $\mathbf{V}$ has a fibered core construction whenever for every 
$C\in\mathbf{V}$, every object $p$ in the $\infty$-cosmos $\mathbf{Cart}(\mathbf{V})$ has an $\infty$-groupoid core
$\iota\colon p^{\times}\rightarrow p$ in the sense of \cite[Definition 12.1.14]{riehlverityelements}. That means each $p^{\times}$ is a 
discrete object in $\mathbf{Cart}(\mathbf{V})$ such that for any other discrete object $q\in\mathbf{Cart}(\mathbf{V})$ the pushforward
\begin{align}\label{equdefcorecosmos}
\iota_{\ast}\colon\mathbf{Cart}(\mathbf{V})(q,p^{\times})\rightarrow\mathbf{Cart}(\mathbf{V})(q,p)
\end{align}
is an equivalence of quasi-categories. Here, note that the discrete objects in $\mathbf{Cart}(\mathbf{V})$ are exactly the discrete cartesian 
fibrations by definition \cite[Definition 5.5.3]{riehlverityelements}.

\begin{definition}\label{defGcompcosmos}
Let $\mathbf{V}$ be an $\infty$-cosmos with a fibered core construction. Let $p\colon E\twoheadrightarrow C$ be a cartesian fibration in
$\mathbf{V}$ and $G\colon I\rightarrow J$ be a map of simplicial sets. Say that $p$ has \emph{$G$-comprehension} if the functor of
$\infty$-categories underlying the (essentially unique) restriction $(p^G)^{\times}\colon (p^J)^{\times}\rightarrow(p^I)^{\times}$ obtained 
from (\ref{equdefcorecosmos}) has a right adjoint in $\mathbf{V}$.
\end{definition}

Unsurprisingly, for instance, every $\infty$-cosmos $\mathbf{V}$ of $(\infty,1)$-categories has a fibered core construction, as can be seen 
simply by reflecting the core construction from Definition~\ref{defcore} along the induced cosmological biequivalence
$\mathbf{Cart}(\mathbf{V})\xrightarrow{\sim}\mathbf{Cart}(\mathbf{QCat})$.

\begin{proposition}[Invariance of $G$-comprehension]
Let $\mathbf{V}$ be an $\infty$-cosmos of $(\infty,1)$-categories, $p\colon E\twoheadrightarrow C$ be a cartesian fibration in
$\mathbf{V}$ and $G\colon I\rightarrow J$ be a map of simplicial sets. Then $p$ has $G$-comprehension if and only if the associated
$C_0$-indexed quasi-category $\mathrm{St}(p_0)$ has $G$-comprehension in the sense of Notation~\ref{notationGcompstrict}.
\end{proposition}
\begin{proof}
The induced cosmological functor $\mathbf{Cart}(\mathbf{V})\xrightarrow{\sim}\mathbf{Cart}(\mathbf{QCat})$ preserves both simplicial 
cotensors and the fibered core construction. The cosmological biequivalence $(\cdot)_0\colon\mathbf{V}\rightarrow\mathbf{QCat}$ preserves and 
reflects right adjoints, and so $p$ has $G$-comprehension in $\mathbf{V}$ if and only if the functor
$(p_0^G)^{\times}\colon\llbracket J,E_0\rrbracket\rightarrow\llbracket I,E_0\rrbracket$ of quasi-categories has a right adjoint. This in turn 
holds if and only if the presheaf $\mathrm{St}(p_0)$ has $G$-comprehension by Proposition~\ref{propstrrect}.
\end{proof}

\bibliographystyle{amsplain}
\bibliography{BSBib}

\end{document}